\documentclass[12pt]{amsart}
\usepackage{amssymb,amsfonts,amsmath}

\newcommand{\g}{\mathfrak{g}}
\newcommand{\Spec}{\operatorname{Spec}}
\newcommand{\Z}{\mathbb{Z}}
\newcommand{\A}{\mathcal{A}}
\newcommand{\gr}{\operatorname{gr}}
\newcommand{\Dcal}{\mathcal{D}}

\newcommand{\h}{\mathfrak{h}}

\newcommand\M{\mathcal{M}}
\newcommand\Loc{\operatorname{Loc}}
\newcommand\Ext{\operatorname{Ext}}

\newcommand\GL{\operatorname{GL}}
\newcommand\gl{\mathfrak{gl}}

\newcommand\red{/\!\!/\!\!/}
\newcommand\I{\mathcal{I}}

\newcommand\J{\mathcal{J}}

\newcommand\End{\operatorname{End}}

\newcommand\Hom{\operatorname{Hom}}

\renewcommand\a{\mathfrak{a}}

\newcommand\quo{/\!/}

\newcommand\W{\mathbb{A}}

\newcommand\C{\mathbb{C}}

\newcommand\param{\mathfrak{P}}

\newcommand\HC{\operatorname{HC}}

\newcommand\ZZ{\mathbb{Z}}
\newcommand{\CC}{\mathsf{CC}}

\newcommand{\WC}{\mathfrak{WC}}

\newcommand{\VA}{\operatorname{V}}
\newcommand{\B}{\mathcal{B}}
\newcommand{\Weyl}{\mathbf{A}}
\newcommand{\Cat}{\mathcal{C}}
\newcommand{\OCat}{\mathcal{O}}

\newcommand{\Leaf}{\mathcal{L}}

\newcommand{\Af}{\mathfrak{A}}
\newcommand{\Ca}{\mathsf{C}}
\newtheorem{Thm}{Theorem}[section]
\newtheorem{Prop}[Thm]{Proposition}
\newtheorem{Cor}[Thm]{Corollary}
\newtheorem{Lem}[Thm]{Lemma}
\theoremstyle{definition}

\newtheorem{Rem}[Thm]{Remark}

\numberwithin{equation}{section}
\author{Ivan Losev}
\address{I.L.: Department
of Mathematics, Northeastern University, Boston MA 02115 USA}
\email{i.loseu@neu.edu}
\title{Etingof conjecture for quantized quiver varieties II: affine quivers}
\begin{document}
\begin{abstract}
We study the representation theory of quantizations of Gieseker moduli spaces. Namely, we prove the localization
theorems for these algebras, describe their finite dimensional representations and two-sided ideals as well as
their categories $\mathcal{O}$ in some special cases. We apply this to prove our conjecture with Bezrukavnikov
on the number of finite dimensional irreducible representations of quantized quiver varieties for quivers of
affine type.
\end{abstract}
\maketitle
\tableofcontents
\section{Introduction}\label{S_intro}
\subsection{Classical and quantum quiver varieties}\label{SS_quiver_intro}
This paper continues the study of the representation theory of quantized quiver varieties initiated in \cite{BL}.
So  we start by recalling Nakajima quiver varieties and their quantizations.

Let $Q$ be a quiver (=oriented graph, we allow loops and multiple edges). We can formally represent $Q$ as a quadruple
$(Q_0,Q_1,t,h)$, where $Q_0$ is a finite set of vertices, $Q_1$ is a finite set of arrows,
$t,h:Q_1\rightarrow Q_0$ are maps that to an arrow $a$ assign its tail and head. In this paper we
are interested in the case when $Q$ is of affine type, i.e., $Q$ is an extended quiver of type $A,D,E$.

Pick vectors $v,w\in \ZZ_{\geqslant 0}^{Q_0}$ and vector spaces $V_i,W_i$ with
$\dim V_i=v_i, \dim W_i=w_i$. Consider the (co)framed representation space
$$R=R(v,w):=\bigoplus_{a\in Q_1}\Hom(V_{t(a)},V_{h(a)})\oplus \bigoplus_{i\in Q_0} \Hom(V_i,W_i).$$
We will also consider its cotangent bundle $T^*R=R\oplus R^*$, this is a symplectic vector space that can be identified with
$$\bigoplus_{a\in Q_1}\left(\Hom(V_{t(a)},V_{h(a)})\oplus \Hom(V_{h(a)}, V_{t(a)})\right)\oplus \bigoplus_{i\in Q_0} \left(\Hom(V_i,W_i)\oplus \Hom(W_i,V_i)\right).$$
The group $G:=\prod_{k\in Q_0}\GL(V_k)$ naturally acts on $T^*R$ and this action is Hamiltonian. Its moment map $\mu:T^*R\rightarrow
\g^*$ is dual to
$x\mapsto x_R:\g\rightarrow \C[T^*R]$, where $x_R$ stands for the vector field on $R$ induced by $x\in \g$.

Fix a stability condition $\theta\in \Z^{Q_0}$ that is thought as a character of $G$ via $\theta((g_k)_{k\in Q_0})=\prod_{k\in Q_0}\det(g_k)^{\theta_k}$. Then, by definition, the quiver variety $\M^\theta(v,w)$ is the GIT Hamiltonian reduction
$\mu^{-1}(0)^{\theta-ss}\quo G$. We are interested in two extreme cases: when $\theta$ is generic (and so $\M^\theta(v,w)$
is smooth and symplectic) and when $\theta=0$ (and so $\M^\theta(v,w)$ is affine). We will write $\M(v,w)$ for
$\Spec(\C[\M^\theta(v,w)])$, this is an affine variety independent of $\theta$ and a natural projective morphism
$\rho:\M^\theta(v,w)\rightarrow \M(v,w)$ is a resolution of singularities.

Under an additional restriction on $v$, we have the equality $\M(v,w)=\M^0(v,w)$. Namely, let $\g(Q)$
be the affine Kac-Moody algebra associated to $Q$. Let us set $\omega:=\sum_{i\in Q_0}w_i\omega^i, \nu:=\omega-\sum_{i\in Q_0}v_i\alpha^i$, where we write $\omega^i$ for the fundamental weight and $\alpha^i$ for a simple root corresponding to $i\in Q_0$.
Then we have $\M^0(v,w)=\M(v,w)$ provided $\nu$ is dominant.

Note also that we have compatible $\C^\times$-actions on $\M^\theta(v,w),\M(v,w)$ induced from the action on $T^*R$
given by $t.(r,\alpha):=(t^{-1}r,t^{-1}\alpha), r\in R, \alpha\in R^*$.

A special case of most interest and importance for us in this paper is the Gieseker moduli spaces $\M^\theta(n,r)$
where $n,r\in \Z_{>0}$. It corresponds to the case when $Q$ is a quiver with a single vertex and a single
arrow (that is a loop), with $v=n, w=r$. This space parameterizes torsion free sheaves of rank $r$ and degree $n$
on $\mathbb{P}^2$ trivialized at the line at infinity (but we will not need this description). The importance
of this case in our work is of the same nature as in the work of Maulik and Okounkov, \cite{MO}, on computing the
quantum cohomology of quiver varieties.

Now let us proceed to the quantum setting. We will work with quantizations of $\M^\theta(v,w),\M(v,w)$. Consider
the algebra $D(R)$ of differential operators on $R$. The group $G$ naturally acts on $D(R)$ with a quantum
comoment map $\Phi:\g\rightarrow D(R), x\mapsto x_R$. We can consider the quantum Hamiltonian reduction
$\A^0_\lambda(v,w)=[D(R)/D(R)\{x_R-\langle \lambda,x\rangle\}]^G$. It is a quantization of $\M^0(v,w)=\M(v,w)$
when $\nu$ is dominant. In the general case one can define a quantization of $\M(v,w)$ in two equivalent
way: as an algebra $\A^0_{\lambda'}(v',w)$ for suitable $\lambda'$ and $v'$ (thanks to quantized LMN isomorphisms
from \cite[2.2]{BL}) or as the algebra of global section of a suitable microlocal sheaf on $\M^\theta(v,w)$
(where $\theta$ is generic). Let us recall the second approach. We can microlocalize $D(R)$ to a sheaf in conical topology (i.e., the topology where ``open'' means ``Zariski open'' and $\C^\times$-stable) so that we can consider the restriction
of $D(R)$ to the $(T^*R)^{\theta-ss}$, let $\mathcal{D}^{\theta-ss}$ denote the restriction. Let $\pi$ stand
for the quotient morphism $\mu^{-1}(0)^{\theta-ss}\rightarrow\mu^{-1}(0)^{\theta-ss}/G=\M^\theta(v,w)$.
Let us notice that $\mathcal{D}^{\theta-ss}/\mathcal{D}^{\theta-ss}\{x_R-\langle \lambda,x\rangle\}$
is scheme-theoretically supported on $\mu^{-1}(0)^{\theta-ss}$ and so can be regarded as a sheaf in
conical topology on that variety. Set
$$\A^\theta_{\lambda}(v,w):=[\pi_*(\mathcal{D}^{\theta-ss}/\mathcal{D}^{\theta-ss}\{x_R-\langle \lambda,x\rangle\})]^G,$$
this is a sheaf (in  conical topology) of filtered algebras on $\M^\theta(v,w)$ such that $\gr\A^\theta_{\lambda}(v,w)=\mathcal{O}_{\M^\theta(v,w)}$.
By the Grauert-Riemenschneider theorem, $H^i(\mathcal{O}_{\M^\theta(v,w)})=0$ for $i>0$. It follows that $\A^\theta_\lambda(v,w)$ has no higher cohomology as well,
and $\gr \Gamma(\A^\theta_\lambda(v,w))=\C[\M(v,w)]$.  One can show, see \cite[3.3]{BPW} or \cite[2.2]{BL}, that $\Gamma(\A^\theta_\lambda(v,w))$
is independent of the choice of $\theta$. We will write $\A_\lambda(v,w)$ for $\Gamma(\A^\theta_\lambda(v,w))$.

In this paper we will be interested in the representation theory of the algebras $\A_\lambda(v,w)$ and, especially, of $\A_\lambda(n,r)$ (quantizations of $\M(n,r)$). Let us point out that the representations of $\A_\lambda^\theta(v,w)$ and of $\A_\lambda(v,w)$
are closely related. Namely, we can consider the category of coherent $\A_\lambda^\theta(v,w)$-modules to be denoted by
$\A_\lambda^\theta(v,w)\operatorname{-mod}$ and the category $\A_\lambda(v,w)\operatorname{-mod}$ of all finitely generated
$\A_\lambda(v,w)$-modules. When the homological dimension of $\A_\lambda(v,w)$ is finite, we get adjoint functors
$$R\Gamma_\lambda^\theta: D^b(\A_\lambda^\theta(v,w)\operatorname{-mod})\rightleftarrows D^b(\A_\lambda(v,w)\operatorname{-mod}):L\Loc_\lambda^\theta,$$ where $R\Gamma_\lambda^\theta$ is the derived global section
functor, and $L\operatorname{Loc}_\lambda^\theta:=\A_\lambda^\theta(v,w)\otimes^L_{\A_\lambda(v,w)}\bullet$.
It turns out that these functors are equivalences, \cite{MN_der}. In particular, they restrict to mutually inverse equivalences
\begin{equation}\label{eq:fin_dim_equi}
D^b_{\rho^{-1}(0)}(\A_\lambda^\theta(v,w)\operatorname{-mod})\rightleftarrows D^b_{fin}(\A_\lambda(v,w)\operatorname{-mod}),
\end{equation}
where on the left hand side we have the category of all complexes with homology supported on $\rho^{-1}(0)$, while on the right hand
side we have all complexes with finite dimensional homology.

\subsection{Results in the Gieseker cases}\label{SS_Gies_result}
In this paper we are mostly dealing with the algebras $\A_\lambda(n,r)$. Note that $R=\C\oplus \bar{R}$, where $\bar{R}:=\mathfrak{sl}_n(\C)\oplus \Hom(\C^n,\C^r)$ and the action of $G$ on $\C$ is trivial. So we have $\M^\theta(n,r)=\C^2\times \bar{\M}^\theta(n,r)$ and
$\A_\lambda(n,r)=D(\C)\otimes\bar{\A}_{\lambda}(n,r)$, where $\bar{\M}^\theta(n,r),\bar{\A}_\lambda(n,r)$ are the reductions associated
to the $G$-action on $\bar{R}$. We will consider the algebra $\bar{\A}_\lambda(n,r)$ rather than $\A_\lambda(n,r)$, all interesting
representation theoretic questions about $\A_\lambda(n,r)$ can be reduced to those about $\bar{\A}_\lambda(n,r)$.

There is one case that was studied very explicitly in the last decade: $r=1$. Here the variety $\M^\theta(n,1)$
is the Hilbert scheme $\operatorname{Hilb}^n(\C^2)$ of $n$ points on $\C^2$ and $\M(r,n)=\C^{2n}/\mathfrak{S}_n$
(the $n$th symmetric power of $\C^2$). The quantization $\bar{\A}_\lambda(n,r)$ is the spherical subalgebra in the
Rational Cherednik algebra $H_\lambda(n)$ for the pair $(\h,\mathfrak{S}_n)$, where $\h$ is the reflection representation
of $\mathfrak{S}_n$, see \cite{GG} for details.
The representation theory of $\bar{\A}_\lambda(n,1)$ was studied, for example, in \cite{BEG,GS1,GS2,rouqqsch,KR,BE,sraco,Wilcox}.
In particular, it is known
\begin{enumerate}
\item when (=for which $\lambda$) this algebra has finite homological dimension, \cite{BE},
\item how  to classify its  finite dimensional irreducible representations, \cite{BEG},
\item how to compute characters of irreducible
modules in the so called category $\mathcal{O}$, \cite{rouqqsch},
\item how to determine the supports of these modules, \cite{Wilcox},
\item how to describe the two-sided ideals of $\bar{\A}_\lambda(n,1)$, \cite{sraco},
\item when an analog of the Beilinson-Bernstein  localization theorem holds, \cite{GS1,KR}.
\end{enumerate}
We will address analogs of (1),(2),(5),(6)  for $\bar{\A}_\lambda(n,r)$ (as well as relatively easy parts of (3) and (4))
in the present paper. We plan to address an analog of (4) in a subsequent paper, while (3) is a work in progress.

Before we state  our main results, let us point out that there is yet another case when
the algebra $\bar{\A}_\lambda(n,r)$ is classical, namely when $n=1$. In this case, $\bar{\A}_\lambda(1,r)=D^\lambda(\mathbb{P}^{r-1})$,
the algebra of $\lambda$-twisted differential operators on $\mathbb{P}^{r-1}$.

First, let us give answers to (1) and (6).

\begin{Thm}\label{Thm:loc}
The following is true.
\begin{enumerate}
\item The algebra $\bar{\A}_\lambda(n,r)$ has finite global dimension (equivalently, $R\Gamma_\lambda^\theta$ is an equivalence) if and only if $\lambda$ is not of the form $\frac{s}{m}$,
where $1\leqslant m\leqslant n$ and $-rm<s<0$.
\item For $\theta>0$, the abelian localization holds for $\lambda$ (i.e., $\Gamma_\lambda^\theta$ is an equivalence) if $\lambda$
is not of the form $\frac{s}{m}$, where $1\leqslant m\leqslant n$ and $s<0$. For $\theta<0$, the abelian localization holds for
$\lambda$ if and only if $\lambda$ is not of the form $\frac{s}{m}$ with $1\leqslant m\leqslant n$ and $s> -rm$.
\end{enumerate}
\end{Thm}

In fact part (2) is a straightforward consequence of (1) and results of McGerty and Nevins, \cite{MN_ab}.

Let us proceed to classification of finite dimensional representations.

\begin{Thm}\label{Thm:fin dim}
The following holds.
\begin{enumerate}
\item The sheaf $\bar{\A}_\lambda^\theta(n,r)$ has a representation supported on $\bar{\rho}^{-1}(0)$ if and only if $\lambda=\frac{s}{n}$
with $s$ and $n$ coprime. If that is the case, then the category $\bar{\A}_\lambda^\theta(n,r)\operatorname{-mod}_{\rho^{-1}(0)}$
is equivalent to $\operatorname{Vect}$.
\item The algebra $\bar{\A}_\lambda(n,r)$ has a finite dimensional representation  if and only if $\lambda=\frac{s}{n}$
with $s$ and $n$ coprime and the homological dimension of $\bar{\A}_\lambda(n,r)$ is finite. If that is the case, then the category $\bar{\A}_\lambda^\theta(n,r)\operatorname{-mod}_{\rho^{-1}(0)}$ is equivalent to $\operatorname{Vect}$.
\end{enumerate}
\end{Thm}
In fact, (2) is an easy consequence of (1) and Theorem \ref{Thm:loc}.

Now let us proceed to the description of two-sided ideals (in the finite homological dimension case).

\begin{Thm}\label{Thm:ideals}
Assume that $\bar{\A}_\lambda(n,r)$ has finite homological dimension and let $m$ stand for the denominator
of $\lambda$ (equal to $+\infty$ if $\lambda$ is not rational). Then there are $\lfloor n/m\rfloor$
proper two-sided ideals in $\bar{\A}_\lambda(n,r)$, all of them are prime, and they form a chain.
\end{Thm}

Finally, let us explain some partial results on a category $\mathcal{O}$ for $\bar{\A}^\theta_\lambda(n,r)$, we will
recall necessary definitions below in Subsection \ref{SS_Cat_O}. We use the notation $\mathcal{O}(\A^{\theta}_{\lambda}(n,r))$
for this category. What we need to know now is the following:
\begin{itemize}
\item The category $\mathcal{O}(\A^{\theta}_{\lambda}(n,r))$ is a highest weight category so it makes sense to speak about
standard objects $\Delta(p)$.
\item The labeling set for standard objects is naturally identified with the set of $r$-multipartitions of $n$.
\end{itemize}

\begin{Thm}\label{Thm:cat_O_easy}
If the denominator of $\lambda$ is bigger than $n$, then the category $\mathcal{O}(\A^\theta_\lambda(n,r))$ is semisimple.
If the denominator of $\lambda$ equals $n$, the category $\mathcal{O}(\A^\theta_\lambda(n,r))$ has only one nontrivial
block. That block is equivalent to the nontrivial block of $\mathcal{O}(\A^\theta_{1/nr}(nr,1))$.
\end{Thm}
In some cases, we can say which simple objects belong to the nontrivial block, we will do this below.


\subsection{Counting result}
We are going to describe $K_0(\A^\theta_\lambda(v,w)\operatorname{-mod}_{\rho^{-1}(0)})$ (we always consider complexified
$K_0$) in the case when $Q$ is of affine type, confirming \cite[Conjecture 1.1]{BL} in this case. The dimension of this $K_0$ coincides with the number of finite dimensional irreducible representations of $\A_\lambda(v,w)$ provided $\lambda$
is regular, i.e., the homological dimension of $\A_\lambda(v,w)$ is finite.

Recall that, by \cite{Nakajima}, the homology group $H_{mid}(\M^\theta(v,w))$ (where ``mid'' stands for $\dim_\C \M^\theta(v,w)$)
is identified with the weight space $L_\omega[\nu]$ of weight $\nu$ (see Subsection \ref{SS_quiver_intro}) in  the irreducible integrable $\g(Q)$-module
$L_\omega$ with highest weight $\omega$. Further, by \cite{BarGin}, we have a natural inclusion $K_0(\A_\lambda(v,w)\operatorname{-mod}_{\rho^{-1}(0)})\hookrightarrow H_{mid}(\M^\theta(v,w))$ given by the characteristic cycle
map $\CC_\lambda$. We will elaborate on this below in Subsection \ref{SS_loc}. We want to describe the image of $\CC_\lambda$.

Following \cite{BL}, we define a subalgebra $\a(=\a_\lambda)\subset\g(Q)$ and an $\a$-submodule $L_\omega^{\a}\subset L_\omega$.
By definition, $\a$ is spanned by the Cartan subalgebra $\mathfrak{t}\subset \g(Q)$ and all root spaces $\g_\beta(Q)$
where $\beta=\sum_{i\in Q_0}b_i\alpha^i$ is a real root with $\sum_{i\in Q_0}b_i\lambda_i\in \Z$. For $L_\omega^{\a}$ we take
the $\a$-submodule of $L_\omega$ generated by the extremal weight spaces (those where the weight is conjugate to the highest one
under the action of the Weyl group).

\begin{Thm}\label{Thm:counting}
Let $Q$ be of affine type.  The image of $K_0(\A_\lambda(v,w)\operatorname{-mod}_{\rho^{-1}(0)})$ in $L_\omega[\nu]$
under $\CC_\lambda$ coincides with $L_\omega^\a\cap L_\omega[\nu]$.
\end{Thm}

\subsection{Content of the paper}
Section \ref{S_prelim} contains some known results and construction. In Section \ref{S_parab_ind} we introduce
our main tool for inductive study of categories $\mathcal{O}$. In Section \ref{S_fin_dim} we will prove
Theorems \ref{Thm:fin dim} (most of it, in fact) \ref{Thm:cat_O_easy}. In Section \ref{S_loc}
we prove Theorem \ref{Thm:loc}. Finally, in Section \ref{S_aff_wc_count} we prove Theorem
\ref{Thm:ideals} and also complete the proof of Theorem \ref{Thm:counting}.
In the beginning of each section, its content is described in more detail.

{\bf Acknowledgments}. I would like to thank Roman Bezrukavnikov, Dmitry Korb, Davesh Maulik, Andrei Okounkov and Nick
Proudfoot for stimulating discussions. My work was supported by the NSF  under Grant  DMS-1161584.
\section{Preliminaries}\label{S_prelim}
This section basically contains no new results. We start with discussing conical symplectic resolutions.
Then, in Subsection \ref{SS_Gies}, we list some further properties of Gieseker moduli spaces.
Subsection \ref{SS_leaves} describes the symplectic leaves of the varieties $\M(v,w)$.

After that, we proceed to quantizations. We discuss some further properties, with emphasis on  the Gieseker case,
in Subsection \ref{SS_quant}.  We discuss (derived and abelian) localization theorems for quantized
quiver varieties, Subsection \ref{SS_loc}.   Then we proceed to the homological duality
and wall-crossing functors, one of our main tools to study the representation theory of quantized
quiver varieties, Subsection \ref{SS_dual_WC}. In Subsection \ref{SS_Cat_O}, we recall the definition
of categories $\mathcal{O}$ and list some basic properties. Then, in Subsection \ref{SS_HC_bimod}, we recall one more important
object in this representation theory, Harish-Chandra bimodules. Our main tool to study those is restriction (to so called
{\it quantum slices}) functors, defined in this context in \cite{BL}. We recall quantum slices in Subsection
\ref{SS_quant_slice} and the restriction functors  in Subsection \ref{SS_restr_fun}.

\subsection{Symplectic resolutions}\label{SS_sympl_res}
Although in this paper we are primarily interested in the case of Nakajima quiver varieties
for quiver of affine types (and, more specifically, Gieseker moduli spaces) some of our results easily generalize
to symplectic resolutions of singularities. Here we recall the definition and  describe some structural theory of
these varieties due to Namikawa, \cite{Namikawa}. Our exposition follows \cite[Section 2]{BPW}.

Let $X$ be a smooth symplectic algebraic variety.  By definition, $X$ is called
a symplectic resolution of singularities if $\C[X]$ is finitely generated and the natural morphism $X\rightarrow X_0:=\Spec(\C[X])$ is a resolution of singularities. In this paper we only consider symplectic resolutions $X$ that are projective over $X_0$. We also only care
about resolutions coming with additional structure, a $\C^\times$-action satisfying the following two conditions:
\begin{itemize}
\item The grading induced by the $\C^\times$-action on $\C[X]$ is positive, i.e., $\C[X]=\bigoplus_{i\geqslant 0}\C[X]_i$ and $\C[X]_0=\C$.
\item $\C^\times$ rescales the symplectic form $\omega$, more precisely, there is a positive integer $d$ such that $t.\omega=t^d\omega$
for all $t\in \C^\times$.
\end{itemize}
We call $X$ equipped with such a $\C^\times$-action a {\it  conical symplectic resolution}.

We remark that $X$ admits a universal Poisson deformation, $\widetilde{X}$, over $H^2(X)$. This deformation comes with a $\C^\times$-action and the $\C^\times$-action contracts $\tilde{X}$ to $X$, see \cite[2.2]{quant} or \cite[2.1]{BPW} for details. The generic fiber of $\widetilde{X}$ is affine.

Namikawa associated a Weyl group
$W$ to $X$ that acts on $H^2(X,\mathbb{R})$ as a crystallographic reflection group. We have
$\operatorname{Pic}X=H^2(X,\Z)$.  The (closure of the) movable cone of $X$ in $H^2(X,\mathbb{R})$ is a fundamental chamber for $W$.   Furthermore, there are open subset  $U\subset X, U'\subset X'$ with complements of codimension bigger than $1$ that are
isomorphic. So we get an isomorphism $\operatorname{Pic}(X)\cong\operatorname{Pic}(X')$ that preserves the movable cones.

Namikawa has shown that there are finitely many isomorphism classes of conical symplectic resolutions. Moreover, he proved there is a finite $W$-invariant union $\mathcal{H}$ of hyperplanes in $H^2(X,\mathbb{R})$ with the
the following properties:
\begin{itemize}
\item The union of the complexifications of the hyperplanes in $\mathcal{H}$ is precisely the locus in $H^2(X,\C)$
over which $\widetilde{X}\rightarrow \widetilde{X}_0$ is not an isomorphism.
\item The closure of the movable cone is the union of some chambers for $\mathcal{H}$.
\item Each chamber inside a movable cone is the nef cone of exactly one symplectic resolution.
\end{itemize}

For $\theta\in H^2(X,\mathbb{R})\setminus \mathcal{H}$, let $X^\theta$ be the resolution corresponding to the element of $W\theta$
lying in the movable cone.

We will not need to compute the Namikawa Weyl group. Let us point out that it is trivial provided $X_0$ has no leaves of codimension
2 (we remark that it is known that the number of leaves is always finite).

An example of  a conical symplectic resolution is provided by $\M^\theta(v,w)\rightarrow \M(v,w)$ in the case when $Q$
is an affine quiver ($d=2$ in this case). We always have a natural map $\param:=\C^{Q_0}\rightarrow H^2(\M^\theta(v,w))$ that is always
injective. In the case of an affine quiver, this map can be actually shown to be an isomorphism, but we will not need that: in the quiver variety setting one can retell the constructions above using $\param$ instead of $H^2(\M^\theta(v,w))$.

\subsection{Gieseker moduli spaces}\label{SS_Gies}
We will need some additional facts about varieties $\M^\theta(n,r)$. First of all, let us point out that $\dim \M^\theta(n,r)=2nr$.

Let us note that we have an isomorphism $\M^\theta(n,r)\cong \M^{-\theta}(n,r)$ (of symplectic varieties with $\C^\times$-actions).
 Define $R^\vee:=\End(V^*)^{\oplus 2}\oplus \Hom(V^*,W^*)\oplus \Hom(W^*,V^*)$.
We have an isomorphism $R\rightarrow R^\vee$ given by  $\iota:(A,B,i,j)\mapsto (-B^*, A^*,-j^*, i^*)$, here 
we write $i$ for an element in $\Hom(W,V)$ and $j$ for an element of $\Hom(V,W)$. This is a symplectomorphism.
Choosing bases in $V$ and $W$, we identify $R$ with $R^\vee$. Note that, under this identification, $\iota$ is not $G$-equivariant,
we have $\iota(g.r)=(g^t)^{-1}\iota(r)$, where the superscript ``t'' stands for the matrix transposition. Also note that $(A,B,i,j)$
is $\det$-stable (equivalently, there is no nonzero $A,B$-stable subspace in $\ker j$)
if and only if $(-B^*,A^*,-j^*,i^*)$ is $\det^{-1}$-stable (i.e., $C\langle B^*,A^*\rangle \operatorname{im}j^*=V^*$).
It follows that $\M^\theta(n,r)\cong \M^{-\theta}(n,r)$.

We will also need some information on the cohomology of $\M^\theta(n,r)$.

\begin{Lem}\label{Lem:top_cohom}
We have $H^i(\M^\theta(n,r))=0$ for odd $i$ or for $i\geqslant 2nr$ and $\dim H^2(\M^\theta(n,r))=\dim H^{2nr-2}(\M^\theta(n,r))=1$.
In particular, $\dim H_{2nr-2}(\M^\theta(n,r))=1$.
\end{Lem}
\begin{proof}
That the odd cohomology groups vanish is \cite[Theorem 3.7,(4)]{NY} (or a general fact about symplectic resolutions, see \cite[Proposition 2.5]{BPW}). According to
\cite[Theorem 3.8]{NY}, we have
$$\sum_{i}\dim H^{2i}(\M^\theta(n,r))t^i=\sum_\lambda t^{\sum_{i=1}^r (r|\lambda^{(i)}|-i (\lambda^{(i)t})_1)},$$
where the summation is over the set of the $r$-multipartitions $\lambda=(\lambda^{(1)},\ldots,\lambda^{(r)})$.
The highest power of $t$ in the right hand side is $rn-1$, it occurs for a single $\lambda$, namely, for
$\lambda=((n),\varnothing,\ldots,\varnothing)$. This shows $\dim H^{2nr-2}(\M^\theta(n,r))=1$.
The equality $\dim H_{2nr-2}(\M^\theta(n,r))=1$ follows.  Also there is a  single $r$-multipartition of $n$ with
$\sum_{i=1}^r (r|\lambda^{(i)}|-i (\lambda^{(i)t})_1)=1$, it is $(\varnothing,\ldots,1,n-1)$.
This implies $\dim H^2(\M^\theta(n,r))=1$.
\end{proof}

It follows, in particular, that the universal deformation of $\M^\theta(n,r)$ coincides with
the ``universal quiver variety'' $\M^\theta_{\param}(n,r):=\mu^{-1}(\g^{*G})^{\theta-ss}/G$.
The isomorphism $\M^\theta(n,r)\cong \M^{-\theta}(n,r)$ extends
to $\M^{-\theta}_{\param}(n,r)\cong \M^\theta_\param(n,r)$ that is, however, not an isomorphism
of schemes over $\param$, but rather induces the multiplication by $-1$ on the base.

On $\M^\theta(n,r)$ we have an action of $\GL(r)\times \C^\times$ induced from the following action
on $T^*R$: $(X,t).(A,B,i,j)=(tA,t^{-1}B, Xi,jX^{-1})$. We will need a description of certain torus
fixed points. First, let $T$ denote the maximal torus in $\GL(r)$. Then, see \cite[Lemma 3.2, Section 7]{Nakajima_tensor}, we see  that
\begin{equation}\label{eq:fixed_pt_decomp}\M^\theta(n,r)^T=\bigsqcup_{n_1+\ldots+n_r=n}\prod_{i=1}^r \M^\theta(n_i,1),\end{equation}
The embedding $\prod_{i=1}^r \M^\theta(n_i,1)\hookrightarrow \M^\theta(n,r)$ is induced from
$\bigoplus_{i=1}^r T^*R(n_i,1)\hookrightarrow T^*R(n,r)$.

Now set $\tilde{T}:=T\times \C^\times$. Then $\M^\theta(n,r)^{\tilde{T}}$ is a finite set that
is in a natural bijection with the set of the $r$-multipartitions of $n$, this follows from (\ref{eq:fixed_pt_decomp})
and the classical fact that $\M^\theta(n_i,1)^{\C^\times}$ is identified with the set of the partitions
of $n_i$. More precisely, $\M^\theta(n_i,1)^{\C^\times}=\M^\theta(n_i,1)^{\C^\times\times \C^\times}$,
where the second copy of $\C^\times$ is contracting. When $\theta>0$, we label the fixed point corresponding to
$(A,B,i,j)$ (we automatically have $j=0$) by the partitions of $n_i$ into sizes of Jordan blocks of $B$.

\subsection{Symplectic leaves}\label{SS_leaves}
Here we want to describe the symplectic leaves of $\M^0_\lambda(v,w):=\mu^{-1}(\lambda)\quo G$ and study the
structure of the variety near a symplectic leaf.

Let us, first, study the leaf containing $0\in \M^0(v,w)$. Similarly to Subsection \ref{SS_Gies_result},
consider the space $\bar{R}$ that is
obtained similarly to $R$ but with assigning $\mathfrak{sl}(V_i)$ instead of $\gl(V_i)$ to any loop $a$ with $t(a)=h(a)=i$
so that $R=\overline{R}\oplus \C^k$, where $k$ is the total number of loops. Let $\bar{\M}^0(v,w)$ be the reduction
of $T^*\bar{R}$ so that $\M^0(v,w)=\bar{\M}^0(v,w)\times \C^{2k}$.

\begin{Lem}\label{Lem:0_leaf}
The point $0$ is a single leaf of $\bar{\M}^0(v,w)$.
\end{Lem}
\begin{proof}
It is enough to show that the maximal ideal of $0$ in $\C[T^*\bar{R}]^G$ is Poisson. Since $\bar{R}$ does not include
the trivial $G$-module as a direct summand, we see that all homogeneous elements in $\C[T^*\bar{R}]^G$ have degree $2$
or higher. It follows that the bracket of any two homogeneous elements also has degree 2 or higher and our claim is proved.
\end{proof}

Now let us describe the slices to symplectic leaves in $\M^0_\lambda(v,w)$, see, for example, \cite[2.1.6]{BL}.
Pick $x\in \M^0_\lambda(v,w)$. We can view $T^*R$ as the representation space of dimension $(v,1)$ for the double $DQ^w$
of the quiver $Q^w$ obtained from $Q$ by adjoining the additional vertex $\infty$ with $w_i$ arrows from $i$
to $\infty$. Pick a semisimple representation of $DQ^w$ lying over $x$. This representation decomposes
as $r^0\oplus r^1\otimes U_1\oplus\ldots\oplus r^k\otimes U_k$, where $r^0$ is an irreducible representation
of $DQ^w$ with dimension  $(v^0,1)$ and $r^1,\ldots,r^k$ are pairwise nonisomorphic irreducible representations of $DQ$
with dimension $v^1,\ldots,v^k$. All representations $r^0,\ldots,r^k$ are mapped to $\lambda$ under the moment map.
Consider the quiver $\underline{Q}:=\underline{Q}_x$ with vertices $1,\ldots,k$ and $-(v^i,v^j)$ arrows between vertices
$i,j$ with $i\neq j$ and $1-\frac{1}{2}(v^i,v^i)$ loops at the vertex $i$. We consider the dimension vector
$\underline{v}:=(\dim U_i)_{i=1}^k$ and the framing $\underline{w}=(\underline{w}_i)_{i=1}^k$ with $\underline{w}_i=w\cdot v^i-(v^0,v^i)$.

\begin{Prop}\label{Prop:leaves}
The following is true:
\begin{enumerate}
\item
The symplectic leaves of $\M^0_\lambda(v,w)$ are parameterized by the decompositions $v=v^0+\underline{v}_1 v^1\oplus \ldots \oplus \underline{v}_k v^k$ (we can permute summands with $\underline{v}_i=\underline{v}_j$ and $v^i=v^j$) subject to the following
conditions: there is an irreducible representation $r^0$ of $DQ^w$ of dimension $(v^0,1)$ and pairwise different irreducible
representations $r^1,\ldots,r^k$ of $DQ$ of dimensions $v^1,\ldots,v^k$, all of them mapping to $\lambda$ under the moment
map.
\item The leaf corresponding to the decomposition as above consists of the isomorphism classes of the representations
$r^0\oplus r^1\otimes U_1\oplus\ldots\oplus r^k\otimes U_k$, where $r^0,\ldots,r^k$ are as above.
\item There is a transversal slice to the leaf as above that is isomorphic to the formal
neighborhood of 0 in  the quiver variety $\underline{\bar{\M}}^0_0(\underline{v},\underline{w})$ for the quiver $\underline{Q}$.
\end{enumerate}
\end{Prop}
\begin{proof}
We have a decomposition $\M_\lambda^0(v,w)^{\wedge_x}\cong D\times \underline{\bar{\M}}^0(\underline{v},\underline{w})^{\wedge_0}$
of Poisson formal schemes, where $D$ stands for the symplectic formal disk and $\bullet^{\wedge_x}$ indicates
the formal neighborhood of $x$.  From Lemma \ref{Lem:0_leaf} it now follows
that the locus described in (2) is a union of leaves. So in order to prove the entire proposition, it remains
to show that the locus in (2) is irreducible. This follows from \cite[Theorem 1.2]{CB_geom}.
\end{proof}

Now assume that $X\rightarrow X_0$ is a  symplectic resolution (not necessarily conical). Then $X_0$ has finitely many symplectic leaves.
Pick a point $x\in X_0$ and consider its formal neighborhood $X_0^{\wedge_x}$. Then, according to Kaledin, \cite{Kaledin_sing},
the Poisson formal scheme $X_0^{\wedge_x}$ decomposes into the product of two formal schemes: the symplectic formal
disk $D$, and a ``slice'' $X_0'$ that is a Poisson formal scheme, where $x$ is a single symplectic leaf.

\subsection{Quantizations}\label{SS_quant}
Let $X=X^\theta$ be a conical symplectic resolution corresponding to a parameter $\theta$.
Now let us consider quantizations of $X$. We will work with microlocal quantizations. Those are sheaves $\A^\theta$ of algebras
in conical topology equipped with the following additional structures:
\begin{itemize}
\item a complete and separated ascending $\Z$-filtration, $\A^\theta=\bigcup_{i\in \Z}\A^\theta_{\leqslant i}$,
\item an action of $\Z/d\Z$ (where $d$ has the same meaning as above) on $\A^\theta$ by filtered algebra automorphisms
such that $1\in \Z/d\Z$ acts on $\A^\theta_{\leqslant i}/\A^\theta_{\leqslant i-1}$ by $\exp(2\pi i\sqrt{-1}/d)$.
\item an isomorphism $\gr\A^\theta\cong \mathcal{O}_X$ of sheaves of graded algebras.
\end{itemize}

Consider the subsheaf $R_\hbar(\A^\theta)$ of $\Z/d\Z$-invariants in the Rees sheaf $R_{\hbar^{1/d}}(\A)$.  Completing $R_\hbar(\A)$
with respect to the $\hbar$-adic topology, we get a homogeneous quantization of
$X$ in the sense of \cite[2.3]{quant}. To get back we take $\C^\times$-finite sections and mod out $\hbar-1$. It follows that the microlocal quantizations of $X$ are canonically parameterized
by $H^2(X,\C)$. We write $\A^\theta_{\hat{\lambda}}$ for the quantization corresponding to $\hat{\lambda}\in H^2(X)$
(and we call $\hat{\lambda}$ the {\it period} of the quantization $\A^\theta_{\hat{\lambda}}$).
In fact, we can also quantize the universal deformation $\tilde{X}$ by a microlocal sheaf
$\tilde{\A}^\theta$ of $\C[H^2(X)]$-algebras (the canonical quantization from \cite{BK,quant}). We remark that
$\A_{\hat{\lambda}}^{\theta,opp}\cong \A^\theta_{-\hat{\lambda}}$, as sheaves of algebras on $X$, this follows from the definition
of a canonical quantization.

We write $\A_{\hat{\lambda}}, \tilde{\A}$ for the global sections of $\A^\theta_{\hat{\lambda}},\tilde{\A}^\theta$, these algebras are independent of $\theta$ by \cite[3.3]{BPW}.

When $X$ is a quiver variety $\M^\theta(v,w)$, the quantization $\A^\theta_\lambda(v,w)$ satisfies the assumptions above.
As we have mentioned in Subsection \ref{SS_sympl_res}, we can embed $\param=\C^{Q_0}$ into $H^2(X)$.
We remark however, that $\A^\theta_{\hat{\lambda}}=\A^\theta_{\hat{\lambda}-\varrho}(v,w)$, where $\varrho$ is half the character of the
action of $G$ on $\bigwedge^{top}R^*$, see, e.g., \cite[2.2]{BL}.

Let us now consider the Gieseker case. Here $\varrho=r/2$. So we have
\begin{equation}\label{eq:opp_iso}
\A_{\lambda}(n,r)^{opp}\cong \A_{-\lambda-r}(n,r).
\end{equation}

\begin{Lem}\label{Lem:iso}
We have $\A_\lambda(n,r)\cong \A_{-\lambda-r}(n,r)$.
\end{Lem}
\begin{proof}
Recall that $\A_{\lambda}(n,r)\xrightarrow{\sim} \Gamma(\A_{\lambda}^{\pm \theta}(n,r))$.
Also recall the identification $\M^{\theta}_{\param}(n,r)\rightarrow \M^{-\theta}_{\param}(n,r)$
that induces $-1$ on $\param$. It follows that $\Gamma(\A_{\hat{\lambda}}^\theta)\cong \Gamma(\A^{-\theta}_{-\hat{\lambda}})$.
Since $\A_{\hat{\lambda}}^\theta\cong \A_{\hat{\lambda}-r/2}^\theta(n,r)$, our claim follows.
\end{proof}

We conclude that $\A_{\lambda}(n,r)^{opp}\cong \A_{\lambda}(n,r)$.

We remark that the isomorphism of Lemma \ref{Lem:iso} is similar in spirit to isomorphisms from \cite[Section 3]{BPW}
provided by the Namikawa Weyl group action. We would like to point out however that our isomorphism
does not reduce to that. Indeed, when $r>2$, the Namikawa Weyl group can  be shown trivial because there is no
symplectic leaf of codimension $2$ in $\M(n,r)$.

\subsection{Localization theorems}\label{SS_loc}
We  assume that $X$ is a conical symplectic resolution of $X_0:=\operatorname{Spec}(\C[X])$. Recall that we write $\rho:X\rightarrow X_0$
for the canonical morphism. Let $\A^\theta$ be a quantization of $X$ and $\A$ be its algebra of global sections.

Consider the categories of modules
$\A\operatorname{-Mod}\supset \A\operatorname{-mod}$ consisting of all and of finitely generated $\A$-modules.
Also consider the category $\A^\theta\operatorname{-Mod}$ of  all quasi-coherent $\A^\theta$-modules and $\A^\theta\operatorname{-mod}$
of all coherent $\A^\theta$-modules, i.e., modules that have a global good filtration (a filtration is called good if the associated graded
object is a coherent sheaf of $\mathcal{O}_X$-modules).

We have the global section, $\Gamma^\theta$, and localization, $\Loc^\theta:=\A^\theta\otimes_{\A}\bullet$, functors
$$\Gamma^\theta: \A^\theta\operatorname{-Mod}\rightleftarrows \A\operatorname{-Mod}:\Loc^\theta, \quad \Gamma^\theta: \A^\theta\operatorname{-mod}\rightleftarrows \A\operatorname{-mod}:\Loc^\theta.$$
For objects in $\A^\theta\operatorname{-mod}, \A\operatorname{-mod}$ we can define supports, those are closed
$\C^\times$-stable subvarieties in $X,X_0$, respectively. For a subvariety $Y_0\subset X_0$, we write $\A\operatorname{-mod}_{Y_0}$
for the full subcategory of $\A\operatorname{-mod}$ consisting of all modules supported inside $Y_0$. Similarly,
for $Y\subset X$, we consider the subcategory $\A^\theta\operatorname{-mod}_Y$. The functors $\Gamma^\theta, \Loc^\theta$
restrict to functors between the subcategories $\A^\theta\operatorname{-mod}_{\rho^{-1}(Y_0)}, \A\operatorname{-mod}_{Y_0}$.

The functors $\Gamma^\theta, \Loc^\theta$ admit derived functors $R\Gamma^\theta: D^b(\A^\theta\operatorname{-Mod})\rightarrow
D^b(\A\operatorname{-Mod})$ (given by taking the \v{C}ech complex for a cover by  affine open subsets)
and $L\Loc^\theta: D^-(\A\operatorname{-Mod})\rightarrow D^-(\A^\theta\operatorname{-Mod})$.
If the homological dimension of $\A$ is finite, we also have $L\Loc^\theta:D^b(\A\operatorname{-Mod})\rightarrow D^b(\A^\theta\operatorname{-Mod})$.
Clearly, the functors $R\Gamma^\theta, L\Loc^\theta$ preserve the subcategories $D^?(\A^\theta\operatorname{-mod}), D^?(\A\operatorname{-mod})$ (as in the case of usual coherent sheaves, the former is identified with the full subcategory in
$D^?(\A\operatorname{-Mod})$ with coherent homology). Also the functors $R\Gamma^\theta,L\Loc^\theta$ restrict to  functors between the subcategories $D^?_{Y_0}(\A\operatorname{-mod})$ and $ D^?_{\rho^{-1}(Y_0)}(\A^\theta\operatorname{-mod})$, consisting of all complexes
with homology supported on $Y_0, \rho^{-1}(Y_0)$.

Now let us suppose that $X:=\M^\theta(v,w)$  and so $X_0=\M(v,w)$. We will write $\Gamma_\lambda^\theta, R\Gamma_\lambda^\theta$ to indicate the dependence on the quantization parameter $\lambda$.

Let us recall some results on when (i.e., for which $\lambda$) the functors $\Gamma_\lambda^\theta, \Loc_\lambda^\theta$
are derived or abelian equivalences.

\begin{Prop}[\cite{MN_der}]\label{Prop:MN_der}
The functor $R\Gamma_\lambda^\theta:D^b(\A_\lambda^\theta(v,w)\operatorname{-mod})\rightarrow D^b(\A_\lambda(v,w)\operatorname{-mod})$
is an equivalence of triangulated categories if and only if $\A_\lambda(v,w)$ has finite homological dimension. In this case,
the inverse equivalence is given by $L\Loc_\lambda^\theta$.
\end{Prop}

In the situation of the previous proposition, we say that the derived localization holds (for $\lambda$), such parameters
are called {\it regular}.  \cite[Conjecture 9.1]{BL} describes a precise locus of the singular (=non-regular) parameters
$\lambda$ and we prove this conjecture, Theorem \ref{Thm:loc}.

We say that the abelian localization holds for $(\lambda,\theta)$ if $\Gamma_\lambda^\theta$  is an equivalence of abelian
categories. The following result  was proved in \cite[Corollary 5.12]{BPW} for arbitrary symplectic resolutions.

\begin{Prop}\label{Prop:abel_loc1}
For any $\lambda$ there is $k_0\in \Z_{>0}$ such that the abelian localization holds for $(\lambda+k\theta,\theta)$
whenever $k\geqslant k_0$.
\end{Prop}

There are also results of McGerty and Nevins, \cite{MN_ab}, that provide a sufficient condition for the functor $\Gamma_\lambda^\theta$
to be exact. We will elaborate on these results applied to the special case of $\A_\lambda(n,r)$ below. We will see that for $\A_\lambda(n,r)$ this sufficient condition is also necessary.

To conclude this section let us mention the characteristic cycle map. Suppose we are in the general case of a projective symplectic
resolution $X\rightarrow X_0$. Then to a module $M'\in \mathcal{\A}^\theta\operatorname{-mod}_{\rho^{-1}(0)}$ we can assign
its characteristic cycle. By definition, it coincides with the sum of irreducible components of $\rho^{-1}(0)$
with multiplicities, where the multiplicity
of the component equals to the generic rank of $\gr M'$ on the component. Of course, the characteristic cycle defines a group
homomorphism $\operatorname{CC}^\theta:K_0(\A^\theta\operatorname{-mod}_{\rho^{-1}(0)})\rightarrow H_{mid}(X)$.
Now to a module $M\in \A\operatorname{-mod}_{fin}$ we can assign $\operatorname{CC}^\theta(\operatorname{Loc}_\lambda^\theta M)$.  It was shown in \cite[3.2]{BL} that, in the quiver variety case, this map is actually
independent of $\theta$. Here is another result that will be of crucial importance for us. This was proved in an unpublished work of
Baranovsky and Ginzburg.

\begin{Prop}\label{Prop:CC_inject}
The map $\operatorname{CC}^\theta$ is injective.
\end{Prop}

\subsection{Duality and wall-crossing functor}\label{SS_dual_WC}
We still consider a conical projective resolution $X$ of $X_0$. Let us take a quantization $\A^\theta$ of $X$
that satisfies the abelian localization theorem. In this case the homological dimension of $\A$ (equivalently, of $\A^\theta$)
does not exceed the homological dimension of $\mathcal{O}_X$ equal to $\dim X$.

It turns out that dimension of support for certain modules can be computed via a suitable functor: the homological duality.
Namely, recall the functor $D:=\operatorname{RHom}_{\A}(\bullet, \A)[N]$ where $N=\frac{1}{2}\dim X$,
defines an equivalence $D^b(\A\operatorname{-mod})\rightarrow D^b(\A^{opp}\operatorname{-mod})^{opp}$.
Moreover, for a simple object $L$ in $\A\operatorname{-mod}$ we have $H_i(DL)=0$ if $i>N$
or $i<N-\operatorname{dim}\operatorname{Supp}L$, see \cite[4.2]{BL}. In particular, $L$ is finite dimensional if and only
if $H_i(DL)=0$ for $i<N$.

In the case when $\A=\A_{\hat{\lambda}}$, we can view $D$ as a functor $D^b(\A_{\hat{\lambda}}\operatorname{-mod})
\rightarrow D^b(\A_{-\hat{\lambda}}\operatorname{-mod})^{opp}$.

We will need a technical property of $D$.

\begin{Lem}\label{Lem:D_higher_small_supp}
Let $L$ be a simple $\A_{\hat{\lambda}}$-module with $\dim \operatorname{Supp}L=\frac{1}{2}\dim X$. Then $H^i(DL)$ is supported on
the complement of the open symplectic leaf of $X_0$ for $i>0$.
\end{Lem}
\begin{proof}
As we know from Commutative Algebra, the irreducible components of the support of  $\operatorname{Ext}^i(\operatorname{gr}L, \C[X_0])$ intersecting  the open leaf have dimension smaller than $\frac{1}{2}\dim X$ provided $i>\frac{1}{2}\dim X$.
Thanks to a standard spectral sequence for the homology of a filtered quotient, we see that the irreducible components of
$\operatorname{Supp}H^i(DL)$ for $i>0$ intersecting the open leaf have dimension less than $\frac{1}{2}\dim X$. However,
no nonzero $\A_{-\hat{\lambda}}$-module can have this property as the support of any $\A_{-\hat{\lambda}}$-module
is a coisotropic subvariety by Gabber's theorem.
\end{proof}

Now let us consider a different family of functors: wall-crossing functors. Below we will recall a connection of some of those
functors with the duality introduce above that was discovered in \cite[Section 4]{BL}.
We will make some
additional assumptions. Let us assume that the all conical projective symplectic resolutions
of $X_0$ are strictly semismall. Recall that $\rho:X\rightarrow X_0$ is called strictly semismall,
if all components of $X^i:=\{x\in X| \dim \rho^{-1}(x)=i\}$ have codimension $2i$ and all components
of $\rho^{-1}(x)$ have the same dimension.

Pick $\chi\in \operatorname{Pic}(X)$. We can uniquely quantize the corresponding line bundle $\mathcal{O}(\chi)$ on $X$
to a $\A^\theta_{\hat{\lambda}+\chi}$-$\A^\theta_{\hat{\lambda}}$-bimodule to be denoted by $\A^{\theta}_{\hat{\lambda},\chi}$.
Let $\A^{(\theta)}_{\hat{\lambda},\chi}$ denote the global sections. We remark that taking the tensor product with
$\A^{\theta}_{\hat{\lambda},\chi}$ gives rise to an equivalence $\mathcal{T}^\theta_{\hat{\lambda},\chi}:\A_{\hat{\lambda}}^\theta\operatorname{-Mod}\rightarrow
\A^{\theta}_{\hat{\lambda}+\chi}\operatorname{-Mod}$.

Pick a different stability condition $\theta'$ and $\hat{\lambda}'\in \hat{\lambda}+\Z^{Q_0}$ such that the abelian localization holds for $(\hat{\lambda}',\theta')$.  We consider the  wall-crossing functor $$\WC_{\hat{\lambda}\rightarrow \hat{\lambda}'}:=\Gamma^{\theta'}_{\hat{\lambda}'}\circ
\mathcal{T}^{\theta'}_{\hat{\lambda},\hat{\lambda}'-\hat{\lambda}}\circ L\operatorname{Loc}^{\theta'}_{\hat{\lambda}}:D^b(\A_{\hat{\lambda}}\operatorname{-mod})
\rightarrow D^b(\A_{\hat{\lambda}'}\operatorname{-mod}).$$ Here we write $\Gamma^{\theta'}_{\hat{\lambda}'}$ for the global
section functor $\A_{\hat{\lambda}'}^{\theta'}\operatorname{-Mod}\rightarrow \A_{\hat{\lambda}'}\operatorname{-Mod}$.   According to \cite[Proposition 6.29]{BPW},
we have $\WC_{\hat{\lambda}\rightarrow \hat{\lambda}'}=\A_{\hat{\lambda},\hat{\lambda}'-\hat{\lambda}}^{(\theta')}\otimes^L_{\A_{\hat{\lambda}}}\bullet$.
In the case of quiver varieties, we will write $\WC_{\lambda\rightarrow \lambda'}$ for a functor
$D^b(\A_\lambda(v,w)\operatorname{-Mod})\rightarrow D^b(\A_{\lambda'}(v,w)\operatorname{-Mod})$.

By a long wall-crossing functor we mean $\WC_{\hat{\lambda}\rightarrow \hat{\lambda}'}$, where $(\hat{\lambda},\theta),(\hat{\lambda}',-\theta)$
satisfy the abelian localization.
A connection between the contravariant duality and the long wall-crossing functor is as follows.
We say that an $\A_{\hat{\lambda}}$-module $M$ is {\it strongly holonomic} if every nonempty
intersection of $\operatorname{Supp}M$ with a symplectic leaf in $X_0$ is lagrangian
in that leaf. Thanks to the assumption that
$X\rightarrow X_0$ is strictly semismall, this is equivalent to the condition that $\rho^{-1}(\operatorname{Supp}M)
\subset X$ is lagrangian. Set $\hat{\lambda}^-:=\hat{\lambda}-n\theta$ for sufficiently large $n$. The choice of $n$ guarantees that
the abelian localization holds for $(\hat{\lambda}^-,-\theta)$.

Consider the subcategory $D^b_{shol}(\A_{\hat{\lambda}}\operatorname{-mod})\subset D^b(\A_{\hat{\lambda}}\operatorname{-mod})$
of all complexes with strongly equivariant homology. It is easy to see that $D$ restricts to an equivalence
$D^b_{shol}(\A_{\hat{\lambda}}\operatorname{-mod})\xrightarrow{\sim} D^b_{shol}(\Af_{-\hat{\lambda}}\operatorname{-mod})^{opp}$.
On the other hand, $\WC_{\hat{\lambda}\rightarrow \hat{\lambda}^-}$ restricts to an equivalence
$D^b_{shol}(\A_{\hat{\lambda}}\operatorname{-mod})\xrightarrow{\sim} D^b_{shol}(\A_{\hat{\lambda}^-}\operatorname{-mod})$.
The following result was proved in \cite[Section 4]{BL}.

\begin{Prop}\label{Prop:WC_vs_D}
There is an equivalence $\iota: D^b_{shol}(\A_{\hat{\lambda}^-}\operatorname{-mod})\xrightarrow{\sim} D^b_{shol}(\A_{-\hat{\lambda}}\operatorname{-mod})^{opp}$ preserving the natural $t$-structures
such that $\iota\circ \WC_{\hat{\lambda}\rightarrow \hat{\lambda}^-}=D$.
\end{Prop}

We will need a corollary of this proposition.

\begin{Cor}\label{Cor:full_supp}
Let $M$ be a simple strongly holonomic $\A$-module. The following are equivalent:
\begin{enumerate}
\item $\dim \operatorname{Supp}M<\frac{1}{2}\dim X$.
\item $M$ is annihilated by a proper ideal of $\A$.
\end{enumerate}
\end{Cor}
\begin{proof}
Let us note that the associated graded of a proper ideal in $\A$ is a Poisson ideal in $\C[X_0]$.
So its associated variety does not intersect the open leaf in $X_0$. Obviously, the support of $M$ is contained
in that associated variety. Since $M$ is strongly holonomic, the implication (2)$\Rightarrow$(1) follows.

Let us prove (1)$\Rightarrow$(2). Consider the $\A=\A_{\hat{\lambda}}$-bimodule $\mathcal{D}:=\A_{\hat{\lambda}^-, N\theta}^{(\theta)}
\otimes_{\A_{\hat{\lambda}^-}}\A_{\hat{\lambda},-N\theta}^{(-\theta)}$. Obviously, $H_0(\WC_{\lambda^-\rightarrow \lambda}\circ
\WC_{\lambda\rightarrow \lambda^-}M)=\mathcal{D}\otimes_{\A_{\hat{\lambda}}}M$. There is a natural homomorphism $\mathcal{D}\rightarrow \A_{\hat{\lambda}}$ (compare to \cite[(5.15)]{BL}) that becomes an isomorphism after microlocalization to the open leaf of $X_0$ because $\rho$ is an isomorphism over
the open leaf. So the image is a nonzero ideal in $\A_{\hat{\lambda}}$, say $J$. If $M$ is not annihilated by that
ideal, we see that $H_0(\WC_{\lambda\rightarrow \lambda^-}M)\neq 0$. It follows that $\dim\operatorname{Supp}M=\frac{1}{2}\dim X$.
(1)$\Rightarrow$(2) is proved.
\end{proof}

We will also need a straightforward corollary of Lemma \ref{Lem:D_higher_small_supp} for strongly holonomic modules.

\begin{Cor}\label{Cor:D_higher_codim}
Let $L$ be a simple strongly holonomic $\A_{\hat{\lambda}}$-module. Then $\dim \operatorname{Supp}H^i(DL)<\frac{1}{2}\dim X$
provided $i>0$.
\end{Cor}

\subsection{Categories $\mathcal{O}$}\label{SS_Cat_O}
Basically, all results of this section can be found in \cite{BLPW}.

First of all, let $\A$ be an associative algebra equipped with a rational action $\alpha$ of $\C^\times$ by algebra automorphism.
Then we can consider the eigendecomposition $\A=\bigoplus_{i\in \Z}\A_i$
and set $\A_{\geqslant 0}:=\bigoplus_{i\geqslant 0}\A_i, \A_{>0}:=\bigoplus_{i>0}\A_i$ and $\Ca_\alpha(A):=\A_{\geqslant 0}/(\A_{\geqslant 0}\cap \A\A_{>0})$. We remark that $\Ca_\alpha(\A)$ is an algebra because $\A\A_{>0}\cap \A_{\geqslant 0}$ is a two-sided ideal in $\A_{\geqslant 0}$. We remark that $\Ca_\alpha$ is a functor from the category of algebras equipped with a  $\C^\times$-action to the category of algebras, we call it the {\it Cartan functor}. This name is justified by the observation that if $\A=U(\g)$ for a semisimple Lie algebra $\g$ and $\alpha$ comes from a regular one-parametric subgroup of $\operatorname{Ad}(\g)$, then $\Ca_\alpha(\A)$ is the universal enveloping algebra of the corresponding Cartan subalgebra.

Let $X$ be a symplectic resolution of $X_0$ and $\A^\theta$ be a quantization of $X$. We assume that that  $X$
is equipped with compatible Hamiltonian $\C^\times$-action $\alpha$. We require that the action on
$X$ commutes with the contracting $\C^\times$-action. It is not difficult to see that $\alpha$
lifts to a Hamiltonian $\C^\times$-action on $\A^\theta$ again denoted by $\alpha$. This action
preserves the filtration and the $\Z/d\Z$-grading. Let $h_\alpha\in \A$ denote the image of $1$ under the
quantum comoment map for $\alpha$.

Consider the category $\mathcal{O}(\A)$ consisting of all modules with locally finite action of $h_\alpha,\A_{>0}$. The action of $\A_{>0}$ is automatically locally nilpotent. If $\alpha(\C^\times)$ has finitely many fixed points, this definition coincides with the definition of the category $\mathcal{O}_a$ from \cite[3.2]{BLPW} (this because the algebra $\Ca_\alpha(\A)$ is finite dimensional, which is proved analogously to \cite[3.1.4]{GL}).

The category $\mathcal{O}(\A)$ has analogs of Verma modules. More precisely, there is an induction functor $\Ca_\alpha(\A)\operatorname{-mod}\rightarrow \OCat(\A), M^0\mapsto \Delta(M^0):=\A\otimes_{\A_{\geqslant 0}}M^0$. By a Verma module, we mean $\Delta(M^0)$ with simple $M^0$.

Now let us consider the case when $\alpha(\C^\times)$ has finitely many fixed points. For $\hat{\lambda}\in H^2(X)$ lying in a Zariski open subset, the algebra $\Ca_\alpha(\A_{\hat{\lambda}})$ is naturally isomorphic to $\C[X^{\alpha(\C^\times)}]$, see \cite[5.1]{BLPW}. In this case, for $p\in X^{\alpha(\C^\times)}$, we will write  $\Delta_{\hat{\lambda}}(p)$ for the corresponding Verma module.

Now let us define the category $\mathcal{O}$ for $\A^\theta$ following \cite[3.3]{BLPW}. Let $Y$ stand for the
contracting locus for $\alpha$, i.e, the subvariety of all points $x\in X$ such that $\lim_{t\rightarrow 0} \alpha(t)x$
exists (and automatically lies in $X^{\alpha(\C^\times)}$). We remark that $Y$ is a  lagrangian subvariety in $X$
stable with respect to the contracting $\C^\times$-action.  We also remark
that $Y=\rho^{-1}(Y_0)$, where $Y_0$ stands for the contracting  locus of the $\C^\times$-action on $X_0$
induced by $\alpha$. This is because, under our assumptions on $X^{\alpha(\C^\times)}$, the fixed point set
$X_0^{\alpha(\C^\times)}$ is a single point and because $\rho$ is proper.
We remark that if $X$ is strictly semismall, then every module in $\mathcal{O}(\A)$
is strongly holonomic. This is because $Y=\rho^{-1}(Y_0)$ is lagrangian.

By definition, the category $\mathcal{O}(\A^\theta)$ consists of all modules $\M\in \A^\theta\operatorname{-mod}_Y$
that admit a global good $h_\alpha$-stable filtration.

We write $D^b_{\mathcal{O}}(\A_{\hat{\lambda}}), D^b_{\mathcal{O}}(\A_{\hat{\lambda}}^\theta)$ for the categories of all complexes (in the corresponding derived categories)  with homology in the categories $\mathcal{O}$.

Let us summarize some properties of categories $\mathcal{O}(\A^\theta),\mathcal{O}(\A)$.

\begin{Prop}\label{Prop:O_prop}
Assume that the action $\alpha$ has finitely many fixed points.
\begin{enumerate}
\item We have $\Gamma^\theta(\OCat(\A^\theta))\subset \OCat(\A),\Loc^\theta(\OCat(\A))\subset \OCat(\A^\theta), R\Gamma^\theta(D^b_{\mathcal{O}}(\A^\theta\operatorname{-mod}))\subset D^b_{\mathcal{O}}(\A\operatorname{-mod}),$ $L\operatorname{Loc}^\theta(D^b_{\mathcal{O}}(\A\operatorname{-mod}))
    \subset D^b_{\mathcal{O}}(\A^\theta\operatorname{-mod})$.
\item The functor $\A_{\hat{\lambda},\chi}^{\theta}\otimes_{\A_{\hat{\lambda}}^\theta}\bullet$ maps $\OCat(\A_{\hat{\lambda}}^\theta)$ to
$\OCat(\A_{\hat{\lambda}+\chi}^\theta)$.
\item The categories $\OCat(\A),\OCat(\A^\theta)$ are length categories, i.e., all objects have finite length.
\item $\OCat(\A)\subset \A\operatorname{-mod}, \OCat(\A^\theta)\subset \A^\theta\operatorname{-mod}$
are Serre subcategories.
\item All modules in $\OCat(\A), \OCat(\A^\theta)$ can be made weakly $\alpha(\C^\times)$-equivariant.
\item
Conversely, all weakly $\alpha(\C^\times)$-equivariant modules in $\A\operatorname{-mod}_{Y_0}$ (resp., $\A^\theta\operatorname{-mod}_{Y}$)
are in $\mathcal{O}(\A)$ (resp., in $\mathcal{O}(\A^\theta)$).
\end{enumerate}
\end{Prop}
\begin{proof}
(1)-(4) were established in \cite[Section 3]{BLPW}. The proof of (5) for $\OCat(\A)$ is standard:
we decompose a module in $\OCat(\A)$ into the direct sum of submodules according to the class of eigenvalues of $h_\alpha$ modulo
$\Z$. It is easy to introduce a weakly equivariant structure on each summand.

Since the localization functor is $\alpha(\C^\times)$-equivariant, we see that $\operatorname{Loc}^\theta(M)$ can be made weakly $\alpha(\C^\times)$-equivariant. So (5) for $\OCat(\A_{\hat{\lambda}}^\theta)$ is true provided the abelian localization holds for $\hat{\lambda}$. So it also holds for $\hat{\lambda}+\chi$ for any integral $\chi$. This is because of (2) and the observation that $\A_{\hat{\lambda},\chi}^\theta$ is $\alpha(\C^\times)$-equivariant. Now our claim follows from Proposition \ref{Prop:abel_loc1}.

Let us prove (6). Again, thanks to the equivariance of the localization functor, it is enough to prove the claim for $\A$.
Pick a weakly $\alpha(\C^\times)$-equivariant module $M\in \A\operatorname{-mod}_{Y_0}$ and let $M=\bigoplus_{i\in \Z}M_i$
be the  eigen-decomposition for the $\alpha(\C^\times)$-action. Since $M$ is supported on $Y_0$, we see that $M_i=0$ for $i\gg 0$.
Since $M$ is finitely generated, it is easy to see that all weight spaces are finite dimensional. Our claim follows.
\end{proof}

Now let us discuss highest weight structure on $\OCat(\A^\theta)$. Consider the so called geometric order on $X^{\alpha(\C^\times)}$ defined as follows. For $p\in X^{\alpha(\C^\times)}$ set $Y_p:=\{x\in X| \lim_{t\rightarrow 0}\alpha(t)x=p\}$, the contracting locus of $p$
so that we have $Y=\bigsqcup_{p\in X^{\alpha(\C^\times)}}Y_p$.
We consider the relation $\leqslant^\theta$ on $X^{\alpha(\C^\times)}$ that is the transitive order
of the pre-order $p\leqslant^\theta p'$ if $p\in \overline{Y}_{p'}$. We write $Y_{\leqslant p'}$ for $\bigsqcup_{p\leqslant^\theta p'}Y_p$
and $Y_{<p'}$ for $Y_{\leqslant p'}\setminus Y_{p'}$.
It was shown in \cite[Section 5]{BLPW} that $\OCat(\A^\theta)$ is a highest weight category with respect to this order. The
standard object $\Delta_{\hat{\lambda}}^\theta(p)$ corresponding to $p$ is a unique indecomposable projective in
$\OCat(\A^\theta)\cap \A^\theta\operatorname{-mod}_{Y_{\leqslant p}}$ that is not contained
in $\A^\theta\operatorname{-mod}_{Y_{<p}}$.

If the abelian localization holds for $(\hat{\lambda},\theta)$, then $\OCat(\A_{\hat{\lambda}})$ is also a highest weight category. Further, if we assume, in addition, that $\hat{\lambda}$ lies in a Zariski open subset, then the standard objects in $\OCat(\A_{\hat{\lambda}})$ are precisely the  Verma modules, see \cite[5.2]{BLPW}.

Now let us examine a connection between various derived categories associated to $\OCat(\A_{\hat{\lambda}}),\OCat(\A_{\hat{\lambda}}^\theta)$.
This is basically the appendix to \cite{BLPW} by the author.

\begin{Lem}\label{Lem:Ext_coinc}
The natural functor $D^b(\mathcal{O}(\mathcal{A}_{\hat{\lambda}}^\theta))\rightarrow D^b_{\mathcal{O}}(\A_{\hat{\lambda}}^\theta\operatorname{-mod})$
is an equivalence. Furthermore, if $\hat{\lambda}\in H^2(X)$ lies in a suitable Zariski open subset and
the abelian localization holds for
$(\hat{\lambda},\theta)$, then $D^b(\mathcal{O}(\mathcal{A}_{\hat{\lambda}}))\rightarrow  D^b_{\mathcal{O}}(\A_{\hat{\lambda}}^\theta)$
is an equivalence.
\end{Lem}

We are going to identify $D^b(\mathcal{O}(\A_{\hat{\lambda}}^\theta))$ with $D^b_{\mathcal{O}}(\A_{\hat{\lambda}}^\theta\operatorname{-mod})$
and $D^b(\mathcal{O}(\A_{\hat{\lambda}}))$ with $D^b_{\mathcal{O}}(\A_{\hat{\lambda}})$.

In the case of $\A_{\lambda}(n,r)$ we consider the categories $\mathcal{O}$ defined for the action of a generic
one-dimensional subtorus in $\GL(r)\times \C^\times$. Then the fixed point are in one-to-one correspondence with the
$r$-multipartitions of $n$. Different choices of a generic torus may give rise to different categories $\mathcal{O}$.

\subsection{Harish-Chandra bimodules}\label{SS_HC_bimod}
Let us recall some basics on Harish-Chandra (shortly, HC) bimodules. Let $\A^\theta,\A'^\theta$ be two quantizations of
$X$ in the sense of Subsection \ref{SS_quant} and $\A,\A'$ be their global sections. For technical reasons, we make a restriction on the contracting $\C^\times$-action $\beta$ on $X$. Namely, we require that there are commuting actions $\underline{\beta}$ and $\gamma$ of $\C^\times$ on $X$ with the following properties:
\begin{itemize}
\item $\beta=\beta'^d \gamma$,
\item and $\gamma$ is Hamiltonian.
\end{itemize}
Then $\gamma$ lifts to $\A^\theta,\A'^\theta$ and these sheaves acquire new filtrations, coming from $\underline{\beta}$.
In the remainder of the section we consider $\A,\A'$ with these new filtrations. Let us point out that $X=\M^\theta(v,w)$
does satisfy our additional condition: we can take $\underline{\beta}$ induced by the action $t.(r,\alpha)=(r,t^{-1}\alpha)$.

Let us recall the definition
of a HC $\A$-$\A'$-bimodule. By definition, this is a finitely generated
$\A$-$\A'$-bimodule $\B$  with a filtration that is compatible with those on $\A,\A'$ such that $\gr\B$ is a $\C[X_0]$-module. By a homomorphism of Harish-Chandra bimodules we mean a  bimodule homomorphism. The category of HC $\A'$-$\A$-bimodules is denoted by $\HC(\A'\text{-}\A)$. We also consider the full subcategory $D^b_{HC}(\A'-\A)$ of the derived category of  $\A'$-$\A$-bimodules with Harish-Chandra homology.

By the associated variety of a HC bimodule $\B$ (denoted by $\VA(\B)$) we mean the support in $X_0$ of the coherent sheaf $\gr\B$, where the associated graded is taken with respect to a filtration as in the previous paragraph (below we call such filtrations {\it good}). It is easy to see that $\gr\B$ is a Poisson $\C[X_0]$-module so $\VA(\B)$ is the union of symplectic leaves.

Using associated varieties and the finiteness of the number of the leaves it is easy to prove the following standard result.

\begin{Lem}\label{Lem:fin_length}
Any HC bimodule has finite length.
\end{Lem}

For $\B^1\in \HC(\A'\text{-}\A)$ and $\B^2\in \HC(\A''\text{-}\A')$ we can take their tensor product $\B^2\otimes_{\A'}\B^1$. This is easily
seen to be a HC $\A''$-$\A$-bimodule. Also the derived tensor product of the objects from $D^b_{HC}(\A''\text{-}\A'),D^b_{HC}(\A'\text{-}\A)$
lies in $D^-_{HC}(\A''\text{-}\A)$ (and in $D^b_{HC}(\A''\text{-}\A)$ provided $\A'$ has finite homological dimension).

\subsection{Quantum slices}\label{SS_quant_slice}
Let $X\rightarrow X_0$, where the contracting $\C^\times$-action satisfies the additional assumptions
imposed  in the previous subsection. Pick a quantization $\A^\theta$ of $X$.
We write $\A_\hbar^\theta$ for the $\hbar$-adic completion of the Rees sheaf $R_\hbar(\A^\theta)$ of $\A^\theta$ (with respect to the action $\underline{\beta}$ so that $t.\hbar=t\hbar$). We write $\A_\hbar$ for the algebra of the global sections of $\A^\theta_\hbar$, this is the $\hbar$-adic completion of the Rees algebra of $\A$.

Pick a point $x\in X_0$.  Then we can form the completions $\A_{\hbar}^{\wedge_x}$ of $\A_\hbar$ at $x$
and $\A_\hbar^{\theta\wedge_x}$ of $\A_\hbar^{\wedge_x}$ at $\rho^{-1}(x)$. Consider the homogenized Weyl algebra
$\mathbb{A}_\hbar$ for the tangent space to the symplectic leaf in $x$. Then  we have an embedding $\mathbb{A}_\hbar^{\wedge_0}\hookrightarrow \A_\hbar^{\wedge_x}$, see \cite[2.1]{W-prim} for  a proof. It was checked in \cite[2.1]{W-prim} that we have the tensor product decomposition
$\A_{\hbar}^{\wedge_x}=\mathbb{A}_\hbar^{\wedge_0}\widehat{\otimes}_{\C[[\hbar]]}\A_\hbar'$ that lifts
the decomposition $X^{\wedge_x}\cong D\times X'$ mentioned in Subsection  \ref{SS_leaves}.
The algebra $\A_\hbar'$ is independent of the choices
up to an isomorphism, as was explained in \cite[2.1]{W-prim}. For a similar reason, we have a decomposition
$\A_{\hbar}^{\theta\wedge_x}\cong \mathbb{A}_\hbar^{\wedge_0}\widehat{\otimes}_{\C[[\hbar]]}{\A_\hbar^{\theta}}'$,
where $\A_\hbar'^{\theta}$ is a formal quantization of the slice $X'$. By the construction, $\A_\hbar'=\Gamma(\A_\hbar'^{\theta})$.

Now suppose that $\A^\theta$ has period $\hat{\lambda}$. Then the period of $\A_\hbar^{\theta\wedge_x}$ coincides
with the image $\hat{\lambda}'$ of $\hat{\lambda}$ under the natural map between the \v{C}ech-De Rham cohomology groups $H^2(X)\rightarrow H^2(X^{\wedge_x})=H^2(X')$. It follows that ${\A'^\theta_\hbar}$ also has period $\hat{\lambda}'$.

Assume that $X'$ is again equipped with a contracting $\C^\times$ action with the same integer $d$ satisfying
the additional restriction in Subsection \ref{SS_HC_bimod}, this holds in the quiver variety setting,
for example. So $X'$ is the formal
neighborhood at $0$ of a conical symplectic resolution $\underline{X}$.
The formal quantization $\A'^\theta_\hbar$ is homogeneous (by the results
of \cite[2.3]{quant}). It follows that it is obtained by completion at $0$ of $\underline{\A}_\hbar^\theta$ for some
quantization $\underline{\A}^\theta$ of $\underline{X}$.

So the product $\A_{\hbar}^{\theta\wedge_x}=\mathbb{A}_\hbar^{\wedge_0}\widehat{\otimes}_{\C[[\hbar]]}{\A_\hbar'^{\theta}}$ comes equipped
with a $\C^\times$-action by algebra automorphisms satisfying $t.\hbar=t\hbar$.
On the other hand, the $\C^\times$-action on $\A_\hbar^{\theta}$ produces a derivation of $\A_\hbar^{\theta\wedge_x}$.
The difference between this derivation and the one produced by the $\C^\times$-action on $\A_\hbar^{\theta\wedge_x}$
has the form $\frac{1}{\hbar}[a,\cdot]$ for some $a\in \A_\hbar^{\wedge_x}$, see \cite[Lemma 5.7]{BL}.

Let us now elaborate on the Gieseker case, which was already considered (in a more general quiver variety case)
in \cite[5.4]{BL}. Recall that the symplectic leaves in $\bar{\M}(n,r)$ are parameterized by partitions
$(n_1,\ldots,n_k)$ with $n_1+\ldots+n_k\leqslant n$. For $x$ in the corresponding leaf, we have $$\underline{\bar{\A}}_{\lambda}(n,r)=
\bar{\A}_{\lambda}(n_1,r)\otimes\bar{\A}_{\lambda}(n_2,r)\otimes\ldots\otimes\bar{\A}_{\lambda}(n_k,r).$$
Also we have a similar decomposition for $\underline{\bar{\A}}_{\lambda}^\theta(n,r)$.

\subsection{Restriction functors for HC bimodules}\label{SS_restr_fun}
In this subsection, we will recall restriction functors $\bullet_{\dagger,x}:\HC(\A'\text{-}\A)\rightarrow \HC(\underline{\A}'\text{-}\underline{\A})$,
where $\underline{\A}',\underline{\A}$ are slice algebras for $\A',\A$, respectively (it is in order for these functors
to behave nicely that we have introduced our additional technical assumption on the contracting $\C^\times$-action). These functors were defined in \cite[Section 5]{BL}
in the case when $\A',\A$ are of the form $\A_\lambda(v,w)$, but the general case (under the assumption that the slice
to $x$ is conical and the contracting action satisfies the additional assumption) is absolutely analogous. We will need several facts about the restriction functors established in
\cite[Section 5]{BL}.

\begin{Prop}\label{Prop:dagger_prop}
The following is true.
\begin{enumerate}
\item The functor $\bullet_{\dagger,x}$ is exact and $H^2(X)$-linear.
\item
The associated variety $\VA(\B_{\dagger,x})$ is uniquely characterized by the property $D\times \VA(\B_{\dagger,x})^{\wedge_0}=\VA(\B)^{\wedge_x}$.
\item The functor $\bullet_{\dagger,x}$ intertwines the Tor's: for $\B^1\in \HC(\A''\text{-}\A')$ and $\B^2\in \HC(\A'\text{-}\A)$
we have a natural isomorphism $\operatorname{Tor}_i^{\A'}(\B^1,\B^2)_{\dagger,x}=\operatorname{Tor}_i^{\underline{\A}'}(\B^1_{\dagger,x}, \B^2_{\dagger,x})$.
\end{enumerate}
\end{Prop}
\section{Parabolic induction}\label{S_parab_ind}
The first goal of this section is to elaborate on the Cartan functor that appeared in Subsection \ref{SS_Cat_O}.
There we were basically dealing with the case when the Hamiltonian action only has finitely many fixed points.
Here we consider a more general case and our goal is to better understand the structure of $\Ca_\alpha(\A)$.
The second goal is to introduce parabolic induction  for categories $\mathcal{O}$.

Not surprisingly, the case of actions on smooth symplectic (even non-affine) varieties is easier to understand.
We extend the definition of $\Ca_\alpha$ to sheaves in Subsection \ref{S_Cart_fun}. There we show that
if $\A^\theta$ is a quantization of $X$, then $\Ca_\alpha(\A^\theta)$ is a quantization of $X^{\alpha(\C^\times)}$.
In Subsection \ref{SS_Ca_sheaf_vs_alg} we compare $\Ca_\alpha(\A)$ (an algebra which is hard to understand directly)
with $\Gamma(\Ca_\alpha(\A^\theta))$ in the case of symplectic resolutions. We will see that, for a Zariski generic quantization
parameter, the two algebras coincide. Next, in Subsection \ref{SS_Ca_param}, we determine the quantization
parameter (=period) of $\Ca_\alpha(\A^\theta)$ from that of $\A^\theta$. Subsection \ref{SS_Ca_Gies} applies
this result to some particular action $\alpha$ in the Gieseker case.

Finally, in Subsection \ref{SS_parab_induc} we introduce parabolic induction.

\subsection{Cartan functor for sheaves}\label{S_Cart_fun}
We start with a symplectic variety $X$ equipped with a $\C^\times$-action that rescales the symplectic form
and also with a commuting Hamiltonian action $\alpha$. Of course, it still makes sense to speak about quantizations
of $X$ that are Hamiltonian for $\alpha$.
We want to construct a quantization $\Ca_\alpha(\A^{\theta})$ of $X^{\alpha(\C^\times)}$ starting from a Hamiltonian quantization
$\A^\theta$ of $X$.

The variety $X$ can be covered by $(\C^{\times})^{2}$-stable open affine subvarieties.
Pick such a subvariety $X'$ with $(X')^{\alpha(\C^\times)}\neq \varnothing$.
Define $\Ca_\alpha(\A^\theta)(X')$ as $\Ca_\alpha(\A^\theta(X'))$.
We remark that the open subsets of the form $(X')^{\alpha(\C^\times)}$ form a base of the Zariski topology on $X^{\alpha(\C^\times)}$.

The following proposition defines the sheaf $\Ca_\alpha(\A^\theta)$.
%

\begin{Prop}\label{Prop:A0}
The following holds.
\begin{enumerate}
\item Suppose that the contracting $\alpha$-locus in $X'$ is a complete intersection defined by homogeneous (for $\alpha(\C^\times)$) equations of positive weight. Then the algebra $\Ca_\alpha(\A^\theta(X'))$ is a  quantization of $\C[X'^{\alpha(\C^\times)}]$.
\item There is a unique sheaf $\Ca_\alpha(\A^\theta)$ of $X^{\alpha(\C^\times)}$ whose sections on $X'^{\alpha(\C^\times)}$
with $X'$ as above coincide with $\Ca_\alpha(\A^\theta(X'))$. This sheaf is a quantization of $X^{\alpha(\C^\times)}$.
\item If $X'$ is a $(\C^\times)^2$-stable affine subvariety, then $\Ca_\alpha(\A^\theta)(X')=\Ca_\alpha(\A^\theta(X'))$.
\end{enumerate}
\end{Prop}
\begin{proof}
Let us prove (1). To simplify the notation, we write $\A$ for $\A^\theta(X')$.
The algebra $\A$ is Noetherian because it is complete and separated with respect to
a filtration whose associated graded is Noetherian.   Let us show that
$\operatorname{gr}\A\A_{>0}=\C[X']\C[X']_{>0}$, this will complete the proof of (1).

In the proof it is more convenient to deal with $\hbar$-adically completed homogenized quantizations. Namely,
let $\A_\hbar$ stand for the $\hbar$-adic completion of $R_\hbar(\A)$. The claim that
$\operatorname{gr}\A\A_{>0}=\C[X']\C[X']_{>0}$ is equivalent to the condition that
$\A_\hbar\A_{\hbar,>0}$ is $\hbar$-saturated meaning that $\hbar a\in \A_{\hbar}\A_{\hbar,>0}$
implies that $a\in \A_{\hbar}\A_{\hbar,>0}$.

Recall that we assume that there are $\alpha$-homogeneous elements $f_1,\ldots,f_k\in \C[X']_{>0}$ that form a regular sequence
generating the ideal $\C[X']\C[X']_{>0}$. We can lift those elements to homogeneous $\tilde{f}_1,\ldots, \tilde{f}_k\in \A_{\hbar,>0}$.
We claim that these elements still generate $\A_\hbar\A_{\hbar,>0}$. Indeed, it is enough to check that $\A_{\hbar,>0}\subset \operatorname{Span}_{\A_\hbar}(\tilde{f}_1,\ldots,\tilde{f}_k)$. For a homogeneous element $f\in \A_{\hbar,>0}\setminus \hbar\A_{\hbar}$
we can find homogeneous elements $g_1,\ldots,g_k$
such that $f-\sum_{i=1}^k g_i\tilde{f}_i$ still has the same $\alpha(\C^\times)$-weight and is divisible by
$\hbar$. Divide by $\hbar$ and  repeat the
argument. Since the $\hbar$-adic topology is complete and separated, we see that $f\in \operatorname{Span}_{\A_\hbar}(\tilde{f}_1,\ldots,\tilde{f}_k)$. So it is enough to check that $\operatorname{Span}_{\A_\hbar}(\tilde{f}_1,\ldots,\tilde{f}_k)$ is $\hbar$-saturated.

Pick elements $\tilde{h}_1,\ldots,\tilde{h}_k$ such that $\sum_{j=1}^k \tilde{h}_j\tilde{f}_j$ is divisible by $\hbar$.
Let $h_j\in \C[X']$ be congruent to $\tilde{h}_j$ modulo $\hbar$
so that $\sum_{j=1}^k h_j f_j=0$. Since $f_1,\ldots,f_k$ form a regular sequence, we see that
there are  elements $h_{ij}\in \C[X']$ such that $h_{j'j}=-h_{jj'}$ and $h_j=\sum_{\ell=1}^k h_{j\ell}f_\ell$.
Lift the elements $h_{jj'}$ to $\tilde{h}_{jj'}\in \A_\hbar$ with $\tilde{h}_{jj'}=-\tilde{h}_{j'j}$.
So we have $\tilde{h}_j=\sum_{\ell=1}^k \tilde{h}_{j\ell}\tilde{f}_\ell+\hbar\tilde{h}'_j$ for
some $\tilde{h}'_j\in \A_\hbar$.
It follows that $\sum_{j=1}^k \tilde{h}_j \tilde{f}_j= \hbar\sum_{j=1}^k \tilde{h}'_j \tilde{f}_j+
\sum_{j,\ell=1}^k \tilde{h}_{j\ell}\tilde{f}_\ell\tilde{f}_j$. But $\sum_{j,\ell=1}^k \tilde{h}_{j\ell}\tilde{f}_\ell\tilde{f}_j=
\sum_{j<\ell} \tilde{h}_{j\ell}[\tilde{f}_\ell,\tilde{f}_j]$. The bracket is divisible by $\hbar$.
But $\frac{1}{\hbar}[\tilde{f}_\ell, \tilde{f}_j]$ is still in $\A_{\hbar,>0}$ and so in $\operatorname{Span}_{\A_\hbar}(\tilde{f}_1,\ldots,\tilde{f}_k)$.
This finishes the proof of (1).

Let us proceed to the proof of (2).  Let us show that we can choose a covering of $X^{\alpha(\C^\times)}$ by
$X'^{\alpha(\C^\times)}$, where $X'$ is as in (1). This is easily reduced to
the affine case. Here the existence of such a covering is deduced from the Luna slice theorem applied to a fixed point for  $\alpha$.
In more detail, for a fixed point $x$, we can choose an open affine neighborhood $U$ of $x$ in $X\quo \alpha(\C^\times)$
with an \'{e}tale morphism $U\rightarrow T_xX\quo \alpha(\C^\times)$ such that $\pi^{-1}(U)\cong U\times_{T_xX\quo \alpha(\C^\times)}T_xX$,
where $\pi$ stands for the quotient morphism for the action $\alpha$. The subset $\pi^{-1}(U)$ then obviously satisfies
the requirements in (1).

It is easy to see that  the algebras $\Ca_\alpha(\A^\theta(X'))$ form a presheaf with respect to the covering
$X'^{\alpha(\C^\times)}$ (obviously, if $X',X''$ satisfy our assumptions, then their intersection does).
Since the subsets $X'^{\alpha(\C^\times)}$ form a base of topology on $X^{\alpha(\C^\times)}$, it is enough
to show that they form a sheaf with respect to the covering. This is easily deduced from the two straightforward claims:
\begin{itemize}
\item $\Ca_\alpha(\A^\theta(X'))$ is complete and separated with respect to the filtration (here we use an easy claim
that, being finitely generated, the ideal $\A^\theta(X')\A^{\theta}(X')_{>0}$ is closed).
\item The algebras $\operatorname{gr}\Ca_\alpha(\A^\theta(X'))=\C[X'^{\alpha(\C^\times)}]$ do form a sheaf --
the structure sheaf $\mathcal{O}_{X^{\alpha(\C^\times)}}$.
\end{itemize}
The proof of (2) is now complete.


To prove (3) it is enough to assume that $X$ is affine. Let $\pi$ denote the categorical quotient map
$X\rightarrow X\quo \alpha(\C^\times)$. It is easy to see that, for  every open $(\C^\times)^2$-stable affine subvariety $X'$
that intersects $X^{\alpha(\C^\times)}$ non-trivially, and any point $x\in X'^{\alpha(\C^\times)}$,
there is some $\C^\times$-stable open affine subvariety $Z\subset X\quo \alpha(\C^\times)$ with $x\in \pi^{-1}(Z)\subset X'$. So we can assume, in addition, that all covering
affine subsets $X^i$ are of the form $\pi^{-1}(?)$. Moreover, we can assume that they are all principal
(and so are given by non-vanishing of $\alpha(\C^\times)$-invariant and $\C^\times$-semiinvariant elements of $\A^\theta(X)$).
Then all algebras $\Ca_\alpha(\A^\theta(X'))$ are obtained from $\Ca_\alpha(\A^\theta(X))$ by microlocalization.
Our claim follows from standard properties of  microlocalization.
\end{proof}

\subsection{Comparison between algebra and sheaf levels}\label{SS_Ca_sheaf_vs_alg}
Now let us suppose that $X$ is a conical symplectic resolution of $X_0$. We write $\A^\theta_\lambda$
for the quantization of $X$ corresponding to $\lambda$ and $\A_\lambda$
for its algebra of global sections. By the construction, for any $\lambda\in H^2(X)$,
there is a natural homomorphism $\Ca_\alpha(\A_{\lambda})\rightarrow \Gamma(\Ca_\alpha(\A_\lambda^\theta))$.
Our goal in this section is to prove the following result.

\begin{Prop}\label{Prop:A0_descr}
Suppose that $H^i(X^{\alpha(\C^\times)},\mathcal{O})=0$ for $i>0$.  There is a Zariski open subset subset $Z\subset H^2(X)$ such
that the homomorphism $\Ca_\alpha(\A_{\lambda})\rightarrow \Gamma(\Ca_\alpha(\A^\theta_\lambda))$ is an isomorphism provided $\lambda\in Z$.
\end{Prop}
\begin{proof}
Let $\tilde{X}$ be the universal deformation of $X$ over $H^2(X)$ and $\tilde{X}_0$ be its affinization. Consider
the natural homomorphism $\Ca_\alpha(\C[\tilde{X}_0])\rightarrow \C[\tilde{X}^{\alpha(\C^\times)}]$. It is an isomorphism outside
of $\mathcal{H}_\C$ (the union of singular hyperplanes)
since $\tilde{X}\rightarrow \tilde{X}_0$ is an isomorphism precisely outside that locus.
Now consider the canonical quantization $\tilde{\A}^\theta$ of $\tilde{X}$. Similarly to the previous
section, $\Ca_\alpha(\tilde{\A}^{\theta})$ is a quantization of $\tilde{X}^{\alpha(\C^\times)}$. The cohomology vanishing for
$X^{\alpha(\C^\times)}$ implies that for $\tilde{X}^{\alpha(\C^\times)}$.
It follows that $\gr \Gamma(\Ca_\alpha(\tilde{\A}^\theta))=\C[\tilde{X}^{\alpha(\C^\times)}]$.
Also there is a natural epimorphism $\Ca_\alpha(\C[\tilde{X}_0])\rightarrow \gr\Ca_\alpha(\tilde{\A})$ and a natural homomorphism
$\gr\Ca_\alpha(\tilde{\A})\rightarrow \gr \Gamma(\Ca_\alpha(\tilde{\A}^{\theta}))$. The resulting homomorphism $\gr\Ca_\alpha(\tilde{\A})\rightarrow \gr\Gamma(\Ca_\alpha(\tilde{\A}^{\theta}))$ is, on one hand, the associated graded of the homomorphism $\Ca_\alpha(\tilde{\A})\rightarrow\Gamma(\Ca_\alpha(\tilde{\A}^{\theta}))$ and on the other hand, an isomorphism over the complement of $\mathcal{H}_\C$. We deduce that the supports of the associated graded modules of the kernel and the cokernel of
$\Ca_\alpha(\tilde{\A})\rightarrow\Gamma(\Ca_\alpha(\tilde{\A}^{\theta}))$ are supported on $\mathcal{H}_{\C}$ as $\C[H^2(X)]$-modules.
It follows that the support of the kernel and of the cokernel of $\Ca_\alpha(\tilde{\A})\rightarrow\Gamma(\Ca_\alpha(\tilde{\A}^{\theta}))$
are Zariski closed subvarieties of $H^2(X)$. We note that $\Gamma(\Ca_\alpha(\tilde{\A}^{\theta}))$ is flat over $H^2(X)$
and the specialization at $\lambda$ coincides with $\Gamma(\Ca_\alpha(\A_{\lambda}^{\theta}))$, this is because of the
vanishing assumption on the structure sheaf. So $\Ca_\alpha(\tilde{\A})$
is generically flat over $H^2(X)$, while the specialization at $\lambda$ always coincides with $\Ca_\alpha(\A_\lambda)$.
This implies the claim of the proposition.
\end{proof}

\subsection{Correspondence between parameters}\label{SS_Ca_param}
Our next goal is to understand how to recover the periods of the direct summands $\Ca_\alpha(\A^{\theta})$ from that of $\A^\theta$.
We will assume that $X^{\alpha(\C^\times)}$ satisfies the cohomology vanishing conditions on  the structure sheaf,
but we will not require that of $X$, the period map still makes sense, see \cite{BK}.
Consider the decomposition $X^{\alpha(\C^\times)}=\bigsqcup_i X^0_i$ into connected components. Let $Y_i$ denote the
contracting locus of $X^0_i$ and let $\A_{i}^{\theta 0}$ be the restriction of $\Ca_\alpha(\A^\theta)$ to $X^0_i$.
To determine the period of $\A_i^{\theta 0}$, we will quantize $Y_i$ and then use results from  \cite{BGKP}
on quantizations of line bundles on lagrangian subvarieties.

First of all, let us consider the case when $X$ is affine and so is quantized by a single algebra, $\A$. We will quantize the contracting
locus $Y$ by a single $\A$-$\A^0$-bimodule (where $\A^0$ stands for $\Ca_\alpha(\A)$), this bimodule is $\A/\A\A_{>0}$.

\begin{Lem}\label{Lem:repel_quant_affine}
Under the above assumptions, the associated graded of $\A/\A\A_{>0}$ is $\C[Y]$.
\end{Lem}
\begin{proof}
This was established in the proof of Proposition \ref{Prop:A0}. More precisely, the case when $Y$ is a complete intersection
given by $\alpha(\C^\times)$-semiinvariant elements of positive weight follows from the proof of assertion (1), while the general case follows similarly to the proof of (3).
\end{proof}

Now let us consider the non-affine case. Let us cover $X\setminus \bigcup_{k\neq i}X^0_k$ with $(\C^\times)^2$-stable open affine subsets $X^j$. We may assume that $X^j$ either does not intersect $Y_i$ or its intersection with $Y_i$ is of the form $\pi_i^{-1}(X^j\cap X_i^0)$, where $\pi_i:Y_i\rightarrow X_i^0$
is the projection. For this we first choose some covering by $(\C^\times)^2$-stable open affine subsets.
Then we delete $Y_i\setminus \pi_i^{-1}(X^j\cap X_i^0)$ from each $X^j$, we still have a covering. We cover the remainder of each $X^j$
by subsets that are preimages of open affine subsets on $X^j\quo \alpha(\C^\times)$, it is easy to see that this covering has required
properties. Let us replace $X$ with the union of $X^j$ that intersect $Y_i$.

After this replacement, we can quantize $Y_i$ by a $\A^\theta$-$\Ca_\alpha(\A^\theta)$-bimodule. We have natural $\A^\theta(X^j)$-$\Ca_\alpha(\A^\theta)(X^j\cap X_i^0)$-bimodule structures on $\A^\theta(X^j)/\A^\theta(X^j)\A^\theta(X^j)_{>0}$
and glue the bimodules corresponding to different $j$ together
along the intersections $X^i\cap X^j$ (we have homomorphisms $\A^\theta(X^i)\rightarrow \A^\theta(X^i\cap X^j)$
that give rise to $\A^\theta(X^i)/\A^\theta(X^i)\A^\theta(X^i)_{>0}\rightarrow \A^\theta(X^i\cap X^j)/\A^\theta(X^i\cap X^j)\A^\theta(X^i\cap X^j)_{>0}$ and to $\Ca_\alpha(\A^\theta(X^i))\rightarrow \Ca_\alpha(\A^\theta(X^i\cap X^j))$). Similarly to the proof of (2) in Proposition \ref{Prop:A0}, we get a sheaf of $\A^\theta$-$\Ca_\alpha(\A^{\theta})$-bimodules on $Y_i$ that we denote by $\A^\theta/\A^\theta \A^\theta_{>0}$. The following is a direct consequence of  the construction.

\begin{Lem}\label{Lem:repel_gener}
The associated graded of $\A^\theta/\A^{\theta}\A_{>0}^\theta$  coincides with the
$\mathcal{O}_X$-$\mathcal{O}_{X^0_i}$-bimodule $\mathcal{O}_{Y_i}$.
\end{Lem}

Now we want to realize $Y_i$ a bit differently (we still use $X$ as in the paragraph preceding Lemma \ref{Lem:repel_gener}, and so can
write $Y$ instead of $Y_i$ and $X^0$ instead of $X^0_i$).  Namely, let $\iota$ denote the inclusion $Y\hookrightarrow X$
and $\pi$ be the projection $Y\rightarrow X^0$. We embed $Y$ into $X\times X^0$ via $(\iota,\pi)$. We equip $X\times X^0$
with the symplectic form $(\omega,-\omega^0)$, where $\omega^0$ is the restriction of $\omega$ to $X^0$.
With respect to this symplectic form $Y$ is a lagrangian subvariety. Further, $\A^\theta\widehat{\otimes}\Ca_\alpha(\A^\theta)^{opp}$
is a quantization of $X\times X^0$ with period $(\lambda,-\lambda^0)$, where $\lambda,\lambda^0$ are periods of
$\A^\theta, \Ca_\alpha(\A^\theta)$.

\begin{Prop}\label{Prop:shift}
The period  $\lambda^0$ coincides with  $\iota^{0*}(\lambda+c_1(K_Y)/2)\in H^2(X^0)=H^2(Y)$,
where $K_Y$ denotes the canonical class of $Y$ and $\iota^0$ is the inclusion $X^0\hookrightarrow X$.
\end{Prop}
\begin{proof}
The period of $\A^\theta\widehat{\otimes}\Ca_\alpha(\A^\theta)^{opp}$ coincides with $p_1(\lambda)-p_2(\lambda^0)$,
where $p_1:X\times X^0\rightarrow X, p_2:X\times X^0\rightarrow X^0$ are the projections.
So the pull-back of the period to $Y$ is $\iota^*(\lambda)-\pi^*(\lambda^0)$. The structure sheaf of $Y$
admits a quantization to a $\A^\theta\widehat{\otimes}\Ca_{\alpha}(\A^\theta)^{opp}$-bimodule,
By \cite[(1.1.3),Theorem 1.1.4]{BGKP}, we have $\iota^*(\lambda)-\pi^*(\lambda^0)=-\frac{1}{2}c_1(K_Y)$.
Restricting this equality to $X^0$, we get the equality required in the proposition.
%
\end{proof}

\subsection{Gieseker case}\label{SS_Ca_Gies}
Now we want to apply Proposition \ref{Prop:shift} to the case when $X=\M^\theta(n,r)$ and $\alpha$ comes from a generic
one-dimensional torus in $\GL(w)$ given by $t\mapsto (t^{d_1},\ldots,t^{d_r})$ with $d_1\gg d_2\gg\ldots\gg d_r$. Recall
that the fixed point components are parameterized by partitions of $n$. Let $X^0_\mu$ denote the component corresponding
to a partition $\mu$ of $n$ and let $Y_\mu$ be its contracting locus. So $Y_\mu\rightarrow X^0_\mu$ is a vector bundle.
We will need to describe this vector bundle. The description is a slight ramification of \cite[Proposition 3.13]{Nakajima_tensor}.

First, consider the following situation. Set $V:=\C^n, W=\C^r$. Choose a decomposition $W=W^1\oplus W^2$
with $\dim W^i=r_i$
and consider the one-dimensional torus in $\GL(w)$ acting trivially on $W^2$ and by $t\mapsto t$ on $W^1$.
The components of the  fixed points in $\M^\theta(n,r)$ are in one-to-one correspondence with decompositions
on $n$ into the sum of two parts. Pick such a decomposition $n=n_1+n_2$ and consider the splitting $V=V^1\oplus V^2$
into the sum of two spaces of the corresponding dimensions and let $X^0_1=\M^\theta(n_1,r_1)\times \M^\theta(n_2,r_2)\subset \M^\theta(n,r)^{\alpha(\C^\times)}$
be the corresponding component. We assume that $\theta>0$.

Nakajima has described the contracting bundle $Y_1\rightarrow X^0_1$. This is the bundle on $X^0_1=\M^\theta(n_1,r_1)\times
\M^\theta(n_2,r_2)$ that is induced from the $\GL(n_1)\times \GL(n_2)$-module $\ker \beta^{12}/\operatorname{im}\alpha^{12}$,
where $\alpha^{12},\beta^{12}$ are certain $\GL(n_1)\times \GL(n_2)$-equivariant linear maps
$$\Hom(V^2,V^1)\xrightarrow{\alpha^{12}}\Hom(V^2,V^1)^{\oplus 2}\oplus \operatorname{Hom}(W^2,V^1)\oplus \Hom(V^2,W^1)\xrightarrow{\beta^{12}}\Hom(V^2,V^1)
$$
We do not need to know the precise form of the maps $\alpha^{12},\beta^{12}$, what we need is that $\alpha^{12}$ is injective
while $\beta^{12}$ is surjective. So $\ker \beta^{12}/\operatorname{im}\alpha^{12}\cong \operatorname{Hom}(W^2,V^1)\oplus \Hom(V^2,W^1)$,
an isomorphism of $\GL(n_1)\times \GL(n_2)$-modules.

It is easy to see that if $\alpha':\C^\times \rightarrow \GL(r_2)$ is a homomorphism of the form $t\mapsto \operatorname{diag}(t^{d_1},\ldots,t^{d_k})$ with $d_1,\ldots,d_k\gg 0$, then the contracting bundle
for the one-parametric subgroup $(\alpha',1):\C^\times \rightarrow \GL(W^1)\times \GL(W^2)$
coincides with the sum of the contracting bundles for $\alpha'$ and for $(t,1)$. So we get the following result.

\begin{Lem}\label{Lem:contr_Gies}
Consider $\alpha:\C^\times\rightarrow \GL(r)$ of the form $t\mapsto \operatorname{diag}(t^{d_1},\ldots,t^{d_r})$ with
$d_1\gg d_2\gg \ldots\gg d_r$. Consider the irreducible component of $\M^\theta(n,r)^{\alpha(\C^\times)}$
corresponding to the decomposition $n=n_1+\ldots+n_r$. Then its contracting bundle is induced from the
following $\prod_{i=1}^r \GL(n_i)$-module: $\sum_{i=1}^r \left((\C^{n_i})^{\oplus r-i}\oplus (\C^{n_i*})^{\oplus i-1}\right)$.
\end{Lem}

For $\A_\lambda^\theta(n_1,\ldots,n_r)$ denote the summand of $\Ca_\alpha(\A^\theta_\lambda(n,r))$ corresponding to
the decomposition $n=n_1+\ldots+n_r$. Let us recall that the value of the period for $\A_\lambda^\theta(n,r)$
is $\lambda+\frac{r}{2}$. Using Lemma \ref{Lem:contr_Gies} and Proposition \ref{Prop:shift}, we deduce the following claim.

\begin{Cor}\label{Cor:shift_Gies}
We have $\A_\lambda^\theta(n_1,\ldots,n_r)=\bigotimes_{i=1}^r \A^\theta_{\lambda+(i-1)}(n_i,1)$.
\end{Cor}

\subsection{Parabolic induction}\label{SS_parab_induc}
Let $X$ be a conical symplectic resolution of $X_0$. We assume that $X$ comes with a Hamiltonian action of a torus
$T$ such that $X^T$ is finite. Let $\mathfrak{C}$ stand for $\operatorname{Hom}(\C^\times,T)$. We introduce
a  pre-order $\prec^\lambda$ on $\mathfrak{C}$ as follows: $\alpha\prec^\lambda \alpha'$ if
$\A_\lambda \A_{\lambda,>0,\alpha}\subset \A_\lambda\A_{\lambda,>0,\alpha'}$.
This gives an equivalence relation $\sim^\lambda$ on $\mathfrak{C}$. Both extend
naturally to $\mathfrak{C}_{\mathbb{Q}}:=\mathbb{Q}\otimes_{\Z}\mathfrak{C}$.

The following lemma explains why this ordering is important.

\begin{Lem}\label{Lem:parab_ind}
Suppose $\alpha\prec\alpha'$. Then $\Ca_{\alpha'}(\Ca_\alpha(\A_\lambda))=\Ca_{\alpha'}(\A_\lambda)$. Further, let
$\Delta_{\alpha'}: \Ca_{\alpha'}(\A_\lambda)\operatorname{-mod}\rightarrow \A_\lambda\operatorname{-mod}, \Delta_{\alpha}:
\Ca_\alpha(\A_\lambda)\operatorname{-mod}\rightarrow \A_\lambda\operatorname{-mod}, \underline{\Delta}:\Ca_{\alpha'}(\A)\operatorname{-mod}
\rightarrow \Ca_\alpha(\A_\lambda)\operatorname{-mod}$ be the Verma module functors. We have $\Delta_{\alpha'}=\Delta_\alpha\circ\underline{\Delta}$.
\end{Lem}
The proof is straightforward.

The lemma shows that the Verma module functor can be studied in stages. This is what we mean by the parabolic induction.

Our goal now is to describe the pre-order $\prec^\lambda$ for $\lambda$ Zariski generic. We say that $\alpha\prec\alpha'$
if, for each $x\in X^T$, we have $T_xX_{>0,\alpha}\subset T_xX_{>0,\alpha'}$. This automatically implies
$T_xX_{\geqslant 0,\alpha}\supset T_xX_{\geqslant 0,\alpha'}$ (via taking the skew-orthogonal complement)
and $T_xX_{<0,\alpha}\subset T_xX_{<0,\alpha'}$.

\begin{Prop}\label{Prop:orders}
Fix $\alpha,\alpha'$. For $\lambda$ in a Zariski open subset, $\alpha\prec^\lambda\alpha'$ is equivalent $\alpha\prec\alpha'$.
\end{Prop}
\begin{proof}
The proof is in  several steps. Suppose $\alpha\prec\alpha'$ and let us check that $\alpha\prec^\lambda\alpha'$.

{\it Step 1}.
We need to check that, for a Zariski generic $\lambda$, we have $\A_{\lambda,>0,\alpha}\subset \A_\lambda\A_{\lambda,>0,\alpha'}$
or, equivalently, $\alpha$ has no positive weights on $\A_\lambda/\A_{\lambda}\A_{\lambda,>0,\alpha'}$. This will follow if
we check that the $\tilde{\A}$-submodule in $\tilde{\A}/\tilde{\A}\tilde{\A}_{>0,\alpha'}$ generated by the elements of positive
weight for $\alpha$ is torsion over $\C[H^2(X)]$ (here, as usual, $\tilde{\A}$ stands for the algebra of global sections
of the canonical quantization $\tilde{\A}^\theta$ of $\tilde{X}$). This, in turn, will follow if we prove an analogous
statement for $\gr\tilde{\A}/\tilde{\A}\tilde{\A}_{>0,\alpha'}$.

{\it Step 2}.
We have an epimorphism $\C[\tilde{X}]/\C[\tilde{X}]\C[\tilde{X}]_{>0,\alpha'}\twoheadrightarrow \gr\tilde{\A}/\tilde{\A}\tilde{\A}_{>0,\alpha'}$. We claim that its kernel is again torsion over
$\C[H^2(X)]$, in fact, it is supported on $\mathcal{H}_{\C}$. Consider the $\hbar$-adic completion
$\tilde{\A}_\hbar$ of $R_\hbar(\A)$. Let $\tilde{\A}_\hbar^{reg}$ denote the (completed) localization
of $\tilde{\A}_\hbar$ to $H^2(X)\setminus \mathcal{H}_{\C}$. Then $\tilde{\A}_{\hbar}^{reg}/\tilde{\A}_\hbar^{reg}\tilde{\A}_{\hbar,>0,\alpha'}^{reg}$ coincides
with the  localization of $\tilde{\A}_{\hbar}/\tilde{\A}_\hbar\tilde{\A}_{\hbar,>0,\alpha'}$.
On the other hand, over $H^2(X)\setminus \mathcal{H}_{\C}$, the ideal $\C[\tilde{X}]\C[\tilde{X}]_{>0,\alpha'}$
is a locally complete intersection (given by elements of positive $\alpha'$-weight), compare to the proof
of (2) in Proposition \ref{Prop:A0}. As in the proof of (1) of Proposition \ref{Prop:A0}, this implies that
$\tilde{\A}_{\hbar}^{reg}/\tilde{\A}_\hbar^{reg}\tilde{\A}_{\hbar,>0,\alpha'}^{reg}$ is flat over $\C[\hbar]$.
So the $\hbar$-torsion in  $\tilde{\A}_{\hbar}/\tilde{\A}_\hbar\tilde{\A}_{\hbar,>0,\alpha'}$ is supported
on $\mathcal{H}_{\C}$. This implies the claim in the beginning of this step.

{\it Step 3}.  So we need to check, that under the assumption $\alpha\prec\alpha'$, the submodule
in $\C[\tilde{X}]/\C[\tilde{X}]\C[\tilde{X}]_{>0,\alpha'}$ generated by the elements of positive $\alpha$-weight
is supported on $\mathcal{H}_{\C}$. This is equivalent to the claim that $\C[X_z]_{>0,\alpha}\subset
\C[X_z]\C[X_z]_{>0,\alpha'}$ for $z\not\in \mathcal{H}_{\C}$. Here we write $X_z$ for the fiber of
$\tilde{X}\rightarrow H^2(X)$ over $z$. Note that we still have $(T_xX_z)_{>0,\alpha}\subset (T_xX_z)_{>0,\alpha'}$
for all $x\in X_z^{\alpha'(\C^\times)}$. The inclusion $\C[X_z]_{>0,\alpha}\subset \C[X_z]\C[X_z]_{>0,\alpha'}$ now follows from the Luna slice theorem (for $\alpha(\C^\times)\alpha'(\C^\times)$ applied to $\alpha'(\C^\times)$-fixed points, we would like to
point out that such points are automatically $\alpha(\C^\times)$-fixed).

The proof of $\alpha\prec\alpha'\Rightarrow \alpha\prec^\lambda\alpha'$ is now complete. We can reverse the argument
to see that if $\alpha\prec^\lambda \alpha'$ for Zariski generic $\lambda$, then $\alpha\prec\alpha'$.
\end{proof}

The equivalence classes for $\prec$ are cones in $\mathfrak{C}_{\mathbb{Q}}$ and the pre-order is by inclusion
of the closures. In particular, there are finitely many equivalence classes. So there is a Zariski open subset
where $\prec^\lambda$ refines $\prec$.

Sometimes we will need to determine when $\alpha\prec^\lambda \alpha'$ for a fixed (non Zariski generic) $\lambda$.
Pick one-parameter subgroups $\alpha,\beta:\C^\times\rightarrow T$.

\begin{Lem}\label{Lem:prec_spec}
For $m\gg 0$, we have $\alpha\prec^\lambda m\alpha+\beta$ for all $\lambda$.
\end{Lem}
\begin{proof}
Clearly, $\alpha\sim^\lambda m\alpha$ for all $m$. The algebra $\gr\A_{\geqslant 0,\alpha}=\C[X_0]_{\geqslant 0,\alpha}$
is finitely generated, as in the proof of \cite[Lemma 3.1.2]{GL}. So we can choose finitely many $T$-semiinvariant generators of the ideal $\C[X_0]_{>0,\alpha}$ in $\C[X_0]_{\geqslant 0,\alpha}$, say $f_1,\ldots,f_k$.
Let $\tilde{f}_1,\ldots,\tilde{f}_k$ denote their lifts to $T$-semiinvariant elements in $\A:=\A_\lambda$,
these lifts are generators of the ideal $\A_{>0,\alpha}$ in $\A_{\geqslant 0,\alpha}$.
Let $a_1,\ldots,a_k>0$ be their weights for $\alpha$ and $b_1,\ldots,b_k$ be their weights for $\beta$.
Take $m\in \Z_{>0}$ such that $ma_i+b_i>0$ for all $i$. The elements $\tilde{f}_1,\ldots,\tilde{f}_k$ then
lie in $\A_{>0,m\alpha+\beta}$ and so $\A\A_{>0,\alpha}\subset \A\A_{>0,m\alpha+\beta}$.
\end{proof}

\section{Finite dimensional representations in the Gieseker case}\label{S_fin_dim}
In this section we will prove (1) of Theorem \ref{Thm:fin dim} and Theorem \ref{Thm:cat_O_easy}.
First, we prove that the homological duality realizes the Ringel duality of highest weight
categories, Subsection \ref{SS_Hom_vs_Ring}.

Then, in Subsection \ref{SS_fin_dim_proof}, we prove part (1) of Theorem \ref{Thm:fin dim}.
The ideas of the proof are as follows: we use the Cartan construction to show that we cannot have finite dimensional
representations when the denominator is different from $n$ and also that, in the denominator $n$ case,
the category $\mathcal{O}$ is not semisimple. Thanks to Subsection \ref{SS_Hom_vs_Ring}, this means
that there is a module with support of dimension $<\frac{1}{2}\dim X$ in $\mathcal{O}$ (for the algebra
$\bar{\A}_\lambda(n,r)$ with Zariski generic $\lambda$). Using the restriction
functors, we see that this module is finite dimensional. Proposition \ref{Prop:CC_inject} then implies
that there is a unique finite dimensional module.

Finally, in Subsection \ref{SS_cat_O_thm} we prove Theorem \ref{Thm:cat_O_easy}. The main idea is to recover
the category from the homological shifts produced by the Ringel duality.

\subsection{Homological duality vs Ringel duality}\label{SS_Hom_vs_Ring}
We start by proving that the homological duality functor $D$ realizes the contravariant Ringel duality
on categories $\mathcal{O}$.

Here we deal with the case when $X\rightarrow X_0$ is a conical symplectic resolution (satisfying the additional
assumption from Subsection \ref{SS_HC_bimod}). We assume that $X$ comes
equipped with a Hamiltonian $\C^\times$-action $\alpha$ that has finitely many fixed points. We choose a period $\hat{\lambda}$
such that
\begin{itemize}
\item[(i)] $\Ca_\alpha(\pm \hat{\lambda})\cong \C[X^{\alpha(\C^\times)}]$
the categories $\OCat(\A_{\hat{\lambda}}), \OCat(\A_{-\hat{\lambda}})$ are highest weight
with standard objects being Verma modules.
\item[(ii)] $D^b(\OCat(\A_{\hat{\lambda}}))\xrightarrow {\sim} D^b_{\OCat}(\A_{\hat{\lambda}}),
D^b(\OCat(\A_{-\hat{\lambda}}))\xrightarrow {\sim} D^b_{\OCat}(\A_{-\hat{\lambda}})$.
\end{itemize}
We recall that these two conditions hold for a Zariski generic $\hat{\lambda}$.

Let us recall the definition of the (contravariant) Ringel duality. Let $\mathcal{C}_1,\mathcal{C}_2$ be two highest weight categories.
Suppose we have a contravariant equivalence $R: \mathcal{C}_1^\Delta\xrightarrow{\sim}\mathcal{C}_2^\Delta$
(the superscript $\Delta$ means the full subcategories of standardly filtered objects).
Then  it restricts to a contravariant duality between $\mathcal{C}_1\operatorname{-proj}$  and $\mathcal{C}_2\operatorname{-tilt}$.
The former denotes the category of the projective objects in $\mathcal{C}_1$, while the latter is the category
of tilting objects in $\mathcal{C}_2$, i.e., objects that are both standardly and costandardly filtered.
The  equivalence $R$ extends to an equivalence $D^b(\mathcal{C}_1)\xrightarrow{\sim} D^b(\mathcal{C}_2)^{opp}$.
Moreover, the category $\mathcal{C}_2$ gets identified with $\operatorname{End}(T)\operatorname{-mod}$ and, under this identification,
the derived equivalence above is $\operatorname{RHom}_{\mathcal{C}_1}(\bullet,T)$. Here $T$ is the tilting generator of $\mathcal{C}_1$,
i.e., the direct sum of all indecomposable tiltings.
For the proofs of the claims above in this paragraph see \cite[Proposition 4.2]{GGOR}.

We say that $\mathcal{C}_2$ is a Ringel dual of $\mathcal{C}_1$ and write $\mathcal{C}_1^\vee$ for $\mathcal{C}_2$.

\begin{Prop}\label{Prop:contr_Ringel}
Take  $\hat{\lambda}$ in a Zariski open set and such that the abelian localization holds for $(\hat{\lambda},\theta),
(-\hat{\lambda},-\theta)$. Then there is an equivalence $\mathcal{O}(\A_{-\hat{\lambda}})\xrightarrow{\sim} \OCat(\A_{\hat{\lambda}})^\vee$ that intertwines the homological duality functor $D:D^b(\mathcal{O}(\A_{\hat{\lambda}}))\rightarrow
D^b(\mathcal{O}(\A_{-\hat{\lambda}}))^{opp}$ and the contravariant Ringel duality functor $\operatorname{RHom}_{\OCat(\A_{\hat{\lambda}})}(\bullet,T): D^b(\OCat(\A_{\hat{\lambda}}))\rightarrow D^b(\OCat(\A_{\hat{\lambda}})^\vee)^{opp}$.
\end{Prop}

Let $\Delta_{\hat{\lambda}}$ denote the sum of all standard objects in $\OCat(\A_{\hat{\lambda}})$.
Of course, $\Delta_{\hat{\lambda}}=\A_{\hat{\lambda}}/\A_{\hat{\lambda}}\A_{\hat{\lambda},>0}$.

We write $\theta$ for an element of the ample cone of $X$.

\begin{Lem}\label{Lem:dual_hom_vanish}
For a parameter $\hat{\lambda}$ in a Zariski open subset,
the object $D(\Delta_{\hat{\lambda}}(p))$ is concentrated in homological degree $0$
and, moreover, its characteristic cycle (an element of the vector space with basis formed by the irreducible
components of the contracting variety $Y$) coincides with the class of (the degeneration of) the contracting component of $p$
at a generic fiber of $\tilde{X}\rightarrow H^2(X)$.
\end{Lem}
\begin{proof}
Let us prove the first claim. What we need to prove is that $\operatorname{Ext}^i(\Delta_{\hat{\lambda}},\A_{\hat{\lambda}})=0$
provided $i\neq \frac{1}{2}\dim X$ for $\hat{\lambda}$ in a Zariski open space.
Our claim will follow follow if we show that the support of $\operatorname{Ext}^i(\tilde{\Delta}, \widetilde{\A})$ in $H^2(X)$
is not dense in $H^2(X)$ and that the $\C[H^2(X)]$-module $\operatorname{Ext}^i(\tilde{\Delta}, \widetilde{\A})$ is generically flat.
Here we write $\widetilde{\Delta}=\widetilde{\A}/\widetilde{\A}\widetilde{\A}_{>0}$.

We can take a graded free resolution of $\gr\widetilde{\Delta}$ and lift it to a free resolution of $\widetilde{\Delta}$.
It follows that the right $\widetilde{\A}$-modules
$\operatorname{Ext}^i(\widetilde{\Delta}, \widetilde{\A})$ are naturally filtered and that the associated
graded modules are subquotients of  $\operatorname{Ext}^i(\gr\widetilde{\Delta}, \C[\widetilde{X}])$.
The claim about generic flatness follows (compare with \cite[Lemma 5.5, Corollary 5.6]{BL}). Also to prove
that claim in the previous paragraph that the support is  not dense
it is enough to prove a similar claim for  $\operatorname{Ext}^i(\gr\widetilde{\Delta}, \C[\widetilde{X}])$.

Set $\widetilde{\Delta}_{cl}:=\C[\widetilde{X}]/\C[\widetilde{X}]\C[\widetilde{X}]_{>0}$.
We have $\widetilde{\Delta}_{cl}\twoheadrightarrow \gr\widetilde{\Delta}$. Moreover, the support of the kernel in $H^2(X)$
is contained in $\mathcal{H}_{\C}$, see Step 2 of the proof of Proposition \ref{Prop:orders}. So it is enough to
show that the support of $\operatorname{Ext}^i(\widetilde{\Delta}_{cl}, \C[\widetilde{X}])$
is not dense when $i\neq \frac{1}{2}\dim X$. This follows  from the observation  that, generically over $H^2(X)$, the ideal  $\C[\widetilde{X}]\C[\widetilde{X}]_{>0}$ is a locally complete intersection in a smooth variety.

The argument above also implies that the associated graded of $D(\Delta_{\hat{\lambda}}(p))$ coincides with
that of $\operatorname{Ext}^{\frac{1}{2}\dim X}(\Delta_{cl,\underline{\lambda}}(p),\C[X_{\underline{\lambda}}])$
for a Zariski generic element $\underline{\lambda}\in H^2(X)$. The latter is just the class of the contracting component
$Y_{\underline{\lambda},p}$ (defined as the sum of components of $X\cap \overline{\C^\times Y_{\underline{\lambda},p}}$
with obvious multiplicities).
%
\end{proof}

\begin{proof}[Proof of Proposition \ref{Prop:contr_Ringel}]
We write $\Delta_{\hat{\lambda}}(p)^\vee$ for $D(\Delta_{\hat{\lambda}}(p))$, thanks to Lemma \ref{Lem:dual_hom_vanish}, this is an object in
$\OCat(\A_{-\hat{\lambda}})$ (and not just a complex in its derived category).  We have $\operatorname{End}(\Delta_{\hat{\lambda}}(p)^\vee)=\C$
and $\Ext^i(\Delta_{\hat{\lambda}}(p)^\vee, \Delta_{\hat{{\lambda}}}(p')^\vee)=0$ if $i>0$ or $p\leqslant^\theta p'$.
We remark that the orders $\leqslant^\theta$ and $\leqslant^{-\theta}$ can be refined
to opposite partial orders (first we refine them to the orders coming by the values of the real moment
maps for the actions of $\mathbb{S}^1\subset \alpha(\C^\times)$, and then refine those),
compare with \cite[5.4]{Gordon}. So it only
remains to prove that the characteristic cycle of $\Delta_{\hat{\lambda}}(p)^\vee$ consists
of the contracting components $Y_{p'}$ with $p'\leqslant^{-\theta}p$.
The characteristic cycle of $\Delta_{\hat{\lambda}}(p)^\vee$ coincides with $\overline{\C^\times Y_{\underline{\lambda},p}}\cap X$,
by Lemma \ref{Lem:dual_hom_vanish}. But the
characteristic cycle of $\Delta_{-\hat{\lambda}}(p)$ is the same. Our claim follows.
\end{proof}

\begin{Rem}
We also have covariant Ringel duality given by $\operatorname{RHom}(T,\bullet)$, it maps costandard objects
to standard ones. Under the assumption that the conical symplectic resolutions of $X_0$ are strictly semismall,
Propositions \ref{Prop:contr_Ringel} and \ref{Prop:WC_vs_D} imply that the long wall-crossing functor
is inverse of the covariant Ringel duality. This proves a part of  \cite[Conjecture 8.27]{BLPW}.
\end{Rem}

\subsection{Proof of Theorem \ref{Thm:fin dim}}\label{SS_fin_dim_proof}
Here we prove (1) of Theorem \ref{Thm:fin dim}. The proof is in several steps.

{\it Step 1}. Let us establish a criterium for the semisimplicity of a highest weight category via the Ringel duality.

\begin{Lem}\label{Lem:hw_techn}
Let $\mathcal{C}$ be a highest weight category and $R:D^b(\mathcal{C})\rightarrow D^b(\mathcal{C}^{\vee})^{opp}$
denote the contravariant Ringel duality. The following conditions are equivalent:
\begin{enumerate}
\item $\mathcal{C}$ is semisimple.
\item We have $H^0(R(L))\neq 0$ for every simple object $L$.
\item every simple  lies in the socle of a standard object.
\end{enumerate}
\end{Lem}
\begin{proof}
The implication (1)$\Rightarrow$(2) is clear. The implication (2)$\Rightarrow$(3) follows from the fact that every standard
object in a highest weight category is included into an indecomposable tilting.

Let us prove (3)$\Rightarrow$(1). Let $\lambda$ be a maximal (with respect to the coarsest highest weight ordering) label. Then the simple $L(\lambda)$ lies in the socle of some standard, say $\Delta(\mu)$. But all simple constituents of $\Delta(\mu)$ are $L(\nu)$ with $\nu\leqslant \mu$. It follows that $\mu=\lambda$. Since $L(\lambda)$ lies in the socle of $\Delta(\lambda)$ and also coincides with the head, we see that $\Delta(\lambda)=L(\lambda)$. So $L(\lambda)$ is projective and therefore spans
a block in the category. Since this holds for any maximal $\lambda$, we deduce that the category $\mathcal{C}$ is semisimple.
\end{proof}

Let us remark that for the category $\mathcal{O}(\bar{\A}_\lambda(n,r))$ condition (2) is equivalent to every simple
having support of dimension $rn-1$. This follows from Subsection \ref{SS_dual_WC} and Proposition \ref{Prop:contr_Ringel}.

Below in this proof we assume that $\lambda$ is chosen as in Proposition \ref{Prop:contr_Ringel}, in particular,
the categories $\OCat(\bar{\A}_\lambda(n,r))$ and $\OCat(\bar{\A}^\theta_\lambda(n,r))$ are equivalent. In the definition of categories
$\mathcal{O}$ we choose the torus of the form $(\alpha,1)$, where $\alpha:\C^\times\rightarrow \GL(r)$
is given by $t\mapsto (t^{d_1},\ldots,t^{d_r})$ with $d_1\gg d_2\gg\ldots\gg d_r$.

{\it Step 2}. Let us prove that the category $\mathcal{O}(\bar{\A}^\theta_{\lambda}(n,r))$ is semisimple, when $\lambda\not\in \mathbb{Q}$ or
the denominator of $\lambda$ is bigger than $n$. The proof is by induction on $n$ (for $n=0$ the claim is vacuous).

By Corollary \ref{Cor:full_supp}, we see that a simple $\bar{\A}_\lambda(n,r)$-module whose support has dimension
$<rn-1$ is annihilated by a proper ideal of $\bar{\A}_\lambda(n,r)$. We claim that any such ideal has finite codimension
under our assumption on $\lambda$. Indeed, otherwise some proper slice algebra has an ideal of finite codimension, see Proposition \ref{Prop:dagger_prop}, which contradicts our inductive assumption. So the support of  a simple has dimension either $rn-1$
or $0$.


If we know that all simple modules have support of dimension $rn-1$, we are done.
But thanks to Corollary \ref{Cor:shift_Gies}, Proposition \ref{Prop:A0_descr} and known results on finite dimensional
$\bar{\A}_\lambda(n,1)$-modules, \cite{BEG}, we see that $\Ca_\alpha(\bar{\A}_\lambda(n,r))$ has no finite dimensional modules
(we obviously have $\Ca_\alpha(\A_\lambda(n,r))=D(\C)\otimes \Ca_\alpha(\bar{\A}_\lambda(n,r))$ and none of the summands of
$\A_\lambda(n,r)^0$ has simple of GK dimension $1$ in category $\mathcal{O}$).

{\it Step 3}. The description of $\Ca_\alpha(\bar{\A}_\lambda(n,r))$ shows that there are no finite dimensional $\bar{\A}_\lambda(n,r)$-modules in the case when the denominator of $\lambda$ is less than $n$.

{\it Step 4}. Now consider the case of denominator $n$. Similarly to Step 2, all simples are either finite dimensional or have support
of dimension $rn-1$. By Lemma \ref{Lem:top_cohom}, the dimension of the middle homology of $\bar{\M}^\theta(n,r)$ is $1$.
Thanks to Proposition \ref{Prop:CC_inject}, the number of finite dimensional irreducibles is 0 or $1$. If there is one such module,
then the category of finite dimensional modules is semisimple because $\mathcal{O}(\bar{\A}_\lambda(n,r))$ is a highest weight category.
Thanks to Step 1, we only need to show that   $\mathcal{O}(\bar{\A}_\lambda(n,r))$ is not semisimple.

One-parameter subgroups $\alpha:t\mapsto \operatorname{diag}(t^{d_1},\ldots, t^{d_r})$ with $d_1\gg\ldots \gg d_r$ form one equivalence
class for the pre-order $\prec$. This cone is a face of the equivalence class containing $(\alpha,1)$. Proposition \ref{Prop:orders}
implies that $\alpha\prec^\lambda (\alpha,1)$. Now we can use Lemma \ref{Lem:parab_ind}.
%


Let us write $\Delta^0,\underline{\Delta}$ for the Verma module functors $\Delta^0:\C[\mathsf{P}_r(n)]\rightarrow \mathcal{O}(\Ca_\alpha(\A_\lambda(n,r)))$
and $\underline{\Delta}:\mathcal{O}(\Ca_\alpha(\A_\lambda(n,r)))\rightarrow \mathcal{O}(\A_\lambda(n,r))$,
here we write $\mathsf{P}_r(n)$ for the set of the $r$-multipartitions of $n$.
By Lemma \ref{Lem:parab_ind}, we have $\Delta=\underline{\Delta}\circ \Delta^0$.
The  category $\mathcal{O}(\Ca_\alpha(\A_\lambda(n,r)))$ is not semisimple:
there is a nonzero homomorphism $\varphi:\Delta^0(p_2)\rightarrow \Delta^0(p_1)$, where $p_1=(\varnothing^{r-1}, (n)),
p_2=(\varnothing^{r-1}, (n-1,1))$. So we get a homomorphism $\underline{\Delta}(\varphi):\Delta(p_2)=\underline{\Delta}(\Delta^0(p_2))
\rightarrow \underline{\Delta}(\Delta^0(p_1))=\Delta(p_1)$. The highest $\alpha$-weight components
of $\Delta(p_2),\Delta(p_1)$ coincide with $\Delta^0(p_2),
\Delta^0(p_1)$, respectively, by the construction. The homomorphism $\Delta^0(p_2)\rightarrow \Delta^0(p_1)$ induced
by $\underline{\Delta}(\varphi)$ coincides with $\varphi$. It follows that $\underline{\Delta}(\varphi)\neq 0$. We conclude that $\mathcal{O}(\A_\lambda(n,r))$ is not semisimple.

This completes the proof of all claims of the theorem but the claim that the category of modules supported on $\rho^{-1}(0)$
is semisimple. The latter is an easy consequence of the observation that, in a highest weight category, we have
$\operatorname{Ext}^1(L,L)=0$. We would like to point out that the argument of the previous paragraph generalizes
to the denominators less than $n$. So in those cases there are also simple $\bar{\A}_\lambda(n,r)$-modules of support with
dimension $<rn-1$.

\subsection{Proof of Theorem \ref{Thm:cat_O_easy}}\label{SS_cat_O_thm}
In this subsection we will prove Theorem \ref{Thm:cat_O_easy}. We have already seen in the previous subsection that if the denominator
is bigger than $n$, then the category $\mathcal{O}$ is semisimple. The case of denominator $n$ will follow from a more precise statement,
Theorem \ref{Thm:catO_str}.

Let us introduce a certain model category. Let $\Cat_n$ denote the nontrivial block for the category $\mathcal{O}$ for the Rational Cherednik algebra $\mathcal{H}_{1/n}(n)$
for the symmetric group $\mathfrak{S}_n$. Let us summarize some properties of this category.
\begin{itemize}
\item[(i)] Its coarsest highest weight poset is linearly ordered: $p_n<p_{n-1}<\ldots<p_1$.
\item[(ii)] The  objects $I(p_i)$ for $i>1$ are universal extensions $0\rightarrow \nabla(p_{i})\rightarrow
I(p_i)\rightarrow \nabla(p_{i-1})\rightarrow 0$. Here we write $\nabla(p_i),I(p_i)$ for the costandard
and the indecomposable injective objects of $\Cat_n$ labeled by $p_i$.
\item[(iii)] The indecomposable tilting objects $T(p_{i-1})$ for $i>1$ coincide with $I(p_i)$.
\item[(iv)] The simple objects $L(p_i)$ with $i>1$ appear in the socles of tiltings, while
$\operatorname{RHom}_{\Cat_n}(L(p_1),T)$
is concentrated in homological degree $n$.
\item[(v)] There is a unique simple in $\Cat_n^\vee$ that appears in the higher cohomology of $\operatorname{RHom}_{\Cat_n}(\bullet,T)$.
\end{itemize}

\begin{Thm}\label{Thm:catO_str}
Consider a parameter of the form $\lambda=\frac{q}{n}$ with coprime $q,n$. Then  the following is true.
\begin{enumerate}
\item The category $\mathcal{O}(\bar{\A}_\lambda^\theta(n,r))$ has only one nontrivial block that is equivalent
to $\Cat_{rn}$. This block contains an irreducible representation supported on $\bar{\rho}^{-1}(0)$.
\item Suppose the one parameter torus used to define the category $\mathcal{O}$ is of the form $t\mapsto (\alpha(t),t)$, where
$\alpha(t)=\operatorname{diag}(t^{d_1},\ldots, t^{d_r})$ with
$d_i-d_{i+1}>n$ for all $i$. Then the labels in the non-trivial block of
$\mathcal{O}(\bar{\A}_\lambda^\theta(n,r))$ are hooks $h_{i,d}=(\varnothing,\ldots, (n+1-d, 1^{d-1}),\ldots, \varnothing)$
(where $i$ is the number of the diagram where the hook appears) ordered by $h_{1,n}>h_{1,n-1}>\ldots>h_{1,1}>h_{2,n}>\ldots>h_{2,1}>\ldots>h_{r,1}$.
\end{enumerate}
\end{Thm}
\begin{proof}
The proof is in several steps. We again deal with the realization of our category as $\OCat(\bar{\A}_\lambda(n,r))$, where
$\lambda$ is Zariski generic and such that $(\lambda,\theta)$ satisfies the abelian localization.

{\it Step 1}. As we have seen in Step 4 of the proof of Theorem \ref{Thm:fin dim}, all simples have maximal dimension of support,
except one, let us denote it by $L$, which is finite dimensional. So all blocks but one consist of modules with support of maximal
dimension. Now arguing as in the first two steps of the proof of Theorem \ref{Thm:fin dim}, we see that the blocks
that do not contain $L$ are simple. Let $\mathcal{C}$ denote the nontrivial block. The label of $L$, denote it by $p_{max}$,
is the largest in any highest weight ordering. For all other labels $p$ the simple $L(p)$ lies in the socle
of the tilting generator $T$.  In other words an analog of (iv) above holds for $\mathcal{C}$ with $rn$ instead of
$n$. In the subsequent steps we will show that $\mathcal{C}\cong \mathcal{C}_{rn}$.

{\it Step 2}. Let us show that an analog of (v) holds for $\mathcal{C}$.  By Corollary \ref{Cor:D_higher_codim},
the higher cohomology of $D(L)$ cannot have support of maximal dimension. It follows that the higher cohomology is
finite dimensional and so are direct sums of a single simple in $\mathcal{O}(\bar{\A}_{-r-\lambda}(n,r))$. Since the Ringel duality is the same as the
homological duality (up to an equivalence of abelian categories, see Proposition \ref{Prop:contr_Ringel}), we are done.

{\it Step 3}.  Let us show that there is a unique minimal label for $\mathcal{C}$, say $p_{min}$. This is equivalent to $\mathcal{C}^\vee$ having a unique maximal label because the orders on $\mathcal{C}$ and $\mathcal{C}^\vee$ are opposite.
But $\mathcal{C}^\vee$ is equivalent to the nontrivial block in $\OCat(\bar{\A}_{-r-\lambda}(n,r))$.
So we are done by Step 1 (applied to $-r-\lambda$ instead of $\lambda$) of this proof.

{\it Step 4}. Let us show that (v) implies that any tilting in $\mathcal{C}$ but one is injective.
Let $R^\vee$ denote the Ringel duality equivalence $D^b(\mathcal{C}^\vee)\rightarrow D^b(\mathcal{C})^{opp}$.
Let us label the tiltings by the label of the top costandard in a filtration with costandard subsequent quotients.
We have $\Ext^i(L(p'),T(p))=\Hom(L(p')[i],T(p))=\Hom((R^\vee)^{-1}T(p)[i], (R^\vee)^{-1}L(p'))$. The objects $(R^{\vee})^{-1}T(p)$ are projective so $\Ext^i(L(p'),T(p))=\Hom((R^\vee)^{-1} T(p), H^i((R^\vee)^{-1} L(p')))$. Similarly to the
previous step (applied to  $\mathcal{C}^\vee$ instead of $\mathcal{C}$ and $(R^{\vee})^{-1}$ instead of $R$),
there is a unique indecomposable projective $P^\vee(p^\vee)$ in $\Cat^\vee$ that can map nontrivially to a higher
homology of $(R^{\vee})^{-1} L(p)$. So if $(R^\vee)^{-1} T(p)\neq P^\vee(p^\vee)$, then $T(p)$ is injective.

{\it Step 5}. We remark that $\nabla(p_{max})$ is injective but not tilting, while $\nabla(p_{min})$ is tilting
but not injective. So the injectives in $\mathcal{C}$ are $\nabla(p_{max})$ and $T(p)$ for $p\neq p_{min}$.
Similarly, the tiltings are $I(p), p\neq p_{max}$, and $\nabla(p_{min})$.

{\it Step 6}. Let $\Lambda$ denote the highest weight poset for $\mathcal{C}$.
Let us define a map $\nu: \Lambda\setminus \{p_{min}\}\rightarrow \Lambda\setminus \{p_{max}\}$.
It follows from  Step 5 that  the socle of any tilting in $\mathcal{C}$ is simple. By definition,
$\nu(p)$ is such that $L(\nu(p))$
is the socle of $T(p)$. We remark that $\nu(p)\leqslant p$ for any highest weight order.

{\it Step 7}. Let us show that any element $p\in \Lambda$ has the form $\nu^i(p_{max})$. Assume the converse
and let us pick the maximal element not of this form, say $p'$. Since $p'\neq p_{max}$, we see that $L(p')$ lies
in the socle of some tilting. But the socle of any indecomposable tilting is simple. So $\nabla(p')$ is a bottom term of a
filtration   with constandard subsequent quotients. By the definition of $\nu$ and the choice of $p'$, $\nabla(p')$ is tilting itself.
Any indecomposable tilting but $\nabla(p_{min})$ is injective and we cannot have a costandard that is injective
and tilting simultaneously. So $p'=p_{min}$. But let us pick a minimal element $p''$ in $\Lambda\setminus \{p_{min}\}$. By above in this step, $\nu(p'')<p''$. So $\nu(p'')=p_{min}$. The claim in the beginning of
the step is established. This proves (i) for $\mathcal{C}$.

{\it Step 8}. (ii) for $\mathcal{C}$ follows from Step 7 and (iii) follows from (ii) and Step 5.

{\it Step 9}. Let us show that $\#\Lambda=rn$. The minimal injective resolution for $\nabla(p_{min})$
has length $\#\Lambda$, all injectives there are different, and the last term is $\nabla(p_{max})$. It follows that
$\operatorname{RHom}(L(p_{max}),\nabla(p_{min}))$ is concentrated in homological degree $\#\Lambda-1$.
The other tiltings are injectives and $\operatorname{RHom}$'s with them amount to $\Hom$'s. Since $\operatorname{RHom}(L(p_{max}),T)$
is concentrated in homological degree $rn-1$ (because of the coincidence of the Ringel and the homological dualities), we are done.

{\it Step 10}. Let us complete the proof of (1). Let us order the labels in $\Lambda$  decreasingly, $p_1>\ldots>p_{rn}$.
Using (ii) we get the following claims.
\begin{itemize}
\item $\operatorname{End}(I(p_i))=\C[x]/(x^2)$ for $i>1$ and $\operatorname{End}(I(p_1))=\C$.
\item $\operatorname{Hom}(I(p_i),I(p_j))$ is 1-dimensional if
$|i-j|=1$ and is $0$ if $|i-j|>1$.
\end{itemize}
Choose some basis elements $a_{i,i+1}, i=1,\ldots,rn-1$  in $\operatorname{Hom}(I(p_i), I(p_{i+1}))$
and also basis elements $a_{i+1,i}\in \operatorname{Hom}(I(p_{i+1}),I(p_i))$. We remark that
the image of the composition map $\Hom(I(p_i),I(p_{i+1}))\times \Hom(I(p_{i+1}),I(p_i))
\rightarrow \operatorname{End}(I(p_i))$ spans the maximal ideal. Choose  generators
$a_{ii}$ in the maximal ideals of $\End(I(p_i)), i=2,\ldots,rn$.
Normalize $a_{21}$ by requiring that $a_{21}a_{12}=a_{22}$, automatically, $a_{12}a_{21}=0$. Normalize $a_{32}$
by $a_{23}a_{32}=a_{22}$ and then normalize $a_{33}$ by $a_{33}=a_{32}a_{23}$. We continue normalizing $a_{i+1,i}$
and $a_{i+1,i+1}$ in this way. We then recover the multiplication table in $\operatorname{End}(\bigoplus I(\lambda_i))$
in a unique way. This completes the proof of (1).

{\it Step 11}. Now let us prove (2).
Let us check that the labeling set  $\Lambda$ for the nontrivial block of $\OCat(\bar{\A}^\theta_\lambda(n,r))$ consists of hooks.
For this, it is enough to check that $\Delta(h_{i,d})$ does not form a block. This in turn, will follow
if we check that there is a nontrivial homomorphism between $\Delta(h_{i,d})$ and some other $\Delta(h_{i,d'})$.
This is done similarly to the second paragraph of Step 4 in the proof of Theorem \ref{Thm:fin dim}.
Now, according to \cite{Korb},
the hooks are ordered as specified in (2) with respect to the geometric order
on the torus fixed points in $\M^\theta(n,r)$ (note that the sign conventions here and in \cite{Korb}
are different).
\end{proof}

\begin{Rem}\label{Rem:simpl_label}
We can determine the label of the simple supported on $\bar{\rho}^{-1}(0)$ in the category $\OCat$ corresponding to
an arbitrary generic torus. Namely, note that $\bar{\rho}^{-1}(0)$ coincides with the closure of a single contracting component
and that contracting component corresponds to the maximal point. Now we can use results of \cite{Korb} to find a label
of the point: it always has only one nontrivial partition and this partition is either $(n)$ or $(1^n)$.
\end{Rem}

\section{Localization theorems in the Gieseker case}\label{S_loc}
In this section we prove Theorem \ref{Thm:loc}. The proof is in the following steps.
\begin{itemize}
\item We apply results of McGerty and Nevins, \cite{MN_ab}, to show that, first, if the abelian localization  fails for $(\lambda,\theta)$, then $\lambda$ is a rational number with denominator not exceeding $n$, and, second,  the parameters $\lambda=\frac{q}{m}$
with $m\leqslant n$ and $-r<\lambda<0$ are indeed singular and the functor $\Gamma_\lambda^\theta$ is
exact when $\lambda>-r, \theta>0$ or $\lambda<0,\theta<0$. Thanks to an isomorphism $\A_{\lambda}^\theta(n,r)\cong
\A_{-\lambda-r}^{-\theta}(n,r)$, this reduces the conjecture to checking that the abelian localization
holds for $\lambda=\frac{q}{m}$ with $q\geqslant 0, m\leqslant n$.
\item Then we reduce the proof to the case when the denominator is precisely $n$ and $\lambda,\theta>0$.
\item Then we will study a connection between the algebras $\Ca_\alpha(\bar{\A}_{\lambda}(n,r)),
\Gamma(\Ca_\alpha(\bar{\A}^\theta_{\lambda}(n,r)))$. We will show that the numbers of simples in the categories
$\OCat$ for these algebras coincide.  We deduce the localization theorem from there.
\end{itemize}

The last step is a crucial one and it does not generalize to other quiver varieties.

\subsection{Results of McGerty and Nevins and consequences}\label{SS_MN_appl}

In \cite{MN_ab}, McGerty and Nevins found a sufficient condition for the functor $\Gamma_\lambda^\theta:\A_\lambda^\theta(n,r)\operatorname{-mod}\rightarrow
\A_\lambda(n,r)\operatorname{-mod}$ to be exact (they were dealing with more general
Hamiltonian reductions but we will only need the Gieseker case). Let us explain what their
result give in the case of interest for us. Consider the quotient functors
$\pi_\lambda: D_R\operatorname{-mod}^{G,\lambda}\twoheadrightarrow \A_\lambda(n,r)\operatorname{-mod}$
and $\pi_\lambda^\theta:D_R\operatorname{-mod}^{G,\lambda}\twoheadrightarrow \A_\lambda^\theta(n,r)\operatorname{-mod}$.

\begin{Prop}\label{Prop:MN}
The inclusion $\ker \pi_{\lambda}^{\det}\subset \ker \pi_\lambda$ holds  provided $\lambda>-r$. Similarly, $\ker\pi_{\lambda}^{\det^{-1}}
\subset \pi_{\lambda}$  provided $\lambda< 0$.
\end{Prop}

I would like to thank Dmitry Korb for explaining me the required modifications to \cite[Section 8]{MN_ab}.

\begin{proof}
We will consider the case $\theta=\det$, the opposite case follows from $\A_\lambda^{-\theta}(n,r)\cong \A_{-r-\lambda}^\theta(n,r)$.
The proof closely follows \cite[Section 8]{MN_ab}, where the case of $r=1$ is considered. Instead of $R=\operatorname{End}(V)\oplus
\operatorname{Hom}(V,W)$ they use $R'=\operatorname{End}(V)\oplus \operatorname{Hom}(W,V)$, then, thanks to the partial
Fourier transform, we have $D(R)\operatorname{-mod}^{G,\lambda}\cong D(R')\operatorname{-mod}^{G,\lambda+r}$. The set of weights in $R'$ for a maximal torus $T\subset \GL(V)$
is independent of $r$ so we have the same Kempf-Ness subgroups as in the case $r=1$:  it is enough to consider the subgroups
$\beta$ with tangent vectors (in the notation of \cite[Section 8]{MN_ab}) $e_1+\ldots+e_k$. The shift in {\it loc.cit.} becomes
$\frac{rk}{2}$ (in the computation of {\it loc.cit.} we need to take the second summand $r$ times, that is all that changes). So we get
that $\ker \pi_{\lambda}^{\det}\subset \ker \pi_\lambda$ provided $k(-\frac{r}{2}-\lambda)\not\in \frac{rk}{2}+\Z_{\geqslant 0}$
for all possible $k$ meaning $1\leqslant k\leqslant n$ (the number $-\frac{r}{2}-\lambda$ is $c'$ in {\it loc.cit.}). The condition
simplifies to $\lambda\not\in -r-\frac{1}{k}\Z_{\geqslant 0}$. This implies the claim of the proposition.
\end{proof}

\subsection{Reduction to denominator $n$ and singular parameters}\label{SS_loc_red_to_n}
Proposition \ref{Prop:MN} allows us to show that certain parameters are singular.
\begin{Cor}\label{Cor:sing}
The parameters $\lambda$ with denominator $\leqslant n$ and $-r<\lambda<0$ are singular.
\end{Cor}
\begin{proof}
Assume the converse. Since $R\Gamma_\lambda^{\pm \theta}$ are equivalences and $\Gamma_\lambda^{\pm \theta}$
are exact, we see that $\Gamma_{\lambda}^{\pm \theta}$ are equivalences of abelian categories. From the
inclusions $\ker \pi_{\lambda}^{\pm \theta}\subset \ker\pi_\lambda$, we deduce that the functors
$\pi_{\lambda}^{\pm \theta}$ are isomorphic. So the wall-crossing functor $\mathfrak{WC}_{\lambda\rightarrow \lambda^-}=\pi_{\lambda^-}^{-\theta}\circ (\C_{\lambda^--\lambda}\otimes\bullet)\circ L\pi_{\lambda}^{\theta*}$
(see \cite[(2.8)]{BL} for the equality)
is an equivalence of abelian categories (where we modify $\lambda$ by adding a sufficiently large integer). However, we have already seen that it does shift some modules,
since not all modules in $\mathcal{O}(\bar{\A}_\lambda(n,r))$ have support of maximal dimension (see the end of the proof of Theorem \ref{Thm:fin dim}).
\end{proof}

Now let us observe that it is enough to check that the abelian localization holds for $\lambda\geqslant 0$
and $\theta>0$. This follows from an isomorphism $\A_{\lambda}^\theta(n,r)\cong \A_{-\lambda-r}^{-\theta}(n,r)$.
This an isomorphism of sheaves on $\M^\theta(n,r)\cong \M^{-\theta}(n,r)$ (see the proof of Lemma \ref{Lem:iso}).

Now let us reduce the proof of Theorem \ref{Thm:loc} to the case when $\lambda$ has denominator $n$.
Let the denominator $n'$ be less then $n$.
As we have seen in \cite[Section 5]{BL}, the abelian localization holds for $(\lambda,\theta>0)$ if and only if the bimodules
$\A^0_{\lambda,\chi}(n,r):=[D(R)/D(R)\{x_R-\langle\lambda,x\rangle\}]^{G,\chi}, \A^0_{\lambda+\chi,-\chi}(n,r)$ with
 $\chi\gg 0$ define mutually dual Morita equivalences, equivalently, the natural homomorphisms
\begin{equation}\label{eq:nat_homs}
\begin{split}&\A^0_{\lambda,\chi}(n,r)\otimes_{\A_{\lambda}(n,r)} \A^0_{\lambda+\chi,-\chi}(n,r)\rightarrow \A_{\lambda+\chi}(n,r),\\
&\A^0_{\lambda+\chi,-\chi}(n,r)\otimes_{\A_{\lambda+\chi}(n,r)}\A^0_{\lambda,\chi}(n,r)\rightarrow \A_{\lambda}(n,r)
\end{split}
\end{equation}
are isomorphisms.

Assume the converse. Let $K^1,C^1,K^2,C^2$ denote the kernel and the cokernel of the first and of the
second homomorphism, respectively. If one of these bimodules is nontrivial, then we can find $x\in \M(n,r)$
such that $K^i_{\dagger,x},C^i_{\dagger,x}$ are finite dimensional, and, at least one of these bimodules
is nonzero. From the classification of finite dimensional irreducibles, we see that the slice algebras
must be of the form $\bar{\A}_{?}(n',r)^{\otimes k}$. But then $\A^0_{\lambda+\chi,-\chi}(n,r)_{\dagger,x}=
\bar{\A}^0_{\lambda+\chi,-\chi}(n',r)^{\otimes k}, \A^0_{\lambda,\chi}(n,r)_{\dagger,x}=
\bar{\A}^0_{\lambda,\chi}(n',r)^{\otimes k}$. Further, applying $\bullet_{\dagger,x}$ to (\ref{eq:nat_homs}) we again get  natural
homomorphisms. But the localization theorem holds for the algebra $\bar{\A}_{\lambda}(n',r)$ thanks to our inductive
assumption, so the homomorphisms of the $\bar{\A}_\lambda(n',r)^{\otimes k}$-bimodules are isomorphisms. This
contradiction justifies the reduction to denominator $n$.

\subsection{Number of simples in $\mathcal{O}(\A_\lambda(n,r))$}\label{SS_loc_simpl_numb}
So we need to prove that the localization theorem holds for positive parameters $\lambda$ with denominator $n$
(the case $\lambda=0$ occurs only if $n=1$ and in that case this is a classical localization theorem for
differential operators on projective spaces).
We will derive the proof from the claim that the number of simple objects in the categories $\OCat(\bar{\A}_\lambda(n,r))$
and $\OCat(\bar{\A}_\lambda^\theta(n,r))$ is the same. For this we will need to study the natural homomorphism
$\varphi:\Ca_\alpha(\bar{\A}_\lambda(n,r))\rightarrow \Gamma(\Ca_\alpha(\bar{\A}^\theta_\lambda(n,r)))$. Here, as before,
$\alpha:\C^\times \rightarrow \GL(r)$ is of the form $t\mapsto (t^{d_1},\ldots,t^{d_r})$, where $d_1\gg d_2\gg\ldots\gg d_r$.

Recall that $\Gamma(\Ca_\alpha(\bar{\A}^\theta_\lambda(n,r)))=\bigoplus \bar{\A}_\lambda(n_1,\ldots,n_r;r)$, where the summation
is taken over all compositions $n=n_1+\ldots+n_r$ and  $\bar{\A}_\lambda(n_1,\ldots,n_r;r)\otimes D(\C)=\bigotimes_{i=1}^r \A_{\lambda+i-1}(n_i,1)$ (the factor $D(\C)$ is embedded into the right hand side ``diagonally'').
Let $\mathcal{B}$ denote the maximal
finite dimensional quotient of $\Gamma(\Ca_\alpha(\bar{\A}_\lambda^\theta(n,r)))$.

\begin{Prop}\label{Prop:surject}
The composition of $\varphi$ with the projection $\Gamma(\Ca_\alpha(\bar{\A}^\theta_\lambda(n,r)))\twoheadrightarrow \mathcal{B}$
is surjective.
\end{Prop}
\begin{proof}
The proof is in several steps.

{\it Step 1}. We claim that it is sufficient to prove that the composition $\varphi_i$ of $\varphi$ with the projection
$\Gamma(\Ca_\alpha(\bar{\A}_\lambda^\theta(n,r)))\rightarrow \bar{\A}_{\lambda+i}(n,1)$ is surjective. Indeed, each $\bar{\A}_{\lambda+i}(n,1),
i=0,\ldots,r-1$ has a unique finite dimensional representation. The dimensions of these representations are
pairwise different, see \cite{BEG1}. Namely, if $\lambda=\frac{q}{n}$, then
the dimension is $\frac{(q+n-1)!}{q!n!}$. So $\mathcal{B}$ is the sum of $r$
pairwise non-isomorphic matrix algebras. Therefore the surjectivity of the homomorphism $\Ca_\alpha(\bar{\A}_\lambda(n,r))\rightarrow \mathcal{B}$ follows from the surjectivity of all its $r$ components. We remark that the other summands of
$\Ca_\alpha(\bar{\A}_\lambda(n,r))$ have no finite dimensional representations.

{\it Step 2}. Generators of $\bar{A}_{\lambda+i}(n,1)$ are known. Namely, recall that $\bar{A}_{\lambda+i}(n,1)$ is the spherical
subalgebra in the Cherednik algebra $H_c(n)$ for the reflection representation $\h$ of $\mathfrak{S}_n$ with $c=\lambda+i$. The
latter is generated by $\h,\h^*$. Then algebra $eH_c(n)e$ is generated by $S(\h)^W,S(\h^*)^W$,
see \cite{EG}. On the level of quantum Hamiltonian reduction, $S(\h)^W$ coincides with the image of $S(\g)^G$, while
$S(\h^*)^W$ coincides with the image of $S(\g^*)^G$. Here we write $\g$ for $\mathfrak{sl}_n$. We will show that these images lie in the image of $\varphi_i:\Ca_\alpha(\bar{\A}_{\lambda}(n,r))\rightarrow \bar{\A}_{\lambda+i}(n,1)$, this will establish the surjectivity in Step 1.

{\it Step 3}. Let us produce a natural homomorphism $S(\g^*)^G\rightarrow \Ca_\alpha(\bar{\A}_{\lambda}(n,r))$. First of all, recall that
$\bar{\A}_{\lambda}(n,r)$ is a quotient of $D(\g\oplus (\C^{*n})^r)^{G}$.
The algebra $S(\g^*)^G$ is included into $D(\g\oplus (\C^{*n})^{\oplus r})^G$
as the algebra of invariant functions on $\g$. So we get a homomorphism $S(\g)^G\rightarrow
\bar{\A}_{\lambda}(n,r)$. Since the $\C^\times$-action $\alpha$
used to form $\Ca_\alpha(\bar{\A}_\lambda(n,r))$ is nontrivial only on $(\C^{*n})^{\oplus r}$, we see that
the image of $S(\g^*)^G$ lies in $\bar{\A}_{\lambda}(n,r)^{\alpha(\C^\times)}$. So  we get a
homomorphism $\iota: S(\g^*)^G\rightarrow \Ca_\alpha(\bar{\A}_\lambda(n,r))$.

{\it Step 4}. We claim that $\varphi_i\circ \iota$ coincides with the inclusion $S(\g^*)^G\rightarrow \bar{\A}_{\lambda+i}(n,1)$.
We can filter  the algebra $D(\g\oplus (\C^{*n})^{\oplus r})$ by the order of a differential operator. This induces
filtrations on $\bar{\A}_{\lambda}(n,r),\bar{\A}^\theta_\lambda(n,r)$. We have similar filtrations on the algebras
$\bar{\A}_{\lambda+i}(n,1)$. The filtrations on $\bar{\A}_\lambda(n,r),\bar{\A}_\lambda^\theta(n,r)$ are preserved by $\alpha$
and hence we have filtrations on $\Ca_\alpha(\bar{\A}_\lambda(n,r)),\Gamma(\Ca_\alpha(\bar{\A}_\lambda^\theta(n,r)))$.
It is clear from the construction of the projection $\Gamma(\Ca_\alpha(\bar{\A}_\lambda^\theta(n,r)))\rightarrow
\bar{\A}_{\lambda+i}(n,1)$ that it is compatible with the filtration. On the other hand, the images of $S(\g^*)^G$ in both
$\Ca_\alpha(\bar{\A}_{\lambda}(n,r)),\bar{\A}_{\lambda+i}(n,1)$ lies in the filtration degree 0.
So it is enough to prove the coincidence of the  homomorphisms in the beginning of the step  after passing to associate
graded algebras.

{\it Step 5}. The associated graded homomorphisms coincide with analogous homomorphisms defined on the classical level.
Recall that the components of $\M^\theta(n,r)^{\alpha(\C^\times)}$ that are Hilbert schemes are realized as follows.
Pick an eigenbasis $w_1,\ldots,w_r$ for the fixed $r$-dimensional torus in $\operatorname{GL}_r$.
Then the $i$th component that is the Hilbert schemes consists of  $G$-orbits of $(A,B,0,j)$, where $j:\C^n\rightarrow \C^r$
is a map with image in  $\C w_j$. In particular, the homomorphism $S(\g^*)^G\rightarrow \gr \A_{\lambda+i}(n,1)$
is dual to the morphism given by $(A,B,0,j)\rightarrow A$.

On the other hand, the component of $\M^\theta(n,r)^{\alpha(\C^\times)}$ in consideration  maps
onto $\M(r,n)\quo \alpha(\C^\times)$ (via sending the orbit of $(A,B,0,j)$
to the orbit of the same element). The corresponding homomorphism of algebras is the associated graded of $\bar{\A}_\lambda(n,r)^{\alpha(\C^\times)}\rightarrow \bar{\A}_{\lambda+i}(n,1)$. Then we have the morphism
$\M(r,n)\quo \alpha(\C^\times)\rightarrow \g\quo G$ given by $(A,B,0,j)\mapsto A$.
The corresponding homomorphism of algebras is the associated graded
of $S(\g^*)^G\rightarrow \bar{\A}_\lambda(n,r)^{\alpha(\C^\times)}$. We have checked that the associated graded
homomorphism of $\varphi_i\circ\iota:S(\g^*)^G\rightarrow \bar{\A}_{\lambda+i}(n,1)$ coincides with that of the embedding
$S(\g^*)^G\rightarrow \bar{\A}_{\lambda+i}(n,1)$. This proves the claim of Step 4.

{\it Step 6}. The coincidence of similar homomorphisms $S(\g)^G\rightarrow \bar{\A}_{\lambda+i}(n,1)$ is established
analogously. The proof of the surjectivity of $\Ca_\alpha(\bar{\A}_{\lambda}(n,r))\rightarrow \bar{\A}_{\lambda+i}(n,1)$ is now complete.
\end{proof}

We still have a Hamiltonian action of $\C^\times$ on $\Ca_\alpha(\bar{\A}_\lambda(n,r))$ that makes the homomorphism
$\Ca_\alpha(\bar{\A}_\lambda(n,r))\rightarrow \Gamma(\Ca_\alpha(\bar{\A}_\lambda^\theta(n,r)))$ equivariant. So we can form
the category $\mathcal{O}(\Ca_\alpha(\bar{\A}_\lambda(n,r)))$ for this action. By Lemma \ref{Lem:prec_spec}, we have
$\alpha\prec^\lambda (m\alpha,1)$ for $m\gg 0$. We rescale $\alpha$ and assume that $m=1$. Recall, Lemma \ref{Lem:parab_ind}, that we have an isomorphism $\Ca_{1}(\Ca_\alpha(\bar{\A}_\lambda^\theta(n,r)))\cong
\Ca_{(\alpha,1)}(\bar{\A}_\lambda^\theta(n,r))$. So there is
a natural bijection between the sets of simples in $\mathcal{O}(\Ca_\alpha(\bar{\A}_\lambda(n,r)))$ and in
$\mathcal{O}(\bar{\A}_\lambda(n,r))$.

\begin{Prop}\label{Prop:simple_numbers}
The number of simples in $\mathcal{O}(\Ca_\alpha(\bar{\A}_\lambda(n,r)))$
is bigger than or equal to that in $\mathcal{O}(\Gamma(\Ca_\alpha(\bar{\A}_\lambda^\theta(n,r))))$.
\end{Prop}
\begin{proof}
The proof is again in several steps.

{\it Step 1}.
We have a natural homomorphism $\C[\g]^G\rightarrow \bigoplus \bar{\A}_\lambda(n_1,\ldots,n_r;r)$. It can be described as follows.
We have an identification $\C[\g]^G\cong \C[\h]^{\mathfrak{S}_n}$. This algebra embeds into $\bar{\A}_\lambda(n_1,\ldots,n_r;r)$
(that is a spherical Cherednik algebra for the group $\prod_{i=1}^r \mathfrak{S}_{n_i}$ acting on $\h$) via the inclusion $\C[\h]^{\mathfrak{S}_n}\subset
\C[\h]^{\mathfrak{S}_{n_1}\times\ldots\times \mathfrak{S}_{n_r}}$. For the homomorphism
$\C[\g]^G\rightarrow \bigoplus \bar{\A}_\lambda(n_1,\ldots,n_r;r)$ we take the direct sum of these embeddings. Similarly
to Steps 4,5 of the  proof of Proposition \ref{Prop:surject}, the maps $\C[\g]^G\rightarrow \Ca_\alpha(\bar{\A}_\lambda(n,r)), \Gamma(\Ca_\alpha(\bar{\A}_\lambda^\theta(n,r)))$
are intertwined by the homomorphism $\Ca_\alpha(\bar{\A}_\lambda(n,r))\rightarrow \Gamma(\Ca_\alpha(\bar{\A}_\lambda^\theta(n,r)))$.

{\it Step 2}. Let $\delta\in \C[\g]^G$ be the discriminant. We claim that $\Ca_\alpha(\bar{\A}_\lambda^\theta(n,r))[\delta^{-1}]\xrightarrow{\sim} \Gamma(\Ca_\alpha(\bar{\A}_\lambda^\theta(n,r)))[\delta^{-1}]$. Since $\delta$ is $\alpha(\C^\times)$-stable,
we have $\Ca_\alpha(\bar{\A}_\lambda(n,r))[\delta^{-1}]=\Ca_\alpha(\bar{\A}_\lambda(n,r)[\delta^{-1}])$.
We will describe the algebra $\Ca_\alpha(\bar{\A}_\lambda(n,r)[\delta^{-1}])$ explicitly and see that
$\Ca_\alpha(\bar{\A}_\lambda(n,r)[\delta^{-1}])\xrightarrow{\sim} \Gamma(\Ca_\alpha(\bar{\A}_\lambda^\theta(n,r)))[\delta^{-1}]$.

{\it Step 3.} We start with the description of $\bar{\A}_\lambda(n,r)[\delta^{-1}]$. Let $\g^{reg}$ denote the locus
of the regular semisimple elements in $\g$. Then $\bar{\A}_\lambda(n,r)[\delta^{-1}]=D(\g^{reg}\times \operatorname{Hom}(\C^n,\C^r))\red_\lambda G$. Here $\red_\lambda$ denotes the quantum Hamiltonian reduction with parameter $\lambda$.

Recall that $\g^{reg}=G\times_{N_G(\h)}\h^{reg}$ and so $\g^{reg}\times \Hom(\C^n,\C^r)=G\times_{N_G(\h)}(\h^{reg}\times \operatorname{Hom}(\C^n,\C^r))$. It follows that
\begin{align*}&D(\g^{reg}\times \operatorname{Hom}(\C^n,\C^r))\red_\lambda G=D(\h^{reg}\times \Hom(\C^n,\C^r))\red_\lambda N_G(\h)=\\
&(D(\h^{reg})\otimes D(\operatorname{Hom}(\C^n,\C^r))\red_\lambda H)^{\mathfrak{S}_n}= \left(D(\h^{reg})\otimes D^\lambda(\mathbb{P}^{r-1})^{\otimes n}\right)^{\mathfrak{S}_n}.\end{align*}
Here, in the second line, we write $H$ for the Cartan subgroup of $G$ and take the diagonal action of $\mathfrak{S}_n$.
In the last expression, it permutes the tensor factors.
A similar argument shows that $\bar{\M}^\theta(n,r)_{\delta}= (T^*(\h^{reg})\times T^*(\mathbb{P}^{r-1})^{n})/\mathfrak{S}_n$
and the restriction of $\bar{\A}^\theta_\lambda(n,r)$ to this open subset is $\left(D_{\h^{reg}}\otimes (D^\lambda_{\mathbb{P}^{r-1}})^{\otimes n}\right)^{\mathfrak{S}_n}$.

{\it Step 4}. Now we are going to describe the algebra $\Ca_\alpha(\left(D(\h^{reg})\otimes D^\lambda(\mathbb{P}^{r-1})^{\otimes n}\right)^{\mathfrak{S}_n})$. First of all, we claim that
\begin{equation}\label{eq:Ca_eq}\Ca_\alpha(\left(D(\h^{reg})\otimes D^\lambda(\mathbb{P}^{r-1})^{\otimes n}\right)^{\mathfrak{S}_n})=
(\Ca_\alpha\left(D(\h^{reg})\otimes D^\lambda(\mathbb{P}^{r-1})^{\otimes n}\right))^{\mathfrak{S}_n}\end{equation}
There is a natural homomorphism from the left hand side to the right hand side.
To prove that it is an isomorphism one can argue as follows. First, note, that since the $\mathfrak{S}_n$-action
on $\h^{reg}$ is free, we have
$$D(\h^{reg})\otimes D^\lambda(\mathbb{P}^{r-1})^{\otimes n}=D(\h^{reg})\otimes_{D(\h^{reg})^{\mathfrak{S}_n}}\left(D(\h^{reg})\otimes D^\lambda(\mathbb{P}^{r-1})^{\otimes n}\right)^{\mathfrak{S}_n}$$
Since $D(\h^{reg})$ is $\alpha(\C^\times)$-invariant, the previous equality implies (\ref{eq:Ca_eq}).

{\it Step 5}. Now let us describe $\Ca_\alpha(\left(D(\h^{reg})\otimes D^\lambda(\mathbb{P}^{r-1})^{\otimes n}\right)=D(\h^{reg})\otimes \Ca_\alpha\left((D^\lambda(\mathbb{P}^{r-1}))^{\otimes n}\right)$. The $\C^\times$-action on the tensor product $(D^\lambda(\mathbb{P}^{r-1}))^{\otimes n}$ is diagonal and it is easy to see that
$\Ca_\alpha\left((D^\lambda(\mathbb{P}^{r-1}))^{\otimes n}\right)=\left(\Ca_\alpha(D^\lambda(\mathbb{P}^{r-1}))\right)^{\otimes n}$.
  So we need to compute $\Ca_\alpha(D^\lambda(\mathbb{P}^{r-1}))$. We claim that this algebra is isomorphic to $\C^{\oplus r}$. Indeed, $D^\lambda(\mathbb{P}^{r-1})$ is a quotient of the central reduction $U_{\tilde{\lambda}}(\mathfrak{sl}_r)$ of $U(\mathfrak{sl}_r)$ at
the central character $\tilde{\lambda}:=\lambda\omega_{r}$. We remark that $\lambda\omega_r+\rho$ is regular because $\lambda\geqslant 0$. We have $\Ca_\alpha(U_{\tilde{\lambda}}(\mathfrak{sl}_r))=\C^{\oplus r!}$ and $\Ca_\alpha(D^\lambda(\mathbb{P}^{r-1}))$ is a quotient of that. The number of irreducible representations of $\Ca_\alpha(D^\lambda(\mathbb{P}^{r-1}))$ equals to the number of simples in the category $\mathcal{O}$ for $D^\lambda(\mathbb{P}^{r-1})$ that coincides with $r$ since the localization holds. An isomorphism  $\Ca_\alpha(D^\lambda(\mathbb{P}^{r-1}))=\C^{\oplus r}$ follows.

{\it Step 6}. So we see that $\Ca_\alpha(\bar{\A}_\lambda(n,r)[\delta^{-1}])= \left(D(\h^{reg})\otimes (\C^{\oplus r})^{\otimes n}\right)^{\mathfrak{S}_n}$.
By similar reasons, we have $\Gamma([\bar{\M}^\theta(n,r)_\delta]^{\alpha(\C^\times)}, \Ca_\alpha(\bar{\A}^\theta_\lambda(n,r)))=
\left(D(\h^{reg})\otimes (\C^{\oplus r})^{\otimes n}\right)^{\mathfrak{S}_n}$. The natural homomorphism
\begin{equation}\label{eq:local_iso1}\Ca_\alpha(\bar{\A}_\lambda(n,r)[\delta^{-1}])\rightarrow \Gamma((\bar{\M}^\theta(n,r)_\delta)^{\alpha(\C^\times)}, \Ca_\alpha(\bar{\A}^\theta_\lambda(n,r)))\end{equation}
is an isomorphism by the previous two steps. Also we have a natural homomorphism \begin{equation}\label{eq:local_iso2}\Gamma(\Ca_\alpha(\bar{\A}^\theta_\lambda(n,r)))[\delta^{-1}]\rightarrow
\Gamma([\M^\theta(n,r)_\delta]^{\alpha(\C^\times)}, \Ca_\alpha(\bar{\A}^\theta_\lambda(n,r))).\end{equation}
The latter homomorphism is an isomorphism from the explicit description
of $\Ca_\alpha(\bar{\A}^\theta_\lambda(n,r))$. Indeed, $\Ca_\alpha(\bar{\A}^\theta_\lambda(n,r))$
is the direct sum of quantizations of products of Hilbert schemes. The morphism $\prod \operatorname{Hilb}_{n_i}(\C^2)
\rightarrow \prod \C^{2n_i}/\mathfrak{S}_n$ is an isomorphism over the non-vanishing locus of $\delta$.
This implies that (\ref{eq:local_iso2}) is an isomorphism.

By the construction,  (\ref{eq:local_iso1}) is the composition of
$\Ca_\alpha(\bar{\A}_\lambda(n,r)[\delta^{-1}])\rightarrow
\Gamma(\Ca_\alpha(\bar{\A}^\theta_\lambda(n,r)))[\delta^{-1}]$ and (\ref{eq:local_iso2}).
So we have proved that  $\Ca_\alpha(\bar{\A}_\lambda(n,r))[\delta^{-1}]\rightarrow \Gamma(\Ca_\alpha(\bar{\A}^\theta_\lambda(n,r)))[\delta^{-1}]$ is
an isomorphism.

{\it Step 7}. For $p\in \bar{\M}^\theta(n,r)^{T\times \C^\times}$ let $L^0(p)$ be the corresponding irreducible
$\Gamma(\Ca_\alpha(\bar{\A}^\theta_\lambda(n,r)))$-module from category $\mathcal{O}$. These modules are either finite dimensional
(those are parameterized by the multi-partitions with one part equal to $(n)$ and others empty) or has support of maximal
dimension. It follows from Proposition \ref{Prop:surject} that all finite dimensional $L^0(p)$ restrict
to pairwise non-isomorphic $\Ca_\alpha(\bar{\A}_\lambda(n,r))$-modules. Now consider $L^0(p)$ with support of maximal dimension.
We claim that the localizations $L^0(p)[\delta^{-1}]$ are pairwise non-isomorphic simple $\Gamma(\Ca_\alpha(\bar{\A}^\theta_\lambda(n,r)))[\delta^{-1}]$-modules.
Let us consider $p=(p^1,\ldots,p^r)$ and $p'=(p'^1,\ldots,p'^r)$ with $|p^i|=|p'^i|$ for all $i$ and show that the corresponding localizations are simple and, moreover, are isomorphic only if $p=p'$. This claim holds if we localize to the regular locus for $\prod_{i=1}^r\mathfrak{S}_{|p^i|}$.
Indeed, this localization realizes the KZ functor that is a quotient onto its image. So the images of 
$L^0(p), L^0(p')$ under this localization are simple and non-isomorphic. 
Then we further restrict the localizations of $L^0(p),L^0(p')$ to the locus
where $x_i\neq x_j$ for all $i,j$. But there is no monodromy of the D-modules $L^0(p)[\delta^{-1}],L^0(p')[\delta^{-1}]$
along those additional hyperplanes and these D-modules  have regular singularities everywhere.
It follows that they remain simple and nonisomorphic (if $p\neq p'$).

{\it Step 8}. So we see that the $\Ca_\alpha(\bar{\A}_\lambda(n,r))[\delta^{-1}]$-modules  $L^0(p)[\delta^{-1}]$ are simple
and pair-wise non-isomorphic. The $\Ca_\alpha(\bar{\A}_\lambda(n,r))$-module $L^0(p)$ is not finitely generated a priori
but always lies in the ind-completion of the category $\mathcal{O}$ (thanks to the weight decomposition).
Pick a finitely generated $\Ca_\alpha(\A_\lambda(n,r))$-lattice $L^0_1(p)$ for $L^0(p)[\delta^{-1}]$ inside $L^0(p)$.
This now an object in the category $\mathcal{O}$. There is a simple constituent $\underline{L}^0(p)$
of $L^0_1(p)$ with $\underline{L}^0(p)[\delta^{-1}]=L^0(p)[\delta^{-1}]$ because the right hand side
is simple. The finite dimensional modules $L^0(p)$ together with the modules of the form $\underline{L}^0(p)$
give a required number of pairwise nonisomorphic simple $\A_\lambda(n,r)^0$-modules.
\end{proof}

\subsection{Completion of proofs}\label{SS_loc_compl}
The following proposition completes the proof of Theorem \ref{Thm:loc}.

\begin{Prop}\label{Prop:loc_compl}
Let $\lambda$ be a positive parameter with denominator $n$. Then the abelian localization holds for
$(\lambda,\det)$.
\end{Prop}
\begin{proof}
Let $\alpha$ be the one-parameter subgroup $t\mapsto (t^{d_1},\ldots, t^{d_r})$ with $d_1\gg\ldots\gg d_r$.
Let $\beta: \C^\times\rightarrow T\times \C^\times$ have the form $t\mapsto (1,t)$. Set $\alpha'=m\alpha+\beta$
for $m\gg 0$. So we have $\alpha\prec^\lambda \alpha'$ for all $\lambda$ thanks to Lemma \ref{Lem:prec_spec}.

 Since $\Gamma_\lambda^\theta: \mathcal{O}_{\alpha'}(\bar{\A}^\theta_\lambda(n,r))
\rightarrow \mathcal{O}_{\alpha'}(\bar{\A}_\lambda(n,r))$ is a quotient functor, to prove that it is an equivalence
it is enough to verify that the number of simples in these two categories is the same. The number of
simples in $\OCat_{\alpha'}(\bar{\A}_\lambda(n,r))$ coincides with that for $\OCat(\Ca_\alpha(\bar{\A}_\lambda(n,r)))$
thanks to Lemma \ref{Lem:parab_ind}.
The latter is bigger than or equal to the number of simples for $\OCat(\bigoplus \bar{\A}_\lambda(n_1,\ldots,n_r;r))$ that, in its
turn coincides with the number of the $r$-multipartitions of $n$ because the abelian localization holds
for all summands $\bar{\A}_\lambda(n_1,\ldots,n_r;r)$. We deduce that the number of simples
in $\mathcal{O}_{\alpha'}(\bar{\A}^\theta_\lambda(n,r))$ and in $\mathcal{O}_{\alpha'}(\bar{\A}_\lambda(n,r))$ coincide.
So we see that $\Gamma^\theta_\lambda:\OCat_{\alpha'}(\bar{\A}^\theta_\lambda(n,r))\twoheadrightarrow \OCat_{\alpha'}(\bar{\A}_\lambda(n,r))$
is an equivalence. Now we are going to show that this implies that $\Gamma^\theta_\lambda:\bar{\A}_\lambda^\theta(n,r)\operatorname{-mod}
\rightarrow \bar{\A}_\lambda(n,r)\operatorname{-mod}$ is an equivalence. Below we write $\OCat$
instead of $\OCat_{\alpha'}$.

Since $\Gamma_\lambda^\theta $ is an equivalence between the categories $\OCat$, we see that $\bar{\A}^{(\det)}_{\lambda,\chi}(n,r)\otimes_{\bar{\A}_\lambda(n,r)}\bullet$
and $\bar{\A}^{(\det)}_{\lambda+\chi,-\chi}(n,r)\otimes_{\bar{\A}_{\lambda+\chi}(n,r)}\bullet$ are mutually
inverse equivalences between $\OCat(\bar{\A}_\lambda(n,r))$ and $\OCat(\bar{\A}_{\lambda+\chi}(n,r))$ for $\chi\gg 0$.
Set $\mathcal{B}:=\bar{\A}^{(\det)}_{\lambda+\chi,-\chi}(n,r)\otimes_{\bar{\A}_{\lambda+\chi}(n,r)}\bar{\A}^{(\det)}_{\lambda,\chi}(n,r)$.
This is a HC $\bar{\A}_\lambda(n,r)$-bimodule with a natural homomorphism to $\bar{\A}_\lambda(n,r)$ such that
the induced homomorphism $\mathcal{B}\otimes_{\bar{\A}_\lambda(n,r)}M\rightarrow M$ is an isomorphism
for any $M\in \OCat(\bar{\A}_\lambda(n,r))$.
It follows from  \cite[Proposition 5.15]{BL} that the kernel and the cokernel of $\mathcal{B}\rightarrow \bar{\A}_\lambda(n,r)$ have proper associated varieties and hence
are finite dimensional. Let $L$ denote an irreducible finite dimensional $\bar{\A}_\lambda(n,r)$-module, it is unique
because of the equivalence $\OCat(\bar{\A}_\lambda(n,r))\cong \OCat(\bar{\A}_{\lambda+\chi}(n,r))$. Since the homomorphism
$\mathcal{B}\otimes_{\bar{\A}_\lambda(n,r)}L\rightarrow L$ is an isomorphism,  we see that $\mathcal{B}\twoheadrightarrow \bar{\A}_\lambda(n,r)$. Let $K$ denote the kernel.  We have an exact sequence
$$\operatorname{Tor}^1_{\bar{\A}_\lambda(n,r)}(\bar{\A}_\lambda(n,r),L)\rightarrow K\otimes_{\bar{\A}_\lambda(n,r)}L
\rightarrow \mathcal{B}\otimes_{\bar{\A}_\lambda(n,r)}L\rightarrow L\rightarrow 0$$
Clearly, the first term is zero, while the last homomorphism is an isomorphism. We deduce that $K\otimes_{\bar{\A}_\lambda(n,r)}L=0$.
But $K$ is a finite dimensional $\bar{\A}_\lambda(n,r)$-bimodule and hence a $\bar{\A}_\lambda(n,r)/\operatorname{Ann}L$-bimodule and so its tensor product with $L$ can only be zero if $K=0$.

So we see that $\bar{\A}^{(\det)}_{\lambda+\chi,-\chi}(n,r)\otimes_{\bar{\A}_{\lambda+\chi}(n,r)}\bar{\A}^{(\det)}_{\lambda,\chi}(n,r)\cong
\bar{\A}_{\lambda}(n,r)$. Similarly, $\bar{\A}^{(\det)}_{\lambda,\chi}(n,r)\otimes_{\bar{\A}_{\lambda}(n,r)}\bar{\A}^{(\det)}_{\lambda+\chi,-\chi}(n,r)\cong
\bar{\A}_{\lambda+\chi}(n,r)$. It follows that $\Gamma_\lambda^\theta$ is an equivalence  $\bar{\A}_\lambda^\theta(n,r)\operatorname{-mod}\cong \bar{\A}_\lambda(n,r)\operatorname{-mod}$.
\end{proof}

Now we can complete the proof of (2) of Theorem \ref{Thm:fin dim}. It remains to show that $\bar{\A}_\lambda(n,r)$
with $-r<\lambda<0$ has no finite dimensional irreducible representations. Assume the converse, let $L$
denote a finite dimensional irreducible representation. Since $L\Loc_\lambda^\theta(\bar{\A}_\lambda(n,r))=\bar{\A}_\lambda^\theta(n,r)$
and $R\Gamma_\lambda^\theta(\bar{\A}^\theta_\lambda(n,r))=\bar{\A}_\lambda(n,r)$, we see that
$R\Gamma_\lambda^\theta\circ L\Loc_\lambda^\theta$ is the identity functor of $D^-(\bar{\A}_\lambda(n,r)\operatorname{-mod})$. The homology of $L\Loc_\lambda^\theta(L)$ are supported
on $\bar{\rho}^{-1}(0)$. It follows that the denominator of $\lambda$ is $n$.

Recall that $\Gamma_\lambda^\theta$ is an exact functor. Since $R\Gamma_\lambda^\theta\circ L\Loc_\lambda^\theta$
is the identity, the functor $\Gamma_\lambda^\theta$ does not kill
the simple $\bar{\A}_\lambda^\theta(n,r)$-module $\tilde{L}$ supported on $\bar{\rho}^{-1}(0)$.
On the other hand, $\Gamma_\lambda^\theta$ does not kill modules whose support intersects $\bar{\M}^\theta(n,r)^{reg}$,
the open subvariety in $\bar{\M}^\theta(n,r)$, where $\bar{\rho}$ is an isomorphism. In fact, every simple
in $\OCat(\bar{\A}_\lambda^\theta(n,r))$ is either supported on $\bar{\rho}^{-1}(0)$ or its support intersects
$\bar{\M}^\theta(n,r)^{reg}$. This is true when $\lambda'$ has denominator $n$ and satisfies the abelian localization theorem.
Indeed,  every module in $\OCat(\bar{\A}_{\lambda'}(n,r))$ is strictly holonomic. So if it has support of dimension
$rn-1$, then this support intersects regular locus, if not, the module is finite dimensional.    Our claim
about $\bar{\A}_\lambda^\theta(n,r)$-modules follows.

So we see that $\Gamma_\lambda^\theta$ does not kill any irreducible module in $\OCat(\bar{\A}^\theta_\lambda(n,r))$.
So it is an equivalence. However, the proof of Proposition \ref{Prop:loc_compl} shows that this is impossible.
This completes the proof of (2) of Theorem \ref{Thm:fin dim}.

\section{Affine wall-crossing and counting}\label{S_aff_wc_count}
The main goal of this section is to prove Theorem \ref{Thm:counting}. As in \cite[Section 8]{BL}, the proof follows
from the claim that the wall-crossing functor through the wall $\delta=0$ (the affine wall-crossing functor)
is a perverse equivalence with homological shifts less than $\dim \M^\theta(v,w)$. As was pointed out in \cite[9.2]{BL},
this follows from results that we have already proved and the following claims yet to be proved.
\begin{itemize}
\item[(i)] Let $\Leaf$ be a symplectic leaf in $\M(v,w)$. Consider the categories
$\HC_{fin}(\underline{\bar{\A}}_{\hat{\param}}(\underline{v},\underline{w}))$ of HC bimodules
over the corresponding slice algebra $\underline{\bar{\A}}_{\hat{\param}}(\underline{v},\underline{w})$
that are finitely generated (left and  right) modules over $\C[\hat{\param}]$
and $\HC_{\overline{\Leaf}}(\A_{\hat{\param}}(v,w))\subset \HC_{\Leaf\cup Y}(\A_{\hat{\param}}(v,w))$
of Harish-Chandra bimodules supported on $\overline{\Leaf}$ and on $\Leaf\cup Y$, where we write $Y$
for the union of all leaves that do not contain $\Leaf$ in their closure. Then, for $x\in \Leaf$, there is
a functor $\bullet^{\dagger,x}: \HC_{fin}(\underline{\bar{\A}}_{\hat{\param}}(\underline{v},\underline{w}))
\rightarrow \HC_{\overline{\Leaf}}(\A_{\hat{\param}}(v,w))$ that is right adjoint to
$\bullet_{\dagger,x}:\HC_{\Leaf\cup Y}(\A_{\hat{\param}}(v,w))\rightarrow \HC_{fin}(\underline{\bar{\A}}_{\hat{\param}}(\underline{v},\underline{w}))$.
\item[(ii)] Theorem \ref{Thm:ideals} holds together with a direct analog of \cite[Lemma 5.21]{BL}.
\item[(iii)] For a unique proper ideal $\J\subset \bar{\A}_\lambda(n,r)$, where $\lambda$ has denominator
$n$ and lies outside $(-r,0)$, we have $\operatorname{Tor}^i_{\bar{\A}_\lambda(n,r)}(\bar{\A}_\lambda(n,r)/\J,
\bar{\A}_\lambda(n,r)/\J)=\bar{\A}_\lambda(n,r)/\J$ if $i$ is even, between $0$ and $2nr-2$, and 0 otherwise.
\item[(iv)] The functor $\bullet_{\dagger,x}$ is faithful for $x$ and $\lambda$ specified below.
\end{itemize}
We take a Weil generic $\lambda$  on the hyperplane of the form $\langle \delta,\cdot\rangle=\kappa$, where
$\kappa$ is a fixed rational number with denominator $n'$. The choice of $x$ is
as follows. Recall the description of the symplectic leaves of $\M(v,w)$ in Subsection \ref{SS_leaves}.
We want $x$ to lie in the leaf corresponding to the decomposition $r^0\oplus (r^1)^{\oplus n'}
\oplus (r^2)^{\oplus n'}\oplus\ldots\oplus (r^q)^{n'}\oplus r^{q+1}\oplus\ldots\oplus r^{n-q(n'-1)}$, where $\dim r^i=\delta, i=1,\ldots,q,q=\lfloor n/n'\rfloor$
and $n$ is given by
$$n:=\lfloor \frac{w\cdot v-(v,v)/2}{w\cdot \delta}\rfloor$$
so that $n$ is maximal with the property that $v^0=v-n\delta$ is a root of $Q^w$. We remark that the slice to $x$
is $\M(n',r)^{q}$, where $r$ is given by $r:=w\cdot \delta$.

In the proof, it is enough to assume that $\nu$ is dominant. If this is not true, we can find an element
$\sigma$ of $W(Q)$ such that $\sigma\nu$ is dominant, then we have a quantum LMN isomorphism $\A_\lambda(v,w)\cong
\A_{\sigma\cdot \lambda}(\sigma\cdot v, w)$ and  $\langle\lambda,\delta\rangle= \langle\sigma\cdot\lambda, \delta\rangle$.

\subsection{Functor $\bullet^{\dagger,x}$}\label{S_fun_up_dag}
In this subsection we establish (i) above in greater generality.

Here we assume that $X\rightarrow X_0$ is a conical symplectic resolution  whose all slices are conical
and satisfy the additional assumption on  contracting $\C^\times$-actions from Subsection \ref{SS_HC_bimod}.
Recall that under these assumptions, for a point $x\in X_0$, we can define the exact functor  $\bullet_{\dagger,x}:
\HC(\A_\param)\rightarrow \HC(\underline{\A}_\param)$, where we write $\param$ for $H^2(X,\C)$. In this subsection
we are going to study its adjoint $\bullet^{\dagger,x}:\HC(\underline{\A}_\param)\rightarrow \widetilde{\HC}(\A_\param)$.
Here $\widetilde{\HC}(\A_\param)$ denotes the category of $\A_\param$-bimodules that are sums of their HC subbimodules.
The construction of the functor is similar to \cite{HC,sraco,W_dim}. Namely, we pick a HC $\underline{\A}_{\param}$-bimodule
$\mathcal{N}$, form the Rees bimodule $\mathcal{N}_\hbar$ and its completion $\mathcal{N}_\hbar^{\wedge_0}$ at $0$.
Then form the $\A_{\param,\hbar}$-bimodule $\M_\hbar$ that is the sum of all HC
subbimodules in $\Weyl_\hbar^{\wedge_0}\widehat{\otimes}_{\C[[\hbar]]}\mathcal{N}_\hbar^{\wedge_0}$
(this includes the condition that the Euler derivation acts locally finitely).
It is easy to see that if $\hbar m\in  \M_\hbar$, then
$m\in \M_\hbar$. Then we set $\mathcal{N}^{\dagger,x}:=\mathcal{M}_\hbar/(\hbar-1)\mathcal{M}_\hbar$. Similarly
to \cite[4.1.4]{W_dim}, we see that $\mathcal{N}\mapsto \mathcal{N}^{\dagger,x}$ is a functor and that
$\operatorname{Hom}(\M,\mathcal{N}^{\dagger,x})=\operatorname{Hom}(\M_{\dagger,x},\mathcal{N})$.

Our main result about the functor $\bullet^{\dagger,x}$ is the following claim.

\begin{Prop}\label{Prop:up_dag_prop}
If $\mathcal{N}$ is finitely generated over $\C[\param]$, then $\mathcal{N}^{\dagger,x}\in \HC(\A_\param)$ and $\VA(\mathcal{N}^{\dagger,x})=\overline{\Leaf}$. Here $\Leaf$ stands for the symplectic leaf
of $X$ containing $x$.
\end{Prop}
\begin{proof}
The proof is similar to analogous proofs in \cite[3.3,3.4]{HC},\cite[3.7]{sraco}.  As in those proofs, it is enough to
show that if $\mathcal{M}$ is a  Poisson $\C[\Leaf]^{\wedge_x}$-module of finite rank equipped with an Euler
derivation, then the maximal Poisson $\C[\Leaf]$-submodule of $\mathcal{M}$ that is the sum of its finitely
generated Poisson $\C[\overline{\Leaf}]$-submodules (with locally finite action of the  Euler derivation) is
finitely generated. By an Euler derivation, we mean an endomorphism $\mathsf{eu}$ of $\mathcal{M}$ such that
\begin{itemize}
\item $\mathsf{eu}(am)=(\mathsf{eu}a)m+a(\mathsf{eu}m)$,
\item $\mathsf{eu}\{a,m\}=\{\mathsf{eu}a, m\}+\{a, \mathsf{eu}m\}-d\{a,m\}$.
\end{itemize}
Here by $\mathsf{eu}$ on $\C[\Leaf]^{\wedge_x}$ we mean the derivation induced by the contracting $\C^\times$-action.

{\it Step 1}. First, according to Namikawa, \cite{Namikawa_fund}, the algebraic fundamental group $\pi^{alg}_1(\Leaf)$ is  finite.
Let $\widetilde{\Leaf}$ be the corresponding Galois covering of $\Leaf$. Being the integral closure of
$\C[\overline{\Leaf}]$ in a finite extension of $\C(\Leaf)$, the algebra $\C[\widetilde{\Leaf}]$
is finite over $\C[\overline{\Leaf}]$. The group $\pi_1(\widetilde{\Leaf})$ has no homomorphisms to $\operatorname{GL}_m$
by the choice of $\widetilde{\Leaf}$. Also let us note that the $\C^\times$-action on $\Leaf$
lifts to a $\C^\times$-action on $\widetilde{\Leaf}$ possibly after replacing $\C^\times$ with
some covering torus. We remark that the action produces a positive grading on $\C[\widetilde{\Leaf}]$.

{\it Step 2}. Let $\mathcal{V}$ be a weakly $\C^\times$-equivariant $D_{\widetilde{\Leaf}}$-module.
We claim that $\mathcal{V}$ is the sum of several copies of $\mathcal{O}_{\widetilde{\Leaf}}$. Indeed,
this is so in the analytic category: $\mathcal{V}^{an}:=\mathcal{O}^{an}_{\widetilde{\Leaf}}\otimes_{\mathcal{O}_{\widetilde{\Leaf}}}\mathcal{V}\cong \mathcal{O}_{\widetilde{\Leaf}}^{an}\otimes\mathcal{V}^{fl}$ (where the superscript ``fl'' means flat sections)
because of the assumption on $\pi_1(\widetilde{\Leaf})$.
But then the space $\mathcal{V}^{fl}$ carries a holomorphic $\C^\times$-action that has to be diagonalizable
and by characters. So we have an embedding $\Gamma(\mathcal{V})\hookrightarrow \Gamma(\mathcal{V}^{an})^{\C^\times-fin}$.
Since $\operatorname{Spec}(\C[\widetilde{\Leaf}])$ is normal, any analytic function on
$\widetilde{\Leaf}$ extends to $\operatorname{Spec}(\C[\widetilde{\Leaf}])$. Since the grading on $\C[\widetilde{\Leaf}]$ is positive, any holomorphic $\C^\times$-semiinvariant function must be polynomial. So the embedding above
reduces to $\Gamma(\mathcal{V})\hookrightarrow \C[\widetilde{\Leaf}]\otimes \mathcal{V}^{fl}$.
The generic rank of $\Gamma(\mathcal{V})$ coincides with $\dim \mathcal{V}^{fl}$. Since the module
$\C[\widetilde{\Leaf}]\otimes \mathcal{V}^{fl}$ has no torsion, we see that $\Gamma(\mathcal{V})=
\C[\widetilde{\Leaf}]\otimes \mathcal{V}^{fl}$. It follows that $\mathcal{V}=\mathcal{O}_{\widetilde{\Leaf}}\otimes \mathcal{V}^{fl}$ as a D-module and the claim of this step follows.

{\it Step 3}. Let $Y$ be a symplectic variety.  We claim that a Poisson $\mathcal{O}_{Y}$-module
carries a canonical structure of a $D_{Y}$-module and vice versa. If $\mathcal{N}$ is a $D_Y$-module,
then we equip it with the structure of a Poisson module via $\{f,n\}:=v(f)n$. Here $f,n$ are local
sections of $\mathcal{O}_Y, \mathcal{N}$, respectively, and $v(f)$ is the skew-gradient of $f$,
a vector field on $Y$. Let us now equip a Poisson module with a canonical D-module structure.
It is enough to do this locally, so we may assume that there is an etale map $Y\rightarrow \C^k$.
Let $f_1,\ldots,f_k$ be the corresponding etale coordinates. Then we set $v(f)n:=\{f,n\}$. This defines
a D-module structure on $\mathcal{N}$ that is easily seen to be independent of the choice of an
\'{e}tale chart.

Let us remark that a weakly $\C^\times$-equivariant Poisson module gives rise to a weakly $\C^\times$-equivariant
D-module and vice versa.

So the conclusion of the previous 3 steps is that every weakly $\C^\times$-equivariant Poisson $\mathcal{O}_{\widetilde{\Leaf}}$-module
is the direct sum of several copies of $\mathcal{O}_{\widetilde{\Leaf}}$.

{\it Step 4}. Pick a point $\tilde{x}\in \widetilde{\Leaf}$ lying over $x$ so that $\widetilde{\Leaf}^{\wedge_{\tilde{x}}}$
is naturally identified with $\Leaf^{\wedge_x}$. Of course, any Poisson module over $\widetilde{\Leaf}^{\wedge_{\tilde{x}}}$
is the direct sum of several copies of $\C[\widetilde{\Leaf}]^{\wedge_{\tilde{x}}}$.  So the claim in the beginning
of the proof will follow if we check that any finitely generated Poisson $\C[\widetilde{\Leaf}]$-bimodule in
$\C[\widetilde{\Leaf}]^{\wedge_{\tilde{x}}}$ with locally finite $\mathsf{eu}$-action coincides with
$\C[\widetilde{\Leaf}]$. For this, let us note that the Poisson center of $\C[\widetilde{\Leaf}]^{\wedge_{\tilde{x}}}$
coincides with $\C$. On the other hand, any finitely generated Poisson submodule with locally finite
action of $\mathsf{eu}$ is the sum of  weakly $\C^\times$-equivariant Poisson submodules. The latter
have to be trivial and so are generated by the Poisson central elements. This implies our claim and completes the proof.
\end{proof}

We will also need to consider a map between the sets of two-sided ideals $\mathfrak{Id}(\underline{\A})\rightarrow
\mathfrak{Id}(\A)$ induced by the functor $\bullet^{\dagger,x}$, compare to \cite{wquant,HC,sraco}. Namely, for
$\mathcal{I}\in \mathfrak{Id}(\underline{\A})$ we write $\mathcal{I}^{\dagger_\A,x}$ for the kernel of the natural
map $\A\rightarrow (\underline{\A}/\mathcal{I})^{\dagger,x}$. Alternatively, the ideal $\mathcal{I}^{\dagger_\A,x}$
can be obtained as follows. Consider the ideal $\W_\hbar^{\wedge_0}\widehat{\otimes}_{\C[[\hbar]]}\mathcal{I}^{\wedge_0}_{\hbar}\subset
\A_\hbar^{\wedge_x}$. Set $\J_\hbar:=\W_\hbar^{\wedge_0}\widehat{\otimes}_{\C[[\hbar]]}\mathcal{I}^{\wedge_0}_{\hbar}\cap \A_\hbar$.
This is a $\C^\times$-stable $\hbar$-saturated ideal in $\A_\hbar$ and we set $\I^{\dagger_\A,x}:=\J_{\hbar}/(\hbar-1)\J_{\hbar}$.

We will need some properties of the map $\mathfrak{Id}(\underline{\A})\rightarrow \mathfrak{Id}(\A)$ analogous to those established
in \cite[Theorem 1.2.2]{wquant}.

\begin{Prop}\label{Prop:map_ideal_prop}
The following is true.
\begin{enumerate}
\item $\J\subset (\J_{\dagger,x})^{\dagger_\A,x}$ for all $\J\in \mathfrak{Id}(\A)$ and $(\I^{\dagger_\A,x})_{\dagger,x}\subset \I$
for all $\I\in \mathfrak{Id}(\underline{\A})$.
\item We have $\mathcal{I}_1^{\dagger_\A,x}\cap \mathcal{I}_2^{\dagger_\A,x}=(\mathcal{I}_1\cap \mathcal{I}_2)^{\dagger_\A,x}$.
\item If $\mathcal{I}$ is prime, then so is $\mathcal{I}^{\dagger_\A,x}$.
\end{enumerate}
\end{Prop}
\begin{proof}
(1) and (2) follow from the alternative definition of $\I^{\dagger_\A,x}$ given above. The proof of (3) closely follows that of
an analogous statement, \cite[Theorem 1.2.2,(iv)]{wquant}, let us provide a proof for readers convenience. It is easy to
see that the ideals $\I_{\hbar},\I_\hbar^{\wedge_0}, \W^{\wedge_0}\widehat{\otimes}_{\C[[\hbar]]}\mathcal{I}^{\wedge_0}_{\hbar}$
are prime because of the bijections between the sets of two-sided ideals in $\underline{\A}, \underline{\A}_\hbar,
\underline{\A}_\hbar^{\wedge_0}, \W_\hbar^{\wedge_0}\widehat{\otimes}_{\C[[\hbar]]}\underline{\A}^{\wedge_0}_\hbar$
(we only consider the $\C^\times$-stable $\hbar$-saturated ideals in the last three algebras).

So we need to show that the intersection $\J_\hbar$ of a $\C^\times$-stable $\hbar$-saturated prime ideal $\mathcal{I}'_\hbar\subset
\A^{\wedge_x}_\hbar$ with $\A_\hbar$ is prime. Assume the converse, let there exist ideals $\J^1_\hbar,\J^2_\hbar\supsetneq \J_\hbar$
such that $\J^1_\hbar \J^2_\hbar\subset \J_\hbar$. We may assume that both $\J^i_\hbar$ are $\C^\times$-stable and $\hbar$-saturated.
Indeed, if they are not $\hbar$-saturated, then we can saturate them. To see that they can be taken $\C^\times$-stable
one can argue as follows. The radical of $\J_\hbar$ is $\C^\times$-stable and so we can take appropriate powers of the
radical for $\J^1_\hbar,\J^2_{\hbar}$ if $\J_\hbar$ is not semiprime. If $\J_\hbar$ is semiprime, then
its associated prime ideals are $\C^\times$-stable and we can take their appropriate intersections for $\J^1_\hbar,\J^2_\hbar$.

So let us assume that $\J^1_\hbar,\J^2_\hbar$ are $\hbar$-saturated and $\C^\times$-stable. Then so are $(\J^1_\hbar)^{\wedge_x},(\J^2_\hbar)^{\wedge_x}$. Also let us remark that   $(\J^1_\hbar)^{\wedge_x}(\J^2_\hbar)^{\wedge_x}=
(\J^1_\hbar \J^2_\hbar)^{\wedge_x}\subset \J_\hbar^{\wedge_x}\subset \I'_\hbar$. Without loss of generality,
we may assume that $(\J^1_\hbar)^{\wedge_x}\subset \I'_\hbar$. It follows that $\J^1_\hbar\subset \J_\hbar=\A_\hbar\cap \I'_\hbar$,
and we are done.
\end{proof}

\subsection{Two-sided ideals in $\bar{\A}_\lambda(n,r)$}\label{SS_two_sid_id}
The goal of this subsection is to prove Theorem \ref{Thm:ideals} and more technical statements in (ii) above.
We use the following notation. We write $\A$ for $\bar{\A}_\lambda(n,r)$ and write $\underline{\A}$
for $\bar{\A}_\lambda(n',r)$, where $n'$ is the denominator of $\lambda$.

Let us start with the description of the two-sided ideals in $\underline{\A}$.

\begin{Lem}\label{Lem:ideals_easy}
There is a unique
proper ideal in $\underline{\A}$.
\end{Lem}
\begin{proof}
We have seen in the proof of Theorem \ref{Thm:fin dim} that the proper slice algebras for $\underline{\A}$ have no finite dimensional representations. So every ideal $\J\subset \underline{\A}$ is either of finite codimension or $\VA(\underline{\A}/\J)=\bar{\M}(n',r)$.
The algebra $\underline{\A}$ has no zero divisors so the second option is only possible when $\J=\{0\}$. Now suppose that $\J$
is of finite codimension. Then $\underline{\A}/\J$ (viewed as a left $\underline{\A}$-module) is the sum of several copies of the finite dimensional irreducible $\underline{\A}$-module. So $\J$ coincides with the annihilator of the finite dimensional irreducible module, and we are done.
\end{proof}

Let $\underline{\J}$ denote the unique two-sided ideal.

Now we are going to describe the two-sided ideals in $\underline{\A}^{\otimes k}$. For this we need some notation. Set
$\underline{\I}_i:=\underline{\A}^{\otimes i-1}\otimes \underline{\J}\otimes \underline{\A}^{\otimes k-i-1}$. For a subset
$\Lambda\subset \{1,\ldots,k\}$ define the ideals $\underline{\I}_{\Lambda}:=\sum_{i\in \Lambda} \underline{\I}_i,
\underline{\I}^{\Lambda}:=\prod_{i\in \Lambda} \underline{\I}_i$.

Recall that  a collection of subsets in $\{1,\ldots,k\}$ is called an {\it anti-chain} if none of these subsets is contained
in another. Also recall that an ideal $I$ in an associative algebra $A$ is called {\it semi-prime} if it is the intersection of
prime ideals.

\begin{Lem}\label{Lem:ideals_next}
The following is true.
\begin{enumerate}
\item The prime ideals in $\underline{\A}^{\otimes k}$ are precisely the ideals $\underline{\I}_\Lambda$.
\item For every ideal $\I\subset \underline{\A}^{\otimes k}$, there is a unique anti-chain $\Lambda_1,\ldots,\Lambda_q$
of subsets in $\{1,\ldots,k\}$ such that $\I=\bigcap_{i=1}^p \I_{\Lambda_i}$. In particular, every ideal is
semi-prime.
\item For every ideal $\I\subset \underline{\A}^{\otimes k}$, there is a unique anti-chain $\Lambda_1',\ldots,\Lambda_q'$
of subsets of $\{1,\ldots,k\}$ such that $\I=\sum_{i=1}^q \I^{\Lambda'_i}$.
\item The anti-chains in (2) and (3) are related as follows: from an antichain in (2), we form all possible subsets
containing an element from each of $\Lambda_1,\ldots,\Lambda_p$. Minimal such subsets form an anti-chain in (3).
\end{enumerate}
\end{Lem}
The proof essentially appeared in \cite[5.8]{sraco}.
\begin{proof}
Let us prove (1). Let $\I$ be a prime ideal. Let $x$ be a generic  point in an open leaf  $\Leaf\subset\VA(\underline{\A}^{\otimes k}/\I)$
of maximal dimension. The corresponding slice algebra $\underline{\A}'$ has a finite dimensional irreducible and so is again the product of several copies of $\underline{\A}$. The leaf $\Leaf$ is therefore the product of one-point leaves and full leaves in
$\bar{\M}(n',r)^k$.  An irreducible finite dimensional representation of $\underline{\A}'$ is unique, let $\I'$ be its annihilator.
Then $\I\subset \I'^{\dagger_{\underline{\A}^{\otimes k}},x}$. By Proposition \ref{Prop:up_dag_prop},
$\VA(\underline{\A}^{\otimes k}/\I'^{\dagger_{\underline{\A}^{\otimes k}},x})=\overline{\Leaf}$. It follows from
\cite[Corollar 3.6]{BoKr} that $\I=\I'^{\dagger_{\underline{\A}^{\otimes k}},x}$. So the number of the prime ideals coincides
with that of the non-empty subsets  $\{1,\ldots,k\}$. On the other hand, the ideals $\I^\Lambda$ are all different
(they have different associated varieties)
and all prime (the quotient $\underline{\A}^{\otimes k}/\I^{\Lambda}$ is the product of a matrix algebra and the algebra
$\underline{\A}^{\otimes k-|\Lambda|}$ that has no zero divisors).

Let us prove (2) (and simultaneously (3)). Let us write $\I_{\Lambda_1,\ldots,\Lambda_p}$ for $\bigcap_{j=1}^s \I_{\Lambda_j}$.
For ideals in $\underline{\A}^{\otimes k-1}$
we use notation like $\underline{\I}_{\Lambda_1',\ldots,\Lambda_q'}$. Reordering the indexes, we may assume that $k\in \Lambda_1,\ldots,
\Lambda_s$ and $k\not\in \Lambda_{s+1},\ldots,\Lambda_p$. Set $\Lambda_j':=\Lambda_j\setminus \{k\}$ for $j\leqslant s$.
Then
\begin{equation}\label{eq:ideal1}\I_{\Lambda_1,\ldots,\Lambda_p}=(\underline{\A}^{\otimes k-1}\otimes \underline{\J}+ \underline{\I}_{\Lambda'_1,\ldots,\Lambda'_s}\otimes \underline{\A})\cap (\underline{\I}_{\Lambda_{s+1},\ldots,\Lambda_p}\otimes \underline{\A}).\end{equation}
We claim that the right hand side of (\ref{eq:ideal1}) coincides with
\begin{equation}\label{eq:ideal2}
\underline{\I}_{\Lambda_{s+1},\ldots,\Lambda_p}\otimes \underline{\J}+\underline{\I}_{\Lambda_1',\ldots,\Lambda_s',\Lambda_{s+1},\ldots,\Lambda_p}\otimes \underline{\A}.
\end{equation}
First of all, we notice that (\ref{eq:ideal2}) is contained in (\ref{eq:ideal1}). So we only need to prove the
opposite inclusion.   The projection  of (\ref{eq:ideal1}) to $\underline{\A}^{\otimes k-1}\otimes
\underline{\A}/\underline{\J}$ is contained in $\underline{\I}_{\Lambda_1',\ldots,\Lambda_s',\Lambda_{s+1},\ldots,\Lambda_p}$
and hence also in the projection of (\ref{eq:ideal2}). Also the intersection of (\ref{eq:ideal1}) with $\underline{\A}^{\otimes k-1}\otimes \underline{\J}$ is contained in $\underline{\I}_{\Lambda_{s+1},\ldots,\Lambda_p}\otimes \underline{\J}$. So (\ref{eq:ideal1})
is included into (\ref{eq:ideal2}).

Repeating this argument with the two summands in (\ref{eq:ideal2}) and other factors of $\underline{\A}^{\otimes k}$
we conclude that $\I_{\Lambda_1,\ldots,\Lambda_p}=\sum_j \I^{\Lambda'_j}$, where the subsets
$\Lambda'_j\subset \{1,\ldots,k\}$ are formed as described in (4).   So we see that the ideals (2) are the same
as the ideals in (3) and that (4) holds. What remains to do is to prove that every ideal has the form described in (2).
To start with, we notice that every semi-prime ideal has the form as in (2) because of (1). In particular, the radical
of any ideal has such form.

Clearly,
$\I^{\Lambda'_1}\I^{\Lambda'_2}=\I^{\Lambda'_1\cup\Lambda'_2}$. So it follows any sum of the ideals $\I^{\Lambda_j'}$ coincides with its
square. So if $\I$ is an ideal whose radical is $\I_{\Lambda_1,\ldots,\Lambda_p}$, then $\I$ coincides with
its radical. This completes the proof.
\end{proof}

Now we are ready to establish a result that will imply Theorem \ref{Thm:ideals} together with technical results required in
(ii). Let $x_i\in \bar{\M}(n,r)$ be a point corresponding to the leaf with slice $\bar{\M}(n',r)^{i}$ (i.e. to the semisimple representations of the form $r^0\oplus (r^1)^{n'}\oplus\ldots\oplus (r^i)^{n'}$). We set
$\J_i:=\I^{\dagger_{\A}, x_i}$, where $\I$ is the maximal ideal in $\underline{\A}^{\otimes i}$, equivalently,
the annihilator of the finite dimensional irreducible representation.

\begin{Prop}\label{Prop:ideals_techn}
The ideals $\J_i, i=1,\ldots,q$, have the following properties.
\begin{enumerate}
\item The ideal $\J_i$ is prime for any $i$.
\item $\VA(\A/\J_i)=\overline{\Leaf}_i$, where $\Leaf_i$ is the symplectic leaf containing $x_i$.
\item $\J_1\subsetneq \J_2\subsetneq\ldots\subsetneq \J_q$.
\item Any proper two-sided ideal in $\A$ is one of $\J_i$.
\item We have $(\J_i)_{\dagger,x_j}=\underline{\A}^{\otimes j}$ if $j<i$ and $(\J_i)_{\dagger,x_j}=\sum_{|\Lambda|=j-i+1} \I^{\Lambda}$ else.
\end{enumerate}
\end{Prop}
\begin{proof}
(1) is a special case of (3) of Proposition \ref{Prop:map_ideal_prop}. (2) follows from Proposition \ref{Prop:up_dag_prop},
compare with the proof of (1) in Lemma \ref{Lem:ideals_next}.

Let us prove (3). Since $(\J_{i})_{\dagger,x_i}$ has finite codimension, we see that it coincides with the maximal ideal
in $\underline{\A}^{\otimes i}$. So $(\J_j)_{\dagger,x_i}\subset (\J_i)_{\dagger,x_i}$ for $j<i$. It follows that
$\J_j\subset [(\J_j)_{\dagger,x_i}]^{\dagger_\A,x_i}\subset [(\J_i)_{\dagger,x_i}]^{\dagger_\A,x_i}=\J_i$.

Let us prove (4). The functor $\bullet_{\dagger,x_q}$ is faithful. Indeed, otherwise we have a HC bimodule
$\M$ with $\VA(\M)\cap \Leaf_q=\varnothing$. But $\M_{\dagger,x}$ has to be nonzero finite dimensional
for some $x$ and this is only possible when $x\in \Leaf_i$ for some $i$. But $\Leaf_q\subset \overline{\Leaf}_i$
for all $i$ that shows faithfulness. Since $\bullet_{\dagger,x_q}$ is faithful and exact, it follows that it embeds the lattice of
the ideals in $\A$ into that in $\underline{\A}^{\otimes q}$. We claim that this implies that every ideal in
$\A$ is semiprime. Indeed, the functor $\bullet_{\dagger,x_q}$ is, in addition, tensor and so
preserves products of ideals. Our claim follows from (2) of Lemma \ref{Lem:ideals_next}. But every prime ideal
in $\A$ is some $\J_i$, this is proved analogously to (1) of Lemma \ref{Lem:ideals_next}. Since the ideals
$\J_i$ form a chain, any semiprime ideal is prime and so coincides with some $\J_i$.

Let us prove (5). We will deduce that from the behavior of $\bullet_{\dagger,x}$ on the associated varieties.
We have an action of $\mathfrak{S}_j$ on $\M(n',r)^j$ by permuting factors. The action is induced
from $N_{G}(G_{\tilde{x}})$, where $\tilde{x}$ is a point from the closed $G$-orbit lying over $x$.
It follows that the intersection of any leaf with the slice is $\mathfrak{S}_j$-stable. The associated
variety $\VA(\underline{\A}^{\otimes j}/(\J_i)_{\dagger,x_j})$ is the union of some products with factors
$\{\operatorname{pt}\}$ and $\bar{\M}(n',r)$, where, for the dimension reasons, $\bar{\M}(n',r)$
occurs $j-i$ times. Because of the $\mathfrak{S}_j$-symmetry, all products occur. Now we deduce
the required formula for $(\J_i)_{\dagger,x_j}$ from the description of the two-sided ideals
in $\bar{\A}^{\otimes j}$. This description shows that for each associated variety there is at most
one two-sided ideal.
\end{proof}

\subsection{Computation of Tor's}\label{SS_tor_comput}
Here we consider the case when the denominator of $\lambda$ is $n$ and $\lambda$ is regular. Set $\A:=\bar{\A}_\lambda(n,r)$
and let $\J$ denote a unique proper ideal in this algebra. We want to establish (iii).

\begin{Prop}\label{Prop:Tor}
We have $\operatorname{Tor}_i^{\A}(\A/\J,\A/\J)=\A/\J$ if $i$ is even and $0\leqslant i\leqslant 2rn-2$, and
$\operatorname{Tor}_i^{\A}(\A/\J,\A/\J)=0$ else.
\end{Prop}
The proof closely follows that of \cite[Lemma 7.4]{BL} but we need to modify some parts of that argument.
\begin{proof}
Thanks to the translation equivalences it is enough to prove the claim when $\lambda$ is Zariski generic.

Let $L$ denote a unique finite dimensional irreducible $\A$-module. What we need to show is that
$\operatorname{Tor}_i^{\A}(L^*,L)=\C$ if $i$ is even and $0\leqslant i\leqslant 2nr-2$ and that
the Tor vanishes otherwise. We claim that $\operatorname{Tor}_i^{\A}(L^*,L)=\operatorname{Ext}^i_{\A}(L,L)^*$.
Knowing that, one can argue as follows. By Lemma \ref{Lem:Ext_coinc}, $\operatorname{Ext}^i_{\A}(L,L)=
\operatorname{Ext}^i_{\OCat(\A)}(L,L)$. The block in $\OCat(\A)$ containing $L$ was described in
Theorem \ref{Thm:catO_str}. In this block, we have $\operatorname{Ext}^i(L,L)=\C$ when $i$ is
even, $0\leqslant i\leqslant 2nr-2$, and $\operatorname{Ext}^i(L,L)=0$ otherwise. To see this one considers
the so called BGG resolution, see \cite{BEG}, for the first copy of $L$ and its analog with costandard
objects for the second copy.

So we need to show that $\operatorname{Tor}_i^{\A}(L^*,L)^*=\operatorname{Ext}^i_{\A}(L,L)$.
The proof is similar to that of \cite[Theorem A.1]{BLPW}. Namely, we consider the objects
$\Delta:=\A/\A\A_{>0}, \nabla^\vee:=\A/\A_{<0}\A$ and let $\nabla$ be the restricted dual
of $\nabla^\vee$. Then, since $\lambda$ is Zariski generic, we see that $\Delta$ is the sum
of the standard objects in $\OCat(\A)$, while $\nabla$ is the sum of all costandard objects
in $\OCat(\A)$. Then, as we have checked in the appendix to \cite{BLPW}, we have
$\operatorname{Tor}^{\A}_i(\nabla^\vee,\Delta)=\Ext_{\A}^i(\Delta,\nabla)=0$ for $i>0$ and moreover
$(\nabla^\vee\otimes_{\A}\Delta)^*=\Hom_{\A}(\Delta,\nabla)$ (an equality of $\C \bar{\M}^\theta(n,r)^{T\times\C^\times}$-bimodules).
Taking a resolution $P$ of the first copy of $L$ in  $\operatorname{Ext}^i_{\A}(L,L)$ by  direct summands
of $\Delta$ (which exists because of the structure of $\OCat(\A)$) and of the second copy by
direct summands of $\nabla$, denote this resolution by $Q$, we get
$\operatorname{Ext}^i_{\A}(L,L)=H_i(\operatorname{Hom}_{\A}(P,Q))=H_i(Q^\vee\otimes_{\A}P)^*=\operatorname{Tor}_i^{\A}(Q^\vee,P)^*=
\operatorname{Tor}^i_{\A}(L^*,L)^*$.
\end{proof}

\subsection{Faithfulness}\label{SS_faith}
Now we are going to establish (iv) for $x$ and $\lambda$ specified above. Let $\Leaf$ be the leaf containing $x$
and $\Leaf'$ be the leaf corresponding to the decomposition $r^0\oplus (r^1)^{\oplus n}$, where $r^1$ is an irreducible
representation of dimension $\delta$. Clearly, $\Leaf'\subset \overline{\Leaf}$.

We are going to prove that $\Leaf$ is contained in the associated variety of any HC $\A_\lambda(v,w)$-bimodule $\M$
(or $\A_{\lambda'}(v,w)$-$\A_\lambda(v,w)$-bimodule or $\A_\lambda(v,w)$-$\A_{\lambda'}(v,w)$-bimodule; thanks
to Proposition \ref{Prop:up_dag_prop}, a direct analog of \cite[Corollary 5.19]{BL} holds so that the associated
variety of any HC bimodule $\A_{\lambda'}(v,w)$-$\A_\lambda(v,w)$-bimodule coincides with those of
$\A_{\lambda'}(v,w)/\J_\ell, \A_\lambda(v,w)/\J_{r}$, where $\J_{\ell},\J_r$ are the left and right annihilators).
This is equivalent to saying that $\bullet_{\dagger,x}$ is faithful.
The scheme of the proof is as follows:
\begin{itemize}
\item[(a)] We first show that $\Leaf'\subset \VA(\M)$.
We do this by showing that $\M_{\dagger,y}$ cannot be finite dimensional nonzero for $y$ from a leaf $\Leaf_0$
such that $\Leaf'\not\subset\overline{\Leaf}_0$.
\item[(b)] From $\Leaf'\subset \VA(\M)$ we deduce that $\Leaf\subset \VA(\M)$.
\end{itemize}

Let us deal with (a). As in the proof of \cite[Lemma 7.10]{BL}, it is enough to show that the slice algebra
$\underline{\A}$ corresponding to $y$ has no finite dimensional representations. So let us analyze the structure of the leaves that contain $\Leaf'$ in their closure. For a partition $\mu=(\mu_1,\ldots,\mu_k)$ with $|\mu|<n$, we write $\Leaf(\mu)$ for the leaf corresponding
to the decomposition of the form $r^0\oplus (r^1)^{\oplus \mu_1}\oplus\ldots\oplus (r^k)^{\oplus\mu_k}$. From \cite[Theorem 1.2]{CB_geom} it follows that $\Leaf'=\Leaf(n)\subset
\overline{\Leaf(\mu)}$. There is a natural surjection from the set of leaves in $\bar{\M}(n,r)$ (this is precisely the slice to
$\Leaf'$) to the set of leaves in $\M(v,w)$ whose closure contains $\Leaf'$.


As in the proof of \cite[Lemma 7.10,(1)]{BL}, we need to prove that, for a generic $p\in \ker\delta$, the variety
$\underline{\bar{\M}}_p(\underline{v},\underline{w})$ has no single point symplectic leaves as long as
the corresponding symplectic leaf $\Leaf$ is different from $\Leaf(\mu)$. The latter is equivalent
to the condition that the decomposition of $v$ defining $\Leaf$ contains real roots. So let this decomposition
be $v=v^0+\nu_1\delta+\ldots+\nu_\ell\delta+\sum_{i\in Q_0}m_i\alpha_i$. The slice variety
$\underline{\bar{\M}}_p(\underline{v},\underline{w})$ is the product $\prod_\ell \bar{\M}(\nu_i, r)\times \M(v',w')$,
where $v'=(m_i)_{i\in Q_0}$ and $w'$ is given by $w'_i=w_i-(v^0,\alpha_i)$. We remark that $\dim \M(v',w')>0$.
Indeed, we have
\begin{align*}
\dim \M(v',w')&=2\sum_{i\in Q_0}(w_i-(v^0,\alpha_i))m_i-(\sum_i m_i\alpha_i, \sum_i m_i\alpha_i)=\\&=2w\cdot (v-v^0)-2(v^0,v-v^0)-(v-v^0,v-v^0)=\\
&=2w\cdot (v-v^0)-2(v,v-v^0)+(v-v^0,v-v^0).
\end{align*}
But $w\cdot (v-v^0)\geqslant (v,v-v^0)$ because $\nu$ is dominant. We also have $(v-v^0,v-v^0)\geqslant 0$
with the equality only if $v-v^0=k\delta$. But if the last equality holds, then we have $(v,v-v^0)=0$,
while $w\cdot \delta$ is always positive. The quiver defining  $\M_p(v',w')$ has no loops so that variety cannot have single point leaves.
So we have proved that $\underline{\A}$ cannot have a finite dimensional representation. This implies
that $\Leaf'\subset \VA(\M)$.

Now let us show that $\Leaf\subset \VA(\M)$. The slice algebra corresponding to $\Leaf'$ is $\underline{\A}'=
\bar{\A}_{\langle \lambda,\delta\rangle}(n,r)$. It follows that $\VA(\M_{\dagger,x'})$ contains the leaf
corresponding to the partition $\mu=(n'^{q})$. It follows that $\VA(\M)$ contains
$\Leaf(\mu)$.

\subsection{Affine wall-crossing functor and counting}\label{SS_aff_WC}
Using (i)-(iv) proved above one gets a direct analog of \cite[Theorem 7.2]{BL}. Let us introduce some notation.
Let $\theta,\theta'$ be two stability conditions separated by $\ker\delta$. Let $\lambda,\lambda'$ be two parameters
with $\lambda'-\lambda\in \Z^{Q_0}$, $\langle \lambda,\delta\rangle$ has denominator $n'\leqslant n$, and $(\lambda,\theta)$,
$(\lambda',\theta')$ satisfy the abelian localization. Set $q=\lfloor n/n'\rfloor$.

\begin{Thm}\label{Thm:perv}
There are chains of two-sided ideals $\{0\}\subsetneq \J_{q}\subsetneq \J_{q-1}\subsetneq \ldots \subsetneq \J_1\subsetneq \A_\lambda(v,w)$
and $\{0\}\subsetneq \J'_q\subsetneq \J'_{q-1}\subsetneq\ldots \subsetneq \J'_1\subsetneq \A_{\lambda'}(v,w)$ with the following properties.   Let
$\mathcal{C}_i$  be the subcategory of all modules in $\A_{\lambda}(v,w)\operatorname{-mod}$ annihilated by $\J_{q+1-\lfloor i/(rm-1)\rfloor}$
(this is a Serre subcategory by a direct analog of (1) of \cite[Theorem 7.1]{BL}) and let $\mathcal{C}'_i\subset \A_{\lambda'}(v,w)\operatorname{-mod}$ be defined
analogously. Then the following is true:
\begin{enumerate}
\item $\WC_{\theta\rightarrow \theta'},\WC_{\theta'\rightarrow \theta}$ are perverse equivalences with respect
to these filtrations inducing mutually inverse bijections between simples.
\item For a simple $S\in \mathcal{C}_{j(rm-1)}\setminus\mathcal{C}_{j(rm-1)+1}$, the simple $S'$ is a quotient of $H_{j(rm-1)}(\WC_{\theta\rightarrow \theta'}S)$.
\item The bijection $S\mapsto S'$ preserves the associated varieties of the annihilators.
\end{enumerate}
\end{Thm}

Similarly to \cite[Section 8]{BL}, this theorem implies Theorem \ref{Thm:counting}.


\begin{thebibliography}{99}
\bibitem[BaGi]{BarGin} V. Baranovsky, V. Ginzburg. In preparation.
\bibitem[BGKP]{BGKP} V. Baranovsky, V. Ginzburg, D. Kaledin, J. Pecharich. {\it Quantization of line bundles on Lagrangian subvarieties}.
arXiv:1403.3493.
\bibitem[BEG]{BEG} Yu. Berest, P. Etingof, V. Ginzburg,
{\it Finite-dimensional representations of rational Cherednik
algebras}. Int. Math. Res. Not.  2003,  no. 19, 1053-1088.
\bibitem[BE]{BE} R. Bezrukavnikov, P. Etingof, {\it Parabolic induction and restriction
functors for rational Cherednik algebras}.  Selecta Math.,  14(2009), 397-425.
\bibitem[BK]{BK} R. Bezrukavnikov, D. Kaledin. {\it Fedosov quantization in the algebraic context}.
Moscow Math. J. 4 (2004), 559-592.
\bibitem[BL]{BL} R. Bezrukavnikov, I. Losev, {\it Etingof conjecture for quantized quiver varieites}. arXiv:1309.1716.
\bibitem[BEG1]{BEG1} Yu. Berest, P. Etingof, V. Ginzburg. {\it Finite dimensional representations
of rational Cherednik algebras}.  Int. Math. Res. Not. 19(2003),  1053-1088.
\bibitem[BoKr]{BoKr} W. Borho, H. Kraft. {\it \"{U}ber die
Gelfand-Kirillov-Dimension}. Math. Ann. 220(1976), 1-24.
\bibitem[BPW]{BPW} T. Braden, N. Proudfoot, B. Webster, {\it Quantizations of conical symplectic resolutions I: local and global structure}. arXiv:1208.3863.
\bibitem[BLPW]{BLPW} T. Braden, A. Licata, N. Proudfoot, B. Webster, {\it Quantizations of conical symplectic resolutions II:
category O and symplectic duality}. Unpublished manuscript.
\bibitem[CB]{CB_geom} W.~Crawley-Boevey,  {\it Geometry of the moment map for representations of quivers,} Comp. Math. {\bf 126} (2001), 257--293.
\bibitem[EG]{EG} P.~Etingof and V.~Ginzburg. {\it Symplectic reflection algebras, Calogero-Moser space,
and deformed Harish-Chandra homomorphism}, Invent. Math. {\bf 147} (2002), 243-348.
\bibitem[GG]{GG} W.L. Gan, V. Ginzburg, {\it Almost commuting variety, $\Dcal$-modules and Cherednik algebras}.
IMRP, 2006, doi: 10.1155/IMRP/2006/26439.
\bibitem[GGOR]{GGOR} V. Ginzburg, N. Guay, E. Opdam and R. Rouquier, {\it On the category $\mathcal{O}$ for rational
Cherednik algebras}, Invent. Math., {\bf 154} (2003), 617-651.
\bibitem[G]{Gordon} I. Gordon. {\it Quiver varieties, category O  for rational Cherednik algebras, and Hecke algebras}.
Int. Math. Res. Pap. IMRP  (2008),  no. 3, Art. ID rpn006.
\bibitem[GL]{GL} I. Gordon, I. Losev, {\it On category $\mathcal{O}$ for cyclotomic Rational
Cherednik algebras}. arXiv:1109.2315. To appear in J. Eur. Math. Soc. 
\bibitem[GS1]{GS1} I. Gordon, T. Stafford, {\it Rational Cherednik algebras and Hilbert schemes}, Adv. Math. 198 (2005), no. 1, 222-274.
\bibitem[GS2]{GS2} I. Gordon, T. Stafford, {\it  Rational Cherednik algebras and Hilbert schemes. II. Representations and sheaves},
Duke Math. J.  132  (2006),  no. 1, 73–135.
\bibitem[K]{Kaledin_sing} D. Kaledin, {\it Symplectic singularities from the Poisson point of view}. J. Reine Angew. Math.  600  (2006), 135–156.
\bibitem[KR]{KR} M.~Kashiwara and R.~Rouquier, {\it Microlocalization of rational Cherednik algebras}, Duke Math. J. {\bf 144} (2008) 525-573.
\bibitem[Ko]{Korb} D. Korb. {\it Order on the Fixed Points of the Gieseker Variety With Respect to the Torus Action}.
arXiv:1312.3025.
\bibitem[L1]{wquant} I.V. Losev. {\it Quantized symplectic actions and
$W$-algebras}.   J. Amer. Math. Soc. 23(2010), 35-59.
\bibitem[L2]{HC} I. Losev. {\it Finite dimensional representations of
W-algebras}.   Duke Math. J. 159(2011), n.1, 99-143.
\bibitem[L3]{sraco} I. Losev, {\it Completions of symplectic reflection algebras}. Selecta Math., 18(2012), N1, 179-251.
\bibitem[L4]{quant} I. Losev, {\it Isomorphisms of quantizations via quantization of resolutions}. Adv. Math. 231(2012), 1216-1270.
\bibitem[L5]{W-prim} I. Losev. {\it Primitive ideals in W-algebras of type A}. J. Algebra, 359 (2012), 80-88.
\bibitem[L6]{W_dim} I. Losev, {\it Dimensions of irreducible modules over W-algebras and Goldie ranks}.
arXiv:1209.1083.
\bibitem[MO]{MO} D. Maulik, A. Okounkov. {\it Quantum Groups and Quantum Cohomology}. arXiv:1211.1287.
\bibitem[MN1]{MN_der} K.~McGerty, T.~Nevins, {\it Derived equivalence for quantum symplectic resolutions,} arXiv:1108.6267.
\bibitem[MN2]{MN_ab} K.~McGerty, T.~Nevins, {\it Compatability of t-structures for quantum
symplectic resolutions}. arXiv:1312.7180.
\bibitem[Nak1]{Nakajima} H. Nakajima. {\it Instantons on ALE spaces, quiver varieties and Kac-Moody algebras}.
Duke Math. J. 76(1994), 365-416.
\bibitem[Nak2]{Nakajima_tensor} H. Nakajima. {\it Quiver varieties and tensor products}. Invent. Math.  146  (2001),  no. 2, 399-449.
\bibitem[NY]{NY} H. Nakajima, K. Yoshioka. {\it Lectures on Instanton Counting}.  Algebraic structures and moduli spaces,  31–101,
CRM Proc. Lecture Notes, 38, Amer. Math. Soc., Providence, RI, 2004.
\bibitem[Nam1]{Namikawa_fund} Y. Namikawa, {\it Fundamental groups of symplectic singularities}.
arXiv.1301.1008.
\bibitem[Nam2]{Namikawa} Y. Namikawa. {\it Poisson deformations and Mori dream spaces},
arXiv:1305.1698.
\bibitem[R]{rouqqsch} R. Rouquier, {\it $q$-Schur algebras for complex reflection groups}. Mosc. Math. J. 8 (2008), 119-158.
\bibitem[W]{Wilcox} S. Wilcox, {\it Supports of representations of the rational Cherednik algebra of type A}.
arXiv:1012.2585.
\end{thebibliography}
\end{document}